\documentclass[,reqno]{amsart}
\usepackage{comment}
\usepackage{amssymb}
\usepackage{amsmath}
\usepackage{latexsym,color}
\usepackage{a4wide}

\usepackage{amsmath}
\usepackage{amsthm}
\usepackage{amssymb}
\usepackage{cancel}
\usepackage{xcolor}
\usepackage{color}
\usepackage{calrsfs}
\usepackage[pdfborder={0 0 0}, colorlinks=true, pdfborderstyle={}, linkcolor=dblue, citecolor=dblue, urlcolor=blue]{hyperref}
\usepackage{txfonts,bbm}
\usepackage{pxfonts}
\usepackage{epsfig}
\usepackage{psfrag}
\usepackage{subfigure}
\usepackage{graphicx}
\usepackage[mathscr,mathcal]{eucal}
\usepackage{enumerate}
\usepackage{t1enc}
\usepackage[latin2]{inputenc}
\usepackage[normalem]{ulem}
\definecolor{dblue}{rgb}{0.09,0.32,0.44} 

\DeclareMathOperator{\essup}{essup}

\def \ud{\mathrm d}

\newtheorem{defi}{Definition}[section]

\newtheorem{lemm}[defi]{Lemma}
\newtheorem{thm}[defi]{Theorem}
\newtheorem{remark}[defi]{Remark}
\newtheorem{cor}[defi]{Corollary}

\newtheorem{prop}[defi]{Proposition}

\numberwithin{equation}{section}
\numberwithin{figure}{section}

\renewcommand{\rho}{\varrho}
\renewcommand{\phi}{\varphi}

\newcommand{\be}{\begin{equation}}

	 \let\P=\Pi

	\def\({\left(}
	\def\){\right)}
	\title{Rearranged Stochastic Heat Equation with an Entropy Gradient Structure}
\author[F. Delarue]{Francois Delarue}
\author[R. Likibi Pellat]{Rhoss Likibi Pellat}

\address{Universit\'e C\^ote d'Azur, CNRS, Laboratoire J.A. Dieudonn\'e, 06108 Nice, France
}\email{francois.delarue@univ-cotedazur.fr}
\email{rhoss.likibi-pellat@univ-cotedazur.fr}

    \thanks{F. Delarue
    and R. Likibi Pellat acknowledge the financial support of the European Research Council (ERC) under the European Union's Horizon Europe research and innovation program (ELISA project, Grant agreement No. 101054746). Views and opinions expressed are however
those of the author(s) only and do not necessarily reflect those of the European Union or the
European Research Council Executive Agency. Neither the European Union nor the granting
authority can be held responsible for them.}


\begin{document}
\begin{abstract}
{We extend a previously introduced one-dimensional diffusion model on the space of probability measures, defined via the rearranged stochastic heat equation by, penalizing the dynamics with an additional entropy-driven gradient-descent term. By means of a splitting argument, we prove that despite the opposite effects of rearrangement and entropy minimization, the resulting penalized stochastic heat equation is well defined. We study several properties of the associated dynamics and show, in particular, that solutions admit a density satisfying a corrected version of the Dean--Kawasaki equation. The smoothing properties established for the stochastic heat equation are shown to persist.}
\end{abstract}

\maketitle

\section{Introduction}
{\subsection{From the rearranged SHE to its entropic regulazised version.} The rearranged stochastic heat equation is a reflected version, introduced in \cite{DelHam25}, of the stochastic heat equation on the one-dimensional torus
$\mathbb S := \mathbb R / \mathbb Z$.
Its solution is a stochastic process $(X_t)_{t \ge 0}$ taking values in the space $\mathcal U^2(\mathbb S)$ of square-integrable quantile functions on $\mathbb S$. This space consists of real-valued measurable functions on $\mathbb S$ that are even and non-increasing on $[0,1/2)$; a more precise definition of $\mathcal U^2(\mathbb S)$ and related functional spaces is given in Subsection~\ref{Notations}. 
The equation reads
\begin{align}\label{0.1}
\mathrm d X_t(\mathbf{x})
=
\Delta_x X_t(\mathbf{x})\,\mathrm dt
+ \mathrm d W_t(\mathbf{x})
+ \mathrm d \eta_t(\mathbf{x}),
\qquad t \ge 0,\, \mathbf{x} \in \mathbb{S},
\end{align}
where $(W_t)_{t \ge 0}$ is a colored noise defined on a filtered probability space
$(\Omega,\mathfrak F,(\mathfrak F_t)_{t \ge 0},\mathbb P)$, and admits the expansion
\[
W_t(\mathbf{x})
=
\sum_{k \ge 0} \lambda_k B_t^k e_k(\mathbf{x}),
\qquad \mathbf{x} \in \mathbb S,\quad t \ge 0.
\]
The sequence $(\lambda_k)_{k \ge 0}$ is square-summable, $(e_k)_{k \ge 0}$ denotes the cosine Fourier basis of $L^2(\mathbb S)$ and $\{(B_t^k)_{t \ge 0}\}_{k \ge 0}$ is a family of independent one-dimensional $(\mathcal F_t)_{t \geq 0}$-Brownian motions (see \eqref{eq:noise:1.8}).
In \eqref{0.1}, 
the initial condition $X_0$ is an $\mathcal F_0$-measurable random variable with values in $\mathcal U^2(\mathbb S)$, which satisfies
\[
\mathbb E\bigl[\|X_0\|_{L^2(\mathbb S)}^{2p}\bigr] < \infty
\qquad \text{for all } p \ge 1.
\]
The distinctive feature of \eqref{0.1} is the presence of the term $(\eta_t)_{t \ge 0}$, 
which is a distribution-valued reflection process that enforces the constraint 
$X_t \in \mathcal U^2(\mathbb S)$ for all $t \ge 0$.

In \cite{DelHam25}, existence and uniqueness of solutions to \eqref{0.1} are established. Existence is proved using a splitting method, which alternates over infinitesimal time intervals between the free dynamics of the heat equation and a rearrangement operation. After each free evolution, the rearrangement maps the state variable --which may no longer be a quantile function-- onto the unique quantile function with the same distribution on the torus. The resulting dynamics can be viewed as a diffusion taking values in the space of probability measures, because the set $\mathcal U^2(\mathbb S)$, equipped with the $L^2$ distance on $\mathbb S$ (with respect to the Lebesgue measure), is isometric to the space 
${\mathcal P}_2({\mathbb R})$
of probability measures on $\mathbb R$ equipped with the 2-Wasserstein distance
${\mathscr W}_2$; see \eqref{wasserstein} for the definition of the latter.
The process $\eta$ in the dynamics \eqref{0.1} encodes, at the level of the infinitesimal increments of $X$, the action of the rearrangement.
It is also shown in \cite{DelHam25} that, thanks to the presence 
of the Laplace operator in 
\eqref{0.1},
the semi-group generated by $X$, which acts on real-valued functions defined on 
${\mathcal P}_2({\mathbb R})$, is strong Feller, as it maps bounded measurable functions onto ${\mathscr W}_2$-Lipschitz continuous functions.

Since the process $(X_t)_{t \ge 0}$ describes the random evolution of a probability measure, it can be viewed as a type of systemic noise commonly used in McKean-Vlasov theory; see, for example, \cite{CarmonaDelarueII}. In particular, it can serve as a systemic (or common) noise for a first-order McKean-Vlasov equation, that is, a deterministic differential equation initialized from a random initial condition and driven by coefficients depending on the law of the solution. For instance, a drifted version of \eqref{0.1} is studied in \cite{DelHamarxiv1}.  
However, the approach in \cite{DelHamarxiv1} breaks down if the diffusion in \eqref{0.1} is used in a genuinely diffusive McKean-Vlasov equation. In contrast with the former, the latter includes an additional noise term, typically called idiosyncratic, which acts pathwise on the solution; as a consequence, it affects the statistical state of the system only in a deterministic manner, the evolution of the law resulting solely from the individual action of the noise on each realization.
This observation constitutes the starting point of the present article. As a first step toward a general McKean--Vlasov-type diffusion driven by a noise of the form \eqref{0.1}, we aim here to provide a mathematical description of the action of \eqref{0.1} on a system of independent Brownian motions, or, equivalently, of a system subjected to both an idiosyncratic Brownian motion and a common noise of the type \eqref{0.1}. The issue is in fact more subtle than it may first appear. Indeed, simply adding a Brownian motion to \eqref{0.1} in order to account for the effect of an idiosyncratic noise would not address the problem at all, since the law of the random variable $X_t$ on which the noise $W$ acts is precisely the law of the mapping ${\bf x} \in \mathbb S \mapsto X_t({\bf x})$. To incorporate the statistical effect of an idiosyncratic Brownian motion, one would therefore need to construct such a motion on the torus, which is possible in theory but leads to very challenging observability issues.

The approach adopted in this article is therefore different. 
The main idea is therefore to encode the idiosyncratic Brownian component not by adding a pathwise Brownian motion to the quantile equation, but by inserting, at the level of the law 
of the mapping $x \in \mathbb S \mapsto X_t(x)$, the heat flow generated by such a Brownian motion.
 Let us emphasize that the action of this semigroup should be distinguished from that of the Laplacian in \eqref{0.1}, which acts on the quantile function itself. As we shall see below, the action of the heat flow on a probability measure and the action of the heat flow on the quantile function associated with the same measure are not only different, but in fact have opposite effects.}
 
{\subsection{From a new splitting scheme to existence and uniqueness.}
Our approach relies on a new splitting scheme, in which we alternate between two steps. The first one is a random step in the space of quantile functions, given by the dynamics \eqref{0.1} over a time interval of length~$h$. The second one is a deterministic step in the space of probability measures, obtained by convolving the law of the quantile function at the end of the previous step with a Gaussian density of variance~$h$.

The first challenge is to show that this scheme admits a weak limit as the step size $h$ tends to~$0$, in the sense that the law of the resulting process, viewed as a random variable taking values in a suitable functional space, converges as $h$ vanishes. More precisely, we prove that, as the time step tends to zero, the scheme converges to the limiting equation
\begin{align} \label{eq:new:RSHE} \mathrm{d}X_t(\mathbf{x}) &= -\frac{1}{2}D_{\mathbf{x}} \Big(\frac{1}{D_{\mathbf{x}} X_t(\mathbf{x})}\Big)\,\mathrm{d}t + \Delta_{\mathbf{x}} X_t(\mathbf{x})\,\mathrm{d}t + \mathrm{d}\eta_t(\mathbf{x}) + \mathrm{d}W_t(\mathbf{x}), \qquad \mathbf{x} \in \mathbb S, \end{align}
  see Proposition~\ref{Propo 5.13}. Here, $(\eta_t)_{t \ge 0}$ is a reflection process which, as in \eqref{0.1}, forces the solution to remain a quantile function. The novelty lies in the first term on the right-hand side. A simple computation, given in Remark~\ref{Remark 5.16}, shows that this term coincides with the gradient of the entropy when the latter is expressed in terms of the quantile function rather than the associated density. In other words, this term is generated by the Gaussian convolution step in the splitting scheme, which is thereby explaining the title of the article.

In Theorem~\ref{Main 1}, we prove existence and uniqueness of solutions to \eqref{eq:new:RSHE} in a suitable weak formulation. This result represents the first main contribution of the paper.

A key difficulty in the analysis is that the two deterministic mechanisms entering the splitting scheme do not act in the same direction. On the one hand, the Laplacian in the rearranged stochastic heat equation, together with the rearrangement operation itself, tends to damp oscillations of the quantile function. On the other hand, the heat flow generated by the idiosyncratic Brownian motion acts on the associated probability density, rather than on the quantile function. Thus, the scheme combines two competing regularizing mechanisms: one acts on the quantile function, while the other acts on the corresponding law.

This phenomenon can be made explicit through the smoothing properties of the two dynamics. If $X$ is a symmetric differentiable function on the torus, the Polya--Szego inequality states, in one of its more general forms, that for any increasing convex function $\Phi : \mathbb{R} \to \mathbb{R}$,
\begin{equation*} \int_{\mathbb S} \Phi\!\left( D_{\mathbf{x}} X(\mathbf{x}) \right)\,\mathrm{d}\mathbf{x} \;\ge\; \int_{\mathbb S} \Phi\!\left( D_{\mathbf{x}} X^*(\mathbf{x}) \right)\,\mathrm{d}\mathbf{x}, \end{equation*}
where $X^*$ denotes the rearrangement of $X$, which is differentiable almost everywhere. A similar effect is produced by the Laplacian in \eqref{0.1}, in the sense that, for any $t>0$,\begin{equation*} \int_{\mathbb S} \Phi\!\left( D_{\mathbf{x}} \big[ {\rm e}^{t \Delta} X \big](\mathbf{x}) \right)\,\mathrm{d}x \;\le\; \int_{\mathbb S} \Phi\!\left( D_{\mathbf{x}} X(\mathbf{x}) \right)\,\mathrm{d}\mathbf{x}. \end{equation*}
Thus, both deterministic components of \eqref{0.1}, excluding the noise, act to decrease quantities of the form
\[ \int_{\mathbb S} \Phi\!\left( D_{\mathbf{x}} X(x) \right)\,\mathrm{d}\mathbf{x}. \]

If $X$ is a quantile function and its law admits a density $p$ on $\mathbb{R}$, the above quantity can be rewritten as
\begin{equation*} \int_{\mathbb S} \Phi\!\left( D_{\mathbf{x}} X(\mathbf{x}) \right)\,\mathrm{d}\mathbf{x} = \int_{\mathbb R} \Phi\!\left( \frac{2}{p(z)} \right) p(z)\,\mathrm{d}z. \end{equation*}
By contrast, when heat flow acts on density $p$, it leads to the decay of
\[ \int_{\mathbb R} \Psi\!\left( p(z) \right)\,\mathrm{d}z, \]
whenever $\Psi$ is convex. Reconciling these two forms of dissipation would require identifying functions $\Phi$ and $\Psi$ such that
$\Psi(r)=\Phi(2/r)r$
but this is incompatible with the convexity requirements, except in trivial cases. Consequently, the splitting scheme does not seem to admit a simple Lyapunov functional. This absence of a common monotone structure reflects the underlying analytical difficulties of the problem.

{\subsection{Existence of a density, and related 
Dean-Kawasaki equation}
Another main contribution of this article is to show that, for almost every
$(t,\omega) \in [0,+\infty) \times \Omega$, the law of the mapping
$\mathbb S \ni \mathbf{x} \mapsto X_t(\mathbf{x})$ --with the dependence on $\omega$ left implicit-- has a density
$\mathbb R \ni z \mapsto p_t(z)$, which is square integrable. Here,
$(X_t)_{t \ge 0}$ denotes the solution to \eqref{eq:new:RSHE}, whose existence and uniqueness are guaranteed under the assumptions of Theorem~\ref{Prop5.1}.
Compared with the regularity properties known for \eqref{0.1}, this result constitutes a clear improvement and highlights the regularizing effect of the entropic correction. Indeed, the strongest result currently available for \eqref{0.1}, due to \cite{DelOukarxiv1}, only ensures that, for almost every $(t,\omega) \in [0,+\infty) \times \Omega$, the map $\mathbb S \ni \mathbf{x} \mapsto X_t(\mathbf{x})$ --where $(X_t)_{t \ge 0}$ now denotes the solution to \eqref{0.1}-- has no atoms.

The existence of a density, denoted by $(p_t)_{t \ge 0}$, naturally raises two further questions:
\begin{enumerate}
\item What can be said about the regularity of $(p_t)_{t \ge 0}$ with respect to the variable $y$?
\item Which Fokker--Planck equation is satisfied by $(p_t)_{t \ge 0}$?
\end{enumerate}

Concerning the first question, our results remain limited. We prove that, for almost every $(t,\omega)$, the density $p_t$ belongs to $L^2(\mathbb R)$ and that its inverse $p_t^{-1}$, restricted to the set $\{z \in \mathbb R : p_t(z) > 0\}$, is integrable. The difficulty in pushing the analysis further and in establishing local continuity properties stems once again from the structure of equation~\eqref{eq:new:RSHE}. While the entropic regularization tends to spread out the density, the rearrangement and the Laplacian inherited from \eqref{0.1} have the opposite effect.
In particular, one of the more surprising results of the article is that the density necessarily has compact support. This compactness results from a form of balance between the two competing dynamics underlying the splitting scheme that leads to equation~\eqref{eq:new:RSHE}. Despite numerous attempts to go further, it remains unclear to us whether the action of the heat semigroup on the density --during the entropic descent phase-- can yield additional regularity.

Concerning the equation satisfied by $(p_t)_{t \ge 0}$, we prove in
Theorem~\ref{thm:eq:density} by means of a variant of the It\^o formula established in 
\cite{DelHam25_2}
that it solves, in a weak sense, the following stochastic Fokker--Planck equation:
\begin{equation}
\label{eq:intro:DK}
\begin{split}
{\rm d}_t p_t(x)
&= \frac{1}{2}\Delta_x p_t(x)
 - 4\Delta_x\!\left( \mathbbm{1}_{\{p_t(x)>0\}}\frac{1}{p_t(x)}\right)
 + \frac{1}{2}\sum_{k \in \mathbb N}
 \Delta_x\!\left(\lambda_k^2\, e_k^2\!\left(\frac{\tilde F_t(x)}{2}\right) p_t(x)\right)  \\
&\quad
 - \frac{1}{2}\sum_{k \in \mathbb N}
 D_x\!\left(\lambda_k\, e_k\!\left(\frac{\tilde F_t(x)}{2}\right) p_t(x)\right)
 \,\mathrm{d}B_t^k ,
\end{split}
\end{equation}
where $\tilde F_t$ denotes the cumulative distribution function of the quantile function $X_t$ solution to \eqref{eq:new:RSHE}
This equation should be viewed as a variant of the Dean--Kawasaki equation, with the following modifications:
\begin{enumerate}
\item
The noise driving the equation is not white but coloured. This is not surprising: the original Dean--Kawasaki equation, driven by space--time white noise, is in fact ill posed (except trivial cases). The fact that the Laplacian in \eqref{0.1} is not sufficient to compensate for the singularities of white noise was already explained in \cite{DelHam25}, and the same difficulties arise here.

\item The noise is expanded along the family $(y \mapsto e_k(\tilde F_t(y)/2))_{k \in \mathbb N}$.
Such a composition of the Fourier basis with the cumulative distribution function does not appear in the classical formulation of the Dean--Kawasaki equation. In fact, the resulting family $(y \mapsto e_k(\tilde F_t(y)/2))_{k \in \mathbb N}$ forms an orthonormal basis of the space $L^2(\mathbb R, p_t)$.
From the geometric viewpoint of the space of probability measures, this reflects the fact that the noise inserted in 
\eqref{eq:intro:DK}
naturally acts on the tangent space to ${\mathcal P}_2(\mathbb R)$ at the point $p_t$.
Indeed, it should be noted that, at a formal level, the stochastic term in \eqref{0.1} and in \eqref{eq:new:RSHE} can be written (by composing the cumulative distribution function with the quantile function) as
\begin{equation*}
\sum_{k \in \mathbb N}
e_k(x)\,\mathrm{d}B_t^k
=
\sum_{k \in \mathbb N}
e_k\!\left(\tilde  F_t(\tfrac12 X_t(x) )\right)\,\mathrm{d}B_t^k
=
\sum_{k \in \mathbb N}
\bigl(e_k \circ \tilde F_t\bigr)\!\left(\tfrac12 X_t(x)\right)\,\mathrm{d}B_t^k .
\end{equation*}
In other words, the noise in \eqref{0.1} and 
\eqref{eq:new:RSHE} corresponds to a motion along the distributional vector field
\begin{equation*}
x \;\longmapsto\;
\sum_{k \in \mathbb N}
\lambda_k\, (e_k \circ \tilde F_t)(\tfrac12 x)\, \frac{\mathrm{d}B_t^k}{\mathrm{d}t},
\end{equation*}
obtained by anisotropic forcing along the orthonormal basis $(e_k \circ F_t(\tfrac12 \cdot))_{k \in \mathbb N}$ of 
$L^2({\mathbb R},p_t)$.

Notice that the third term on the right-hand side corresponds to a bracket term arising from the stochastic forcing. Since the stochastic term differs from that of the classical Dean--Kawasaki equation, this contribution also differs in form; nevertheless, it plays the same conceptual role.

\item The second term on the right-hand side of \eqref{eq:intro:DK}, namely the Laplacian of the inverse of $p_t$, does not appear in the standard Dean--Kawasaki equation. As already explained, this term is introduced through \eqref{0.1} in order to obtain additional regularization properties for the semigroup generated by $(X_t)_{t \ge 0}$ when acting on functions defined on ${\mathcal P}_2(\mathbb R)$. We explain below (see Remark \ref{rem:30-06:1}) that the regularization properties established in \cite{DelHam25} for \eqref{0.1} remain valid for \eqref{eq:new:RSHE}.
\end{enumerate}
}
\subsection{Connection with the existing literature}

Our work is part of a line of results initiated roughly twenty years ago by the article \cite{vRenesseSturm,Sturm} on Wasserstein diffusions. A fairly comprehensive survey was provided in the previous work \cite{DelHam25} on the rearranged stochastic heat equation; here we give a brief overview. The approach of \cite{vRenesseSturm} relies on Dirichlet forms. Building on this, the works \cite{Konarovsky0}, and later \cite{KonarovskyivRenesse, KonarovskyivRenesseTobias}, adopt a more trajectorial perspective; all of these are restricted to one dimension. 
In higher dimensions, \cite{DelloSchiavo0,DelloSchiavo24} developed an alternative approach, also based on Dirichlet forms, but readily interpretable from a trajectorial viewpoint; however, their results are limited to random walks on purely atomic measures. Other constructions in higher dimensions are proposed in \cite{Sturm26} and \cite{RenWang}. In particular, the construction of \cite{RenWang} is further exploited in \cite{RenRocknerWangWittman25} for applications to gradient flows on the Wasserstein space, a topic also discussed in \cite{DelHamarxiv1}.
The approach initiated by \cite{DelHamarxiv1} was also revisited in 
\cite{vRenesseWangWeiss}.

To the best of our knowledge, the combination of a Wasserstein diffusion with an idiosyncratic noise component (treated here in the form of a gradient descent term derived from the entropy) has not been studied previously. As explained above, this question is motivated by the extensive literature on mean-field systems with common noise, as surveyed in \cite{CarmonaDelarueII}, which itself traces back to the pioneering work of \cite{KurtzXiong},  subsequently complemented by those of \cite{gess2019,LackerShkolnikovZhang}. Similarly, questions regarding the regularity of the associated semigroup remain largely unexplored in the literature. However, in the case of the Fleming--Viot process, some regularity results have been obtained by \cite{Stannat}. 
Moreover, \cite{Marx2} also obtained regularization results for the semigroup generated by one-dimensional McKean--Vlasov equations driven by infinite-dimensional noise, though only in certain directions.

The study of the Dean--Kawasaki equation has also seen significant recent progress. In general, the equation is ill-posed, and only regularized or discretized models have been studied in depth; see, for instance, \cite{Cornalba_2020, FehrmanGess24, Perkowski,KonarovskyivRenesse3}. The results obtained for \eqref{eq:intro:DK} are complementary: in the absence of the Laplacian acting on $1/p$, the equation resembles (up to the change of basis discussed in the second item following \eqref{eq:intro:DK}) the one studied in \cite{Perkowski} in a multi-dimensional setting. As mentioned above, the additional Laplacian ensures rather general smoothing properties, which, to our knowledge, have not been established for other forms of the Dean--Kawasaki equation. Beyond these properties, the derivation of \eqref{eq:intro:DK} provides a clear link between the rearranged stochastic heat equation and the Dean--Kawasaki equation.

\subsection{Organization.} The paper is organized as follows. We provide a number of preliminary reminders in Section \ref{se:2}, in particular concerning rearrangement and the notion of solution introduced in \cite{DelHam25}. In Section \ref{se:3}, we introduce the time-discretized scheme that serves as the basis for the construction of solutions to the rearranged equation driven by the entropy gradient. In Section \ref{se:4}, we establish the tightness properties needed to pass to the limit as the time step of the scheme tends to $0$. The passage to the limit is studied in Section \ref{se:5}, and the properties of the limiting equation are proved in Section \ref{se:6}, including in particular the main statement, Theorem \ref{Main 1}. Section \ref{se:7} is devoted to the study of the density and the derivation of the regularized Dean--Kawasaki stochastic partial differential equation.


\section{Preliminaries}
\label{se:2}
This section is devoted to the introduction of notation, definitions, and preliminary  results that will be instrumental in the development of this work

\subsection{Notations}\label{Notations}

\paragraph{\it Standard spaces of regular functions.}
\begin{quote}
\setlength{\leftskip}{-2em} 
    For two metric spaces $\mathcal{X}$ and $\mathcal{Y}$, we denote by
$\mathcal{C}(\mathcal{X},\mathcal{Y})$ the space of continuous functions from
$\mathcal{X}$ to $\mathcal{Y}$. For $k \geq 1$, $\mathcal{C}_0^{\infty}(\mathbb{R}^k)$
denotes the space of real-valued infinitely differentiable functions on
$\mathbb{R}^k$ with compact support.
{Throughout, we use the bold letter ${\mathbf x}$
for generic elements of ${\mathbb S}$, and $y$ and $z$ for generic elements of ${\mathbb R}$. Derivatives with respect to ${\mathbf x}$, $y$ and $z$ are respectively denoted by 
$D_{\mathbf x}$, $D_y$ and $D_z$. Similarly, the Laplacian on ${\mathbb S}$
(which is also $D^2_{\mathbf x}$) is denoted by $\Delta_{\mathbf x}$. When clear from the context, we just write $D$ or $\Delta$ for $D_{\mathbf x}$ or $\Delta_{\mathbf x}$.}
\end{quote}

\paragraph{\it Elements of Fourier analysis.}
\begin{quote}
\setlength{\leftskip}{-2em} 
The Lebesgue measure on $\mathbb{S}$ is denoted by $\mathrm{Leb}_{\mathbb{S}}$,
and integration with respect to this measure is written as $\mathrm{d}x$.
For $p \geq 1$, $\|\cdot\|_p$ denotes the usual $L^p$-norm on $(\mathbb{S},\mathrm{Leb}_{\mathbb{S}})$, with $\int_{\mathbb{S}}|f(\mathbf{x})|^p\mathrm{d}\mathbf{x}< \infty$,
 while $\|\cdot\|_{\infty}$ denotes the essential supremum norm i.e. $\|f\|_{\infty}:= \essup\{f(\mathbf{x}): \mathbf{x}\in \mathbb{S}\}$.
The inner product in $L^2(\mathbb{S})$ is written as $\langle \cdot, \cdot \rangle$.
The collection of equivalence classes in $L^2(\mathbb{S})$ containing a non decreasing symmetric function 
{(i.e.,
a symmetric periodic function that is non-decreasing on $(0,1/2)+{\mathbb Z}$)}
will be denoted by $\mathcal{U}^2(\mathbb{S})$.

For $k \in \mathbb{N}\setminus \{0\}$, we define
\[
e_k(\mathbf{x}) := \sqrt{2}\cos(2k\pi \mathbf{x}), \qquad \mathbf{x} \in \mathbb{S},
\]
and, for $k=0$, we let $e_0 \equiv 1$.
The family $(e_k)_{k \in \mathbb{N}}$ forms a complete orthonormal basis of
the subspace $L^2_{\mathrm{Sym}}(\mathbb{S})$ of $L^2(\mathbb{S})$, consisting of square integrable functions on $\mathbb{S}$ that are
almost everywhere symmetric.
 For any function  $f \in L^1(\mathbb{S})$ and $k \in \mathbb{N}$, the $k\text{-th}$ (cosine) Fourier coefficient of $f$ is given by
 $$\hat{f}_k := \langle f, e_k \rangle. $$
For $\gamma \in \mathbb{R}$, we denote by $\mathcal{H}^{\gamma}_{\mathrm{Sym}}(\mathbb{S})$
the Sobolev space of symmetric functions or distributions such that
\[
\|f\|_{2,\gamma}^2 := \sum_{k \in \mathbb{N}} (k \vee 1)^{2\gamma} \vert \hat{f}_k \vert^2 < \infty,
\]
with associated inner product
\[
\langle f,g \rangle_{2,\gamma}
:= \sum_{k \in \mathbb{N}} (k \vee 1)^{2\gamma} \hat{f}_k \hat{g}_k.
\]
When  $\gamma=0$, we just write 
$\| \cdot \|_{2}$
for $\| \cdot \|_{2,0}$
and $\langle \cdot, \cdot \rangle$ for 
$\langle \cdot, \cdot \rangle_{2,0}$.

At last, the periodic heat semigroup on the circle $\mathbb{S}$, which is commonly denoted by $(e^{t\Delta/2})_{t\geq 0}$, has the following kernel at time $t>0$: 
\begin{align*}
	{\mathbb S} \ni {\mathbf x} \mapsto  \frac{1}{\sqrt{2\pi t}}\sum_{n\in \mathbb{Z}}\exp\left\{-\frac{({\mathbf x}-n)^2}{2t}\right\}. 
\end{align*}
We will also make use of the heat kernel on 
${\mathbb R}$, which we denote by 
\begin{align}
	\Gamma_t(x):= \frac{1}{\sqrt{2\pi t}} \exp\left\{-\frac{x^2}{2t}\right\},\quad t>0,\,\, x\in \mathbb{R}. 
    \label{eq:heat:semigroup:torus}
\end{align}

\end{quote}

\paragraph{\it Probabilistic set-up}
\begin{quote}
\setlength{\leftskip}{-2em} 
The tuple $(\Omega,\mathfrak{F},{\mathbb F}=\{\mathfrak{F}_t\}_{t\geq 0},\mathbb{P})$
denotes a generic probability space equipped with a filtration ${\mathbb F}$ satisfying the usual conditions. 
The space is always equipped with 
a collection of 
independent 
${\mathbb F}$-Brownian motions $\{(B_t^k)_{t\geq 0}\}_{k\in \mathbb{N}}$.
For a real $\lambda >1/2$, we then let 
\begin{align}
\label{eq:noise:1.8}
			W_t:= B_t^0e_0 + \sum_{k\in \mathbb{N}}k^{-\lambda} B_t^k e_k \equiv \sum_{k\in \mathbb{N}} \lambda_k B_t^k e_k, \quad t\geq 0.
		\end{align}
		The process $(W_t)_{t\geq 0}$ is an $L^2_{\mathrm{Sym}}(\mathbb{S})\text{-valued}$ Brownian motion with covariance function 
		\begin{align}
			Q: (f,g) \in \left(L^2_{\mathrm{Sym}}(\mathbb{S})\right)^2\mapsto s\wedge t \sum_{k\in \mathbb{N}} \lambda_k^2\hat{f}^k\hat{g}^k = s\wedge t\,\, \langle f,g\rangle_{2,-\lambda},
		\end{align}
where  $(\lambda_k)_{k\in \mathbb{N}}$ is given by $\lambda_0 =1$ and $\lambda_k:= k^{-\lambda}$ for $k \in \mathbb{N}^{*}$.

The tuple $(\Omega,\mathfrak{F},{\mathbb F}=\{\mathfrak{F}_t\}_{t\geq 0},\mathbb{P},(W_t)_{t \geq 0})$
is called a probabilistic set-up. 
{Throughout, we use the generic notation ${\mathscr L}$ to denote the law of a random variable taking values in a Polish space.}
\end{quote}

\subsection{Rearrangement and quantile functions}
\label{subse:rearrangement}
All the material in this subsection is taken from \cite{DelHam25}.
 We recall that, for any measurable function $f:\mathbb{S} \rightarrow \mathbb{R}$, there is a unique function, called \text{\it symmetric non decreasing rearrangement of} $f$\footnote{Here, we make a slight modification compared with \cite{DelHam25}, where the authors instead considered functions that are non-increasing on $[0,1/2]$ and non-decreasing on $[-1/2,0]$. This change is made solely for notational convenience and has no impact on the results.} denoted by $f^{*}: \mathbb{S} \rightarrow [-\infty,\infty]$ satisfying:
\begin{itemize}
	\item[(1)] $f^{*}$ is symmetric with respect to $0$, non decreasing and {left-continuous on the interval $(0,1/2]$ and right-continuous at $0$}; in particular, it belongs to ${\mathcal U}^2({\mathbb S})$.
	\item[(2)] $f^*$ and $f$ have the same distribution under $\text{Leb}_{\mathbb S}$, which is known as Cavalieri's principle, i.e., 
	\begin{align*}
		\text{Leb}_{\mathbb{S}}\left(\{{\mathbf x} \in \mathbb{S}: f^{*}({\mathbf x}) \leq a\}\right) = \text{Leb}_{\mathbb{S}}\left(\{{\mathbf x}\in \mathbb{S}: f({\mathbf x}) \leq a\}\right)
	\end{align*}
\end{itemize}	
The collection of functions $f^{*}$ satisfying the point $(1)$ above is one to one with the set $\mathcal{P}_2(\mathbb{R})$ of probability measures 
on ${\mathbb R}$ that have a finite second moment. The bijection 
is defined by 
$\mathcal{U}^2(\mathbb{S}) \ni f^{*} \mapsto \text{Leb}_{
			\mathbb{S}}\circ(f^{*})^{-1} \in \mathcal{P}_2(\mathbb{R}) $. It is an isometry when $\mathcal{P}_2(\mathbb{R})$ is equipped with the $\mathscr{W}_2\text{-Wasserstein}$ distance, i.e., for any $f^{*}, g^{*} \in \mathcal{U}^2(\mathbb{S}),$
		\begin{align}\label{eq1.7}
			\|f^{*}-g^{*}\|_2 = \mathscr{W}_2\left({\rm Leb}_{
				\mathbb{S}}\circ (f^{*})^{-1}, {\rm Leb}_{
				\mathbb{S}}\circ (g^{*})^{-1}\right),
		\end{align}
		where for any measures $\mu,\nu \in \mathcal{P}_2(\mathbb{R})$ we define the $2\text{-Wasserstein}$ distance by
        \begin{equation}\label{wasserstein}
            \mathscr{W}_2^2(\mu,\nu):= \inf_{\pi \in \mathcal{P}(\mathbb{R}^2):\pi\circ e_x^{-1}=\mu,\pi\circ e_y^{-1}=\nu} \int_{\mathbb{R}^2}|x-y|^2\pi(\mathrm{d}x,\mathrm{d}y),
        \end{equation}
  $e_x:(x,y) \in \mathbb{R}^2\mapsto x$ and  $e_y:(x,y) \in \mathbb{R}^2\mapsto y$ being two evaluation mappings on $\mathbb{R}^2$. 

When $\mu \in {\mathcal P}({\mathbb R})$ is given, the pre-image $f^*$ of $\mu$ by the canonical isometry plays the role of {(periodic and symmetric)} quantile function. We denote  by
$F_{\mu}^{-1}$. It is given by the following 1-periodic function:
\begin{equation}\label{inverse}
	{F}_{\mu}^{-1}(\mathbf{x}) = \tilde{F}_{\mu}^{-1}(2\mathbf{x}) \mathbbm{1}_{[0,1/2]+{\mathbb Z}} (\mathbf{x})+ \tilde{F}_{\mu}^{-1}(-2\mathbf{x})\mathbbm{1}_{(-1/2,0)+{\mathbb Z}}(\mathbf{x}), \quad \mathbf{x} \in \mathbb{S},
\end{equation}
where, $\tilde F_{\mu}^{-1}$ is the standard quantile function of $\mu$, i.e., 
    \begin{align}
    \label{fonction:quantile:sur:R}
	\tilde F_{\mu}^{-1}(y):=\inf\{t \in \mathbb{R}: \tilde F_{\mu}(t) \geq y\}, \,\, \text{ for any } y\in [0,1].
\end{align}

The following result can be found in \cite[Proposition 3]{DelHam25}:
\begin{prop}\label{prop1.4}
The spaces $L^2_{\mathrm{Sym}}(\mathbb{S})$ and $\mathcal{U}^2(\mathbb{S})$ are closed subsets of $L^2(\mathbb{S})$ equipped with the norm $\|\cdot\|_2$. Additionally, for any $\alpha >0$, $\{f \in L_{\mathrm{Sym}}^2(\mathbb{S}): \|f\|_2 \leq \alpha\}$ and $\{f \in \mathcal{U}^2(\mathbb{S}): \|f\|_2 \leq \alpha\}$ are closed subsets of $\mathcal{H}^{-1}_{\mathrm{Sym}}(\mathbb{S})$ equipped with the norm $\|\cdot\|_{2,-1}$. Moreover, for any $\beta>\alpha$, any bounded sets in $\mathcal{H}_{\mathrm{Sym}}^{-\beta}(\mathbb{S})$ is relatively compact in $\mathcal{H}_{\mathrm{Sym}}^{-\alpha}(\mathbb{S}).$
\end{prop}





{We will also make use of the following standard result from Fourier analysis: a function $f \in {\mathcal H}^1_{\rm Sym}({\mathbb S})$ is (up to the choice of an almost everywhere representative) absolutely continuous with respect to the Lebesgue measure on ${\mathbb S}$. Its Radon-Nykodym derivative is 
(almost everywhere) equal to its Sobolev derivative. In particular, when $f \in  {\mathcal H}^1_{\rm Sym}({\mathbb S}) \cap {\mathcal U}^2({\mathbb S})$, the Sobolev derivative $D_{\mathbf x}f$ is almost everywhere non-negative on $(0,1/2)+{\mathbb Z}$, and almost everywhere non-positive on $(-1/2,0)+{\mathbb Z}$.}

{As another important property of quantile functions on the torus, we mention that,}
for any function $f\in {\mathcal U}^2(\mathbb{S})$ and any $t>0$, the {function ${\rm e}^{t \Delta} f$} belongs to {${\mathcal U}^2(\mathbb{S})$}, i.e. remains a quantile function (see, for instance, \cite[Lemma 5 and Lemma 6]{DelHam25}).

\subsection{Definition and properties of solution to RSHE}
We recall
from \cite{DelHam25}
\eqref{0.1}
the definition of  
a solution to 
\eqref{0.1}:
\begin{defi}\label{Defi 0.1}
		On a given 
        probabilistic set-up $(\Omega,\mathfrak{F},{\mathbb F}=\{\mathfrak{F}_t\}_{t\geq 0},\mathbb{P},(W_t)_{t \geq 0})$, and for
        an $\mathfrak{F}_0\text{-measurable}$ initial condition $X_0$ with values in $\mathcal{U}^2(\mathbb{S})$, satisfying 
        ${\mathbb E}[\| X_0\|_2^p]<+\infty$ for all $p \geq 1$, we say that a pair of processes $(X_t,\eta_t)_{t\geq 0}$ solves the rearranged SHE \eqref{0.1}, initialized at $X_0$, if
		\begin{itemize}
			\item[1.] $(X_t)_{t\geq 0}$ is a continuous $\mathbb{F}\text{-adapted}$ process with values in $\mathcal{U}^2(\mathbb{S})$;
			\item [2.]  $(\eta_t)_{t\geq 0}$ is a continuous $\mathbb{F}\text{-adapted}$ process with values in $\mathcal{H}^{-2}_{\mathrm{Sym}}(\mathbb{S})$, starting from $0$ at $0$, such that, with probability one, for any $\phi \in \mathcal{H}^{2}_{\mathrm{Sym}}(\mathbb{S})\cap \mathcal{U}^2(\mathbb{S})$, the path $(\langle \eta_t,\phi\rangle)_{t\geq 0}$ is non-decreasing;
			\item[3.] with probability one, for any $\phi \in \mathcal{H}^{2}_{\mathrm{Sym}}(\mathbb{S})$, for all $t\geq 0$,
			\begin{equation}
				\langle X_t,\phi\rangle = \langle X_0,\phi\rangle + \int_0^t \langle X_s,\Delta \varphi\rangle \mathrm{d}s + \langle W_t,\phi\rangle + \langle \eta_t,\phi \rangle.
			\end{equation}
			\item[4.] for any $t\geq 0$,
			\begin{align}\label{refl1}
				\lim_{\epsilon\searrow 0 }\mathbb{E} \int_0^t {\rm e}^{\epsilon \Delta} X_s\cdot \mathrm{d} \eta_s = 0.
			\end{align}
		\end{itemize}
		\end{defi}
Existence and uniqueness of a solution to the rearranged SHE in the sense of Definition \ref{Defi 0.1} are ensured by 
\cite[Theorem 1]{DelHam25}.

The following It\^o's formula is in order (see, for instance,\cite[Proposition 2.4]{DelHamarxiv1})
\begin{lemm}\label{Ito} Let $X_t$ be the solution to \eqref{0.1} in the sense of Definition \ref{Defi 0.1}. Then, we have
\begin{equation*}
	\mathrm{d}_t\|X_{t}\|_2^{2}= - 2 \|D_{\mathbf{x}}X_{t}\|_2^2 \mathrm{d}t +
	 c_{\lambda} \mathrm{d}t + 2\langle X_t,\mathrm{d}W_t\rangle
\end{equation*} 
where $c_{\lambda}:=\sum_{k \in \mathbb{N}} \lambda_k^2.$
\end{lemm}
\section{Splitting scheme alternating rearranged SHE and Gaussian convolution}
\label{se:3}

\subsection{Definition of the scheme}
For a time horizon $T>0$ and an integer $N \in {\mathbb N}^*$,
we consider 
a time mesh $(t_n)_{n=0,\cdots,N}$ of the interval $[0,T]$ with uniform step size $h:=T/N \in (0,1)$. 
Iterating with respect to the values of $n=0,1,\cdots, N$,  we 
define the process $(X_t^{(h)})_{0 \le t \le T}$ as the solution of the following sequence of rearranged SHEs,   each posed on an interval of length $h$:
\begin{align}\label{scheme1}
	\begin{cases}
		\mathrm{d}X_t^{(h)}(\mathbf{x}) =\displaystyle \Delta_{\mathbf{x}} X_t^{(h)}(\mathbf{x}) \mathrm{d}t +\mathrm{d}W_t(\mathbf{x}) + \mathrm{d}
        \left(         
        \eta_t^{(h)}
        -
{\eta_{t_n^-}^{(h)}}
        \right)        
        (\mathbf{x}),\quad \forall t\in [t_n,t_{n+1}) \\
		X_{t_{n}}^{(h)}(\mathbf{x}) = [\displaystyle\Phi^{(h)}(X_{t_{n}^{-}}^{(h)})](\mathbf{x})  ,\quad  \mathbf{x}\in \mathbb{S},
        \end{cases}
\end{align}
with the convention that 
$\eta_{0-}^{(h)} = 0$, 
and with the following comments:
\begin{enumerate}
\item[i.] In the initial condition
stated on the second line, 
$X_{t_{n}^{-}}^{(h)}$ is defined as
the left-limit
$$X_{t_{n}^{-}}^{(h)}:=\lim_{t \uparrow t_n} X_t^{(h)}.$$
\item [ii.] On the same line, 
$\Phi^{(h)}$ is defined as  the function
\begin{align}\label{eq:Phi}
	\Phi^{(h)}:\mathcal{U}^2({\mathbb S}) \ni \bigl(\mathbf{x} \in {\mathbb S} \mapsto X(\mathbf{x})\bigr) \mapsto \bigl(\mathbf{x} \in {\mathbb S} \mapsto F^{-1}_{{\rm Leb}_{\mathbb S} \circ (X)^{-1} \star \Gamma_h}(\mathbf{x})\bigr),
\end{align}
where, here and afterwards, the symbol $\star$ represents the convolution product on $\mathbb{R}$ and $F^{-1}_{{\rm Leb}_{\mathbb S} \circ (X)^{-1} \star \Gamma_h}$ is the quantile function associated to the probability measure ${{\rm Leb}_{\mathbb S} \circ (X)^{-1} \star \Gamma_h} \in \mathcal{P}_2(\mathbb{R})$.
\item[iii.] On the first line, 
the term 
$\eta_{t_n^-}^{(h)}$  
in the subtraction 
$\eta_t^{(h)} - \eta_{t_n^-}^{(h)}$ 
is a way to cumulate, in a continuous manner, the variations of 
$(\eta_t^{(h)})_{t \in [0,T]}$ over each of the successive intervals. Indeed, recall from point 2 in Definition~\ref{Defi 0.1} that
$\eta_{t_n}^{(h)} - \eta_{t_n^-}^{(h)} = \lim_{t \searrow t_n} [\eta_t^{(h)} - \eta_{t_n^-}^{(h)} ] = 0$.
The process $(\eta_t^{(h)})_{t \in [0,T]}$
constructed this way is continuous (with values in $\mathcal{H}^{-2}_{\mathrm{Sym}}(\mathbb{S})$).
\item[iv.]
 The construction scheme can be iterated 
along the lines of \eqref{scheme1},
because \cite[Proposition 4 \& Theorem 1]{DelHam25}
guarantee that, at the end of each time step, ${\mathbb E}[\| X_{t_n^-}^{(h)}\|_2^p] < + \infty$ for all
$p \geq 1$, and then
${\mathbb E}[\| X_{t_n} \|_2^p] < + \infty$  at the beginning of each time step.
\end{enumerate}
 \color{black}

As described above, the value of the scheme $X_{t}^{(h)}$, at time $t_n$, should be understood as the quantile function of the convolution product between the law induced by the solution of the rearranged SHE on the previous step and  the Gaussian kernel $\Gamma_h$.


Our purpose is to study the  asymptotic behavior of 
$(X_t^{(h)})_{0 \leq t \le T}$ as $h$ tends to $0$. In order to do so, it is convenient to reformulate 
the sequence of equations 
\eqref{scheme1} in the form of a single equation with jumps. 
Indeed, by setting
\begin{align}
\label{eq:def:zeta:t:h}
	\zeta_{t}^{(h)} := \sum_{j=1}^{N_t} \left(\Phi^{(h)}(X_{t_j^-}^{(h)}) -X_{t_j^-}^{(h)}
    \right) + 
    \eta_{t}^{(h)}=: {\mathscr J}_t^{(h)}(X^{(h)}) +  \eta^{(h)}_{t},
\end{align}
with $N_t= \lfloor t/h\rfloor$,  the scheme \eqref{scheme1} can be rewritten as
\begin{align}\label{scheme0}
\langle X_{t}^{(h)} -X_{0}^{(h)},\phi \rangle = \langle \zeta_{t}^{(h)},\phi\rangle  + \int_{0}^{t} \langle X_s^{(h)},\Delta_{\mathbf{x}} \phi \rangle \mathrm{d}s + \langle W_{t},\phi \rangle, \quad t \in [0,T],
\end{align} 
for any symmetric and smooth function $\phi$. 

\subsection{Properties of the quantile of the convolution }
The process $(\zeta_t^{(h)})_{0 \le t \le T}$ plays a key role in our analysis. 
{It is a jump process.
By continuity of 
$(\eta^{(h)}_t)_{0 \le t \le T}$ in
$\mathcal{H}^{-2}_{\mathrm{Sym}}(\mathbb{S})$, the jumps of 
$(\zeta_t^{(h)})_{0 \le t \le T}$ occur at nodes $t_n$, $n =0,\cdots,N$, of the mesh.} At time $t_n$ the jump (in $\mathcal{H}^{-2}_{\mathrm{Sym}}(\mathbb{S})$) is given by 
\begin{equation}\label{eq:jump}
\zeta_{t_n}^{(h)} - 
\zeta_{t_n^-}^{(h)} 
= \Phi^{(h)}(X_{t_n^-}^{(h)}) - X_{t_n^-}^{(h)}. 
\end{equation}
To better understand the structure of the jumps, we first collect a series of results on the mapping $\Phi^{(h)}$. We first provide an
identity connecting the $L^2(\mathbb{S})\text{-norms}$ of $\Phi^{(h)}(X)$ and  $X$,
for a given element $X \in L^2({\mathbb S})$. 
We also derive further bounds, which turn to be very useful in the sequel. 

\begin{lemm} For fixed $h>0$, consider the function $\Phi^{(h)}$ as given by \eqref{eq:Phi}. Then, for any $X \in L^2({\mathbb S})$,  we have
    \begin{align}\label{eq1.16}
	\|\Phi^{(h)}(X)\|_2^2 = \|X\|_2^2  +h.
\end{align}Moreover,  
\begin{align}
\| \Phi^{(h)}(X) - X \|_2^2 \leq h,\label{eq:1.16:bis} \\
2 \langle X,\Phi^{(h)}(X)-X \rangle \leq h. \label{eq:2.7}
\end{align}
\end{lemm}
\begin{proof}
    From the definition of the scalar product in $L^2(\mathbb{S})$, we observe that 
\begin{align*}
\|\Phi^{(h)}(X)\|_2^2 &= \int_{\mathbb{S}} |\Phi^{(h)}(X)|^2(\mathbf{x})\mathrm{d}\mathbf{x}
= \int_{\mathbb{R}} |\mathbf{x}|^2\mathrm{d}\left({\rm Leb}_{\mathbb{S}}\circ(X)^{-1}\star\Gamma_h\right)(\mathbf{x})\\
&= \int_{\mathbb{R}} |\mathbf{x}|^2\mathrm{d}\left({\rm Leb}_{\mathbb{S}}\circ(X)^{-1}\right)(\mathbf{x}) + \int_{\mathbb{R}} |\mathbf{x}|^2\Gamma_h(\mathbf{x})\mathrm{d}\mathbf{x} ,
\end{align*}
where the last equality 
follows from the interpretation of the convolution as the law of the sum of two independent  random variables, one is equal to $X$ and another is centered Gaussian. This leads to \eqref{eq1.16}. 

The 
bound \eqref{eq:1.16:bis}
follows from \eqref{eq1.7}: 
\begin{align*}
    \| \Phi^{(h)}(X) - X\|_2^2 &= \mathscr{W}_2^2\left( {\rm Leb}_{\mathbb{S}}\circ (X)^{-1}\star \Gamma_h, {\rm Leb}_{\mathbb{S}}\circ (X)^{-1}  \right)\\
    &\leq \int_{\mathbb{S}\times \mathbb{R}}|X(\mathbf{x}) + z - X(\mathbf{x})|^2\Gamma_h(z)\mathrm{d}\mathbf{x}\mathrm{d}z = h.
\end{align*}

As for the bound \eqref{eq:2.7}, it  follows by expanding 
$\Phi^{(h)}(X)$ in the form 
$ X+\Phi^{(h)}(X)- X$ in the $L^2(\mathbb{S})\text{-norm}$, together with the identity \eqref{eq1.16} and the fact that $\|\Phi^{(h)}(X)- X\|_2^2 \geq 0$.
\end{proof}

The second statement in this subsection
addresses the difference between the quantile functions of $\mu \star \Gamma_h$ and $\mu$, for a given probability measure
$\mu \in {\mathcal P}_2({\mathbb R})$. 
In particular, when $\mu = \mathrm{Leb}_{\mathbb S} \circ X^{-1}$,  for some $X \in \mathcal{U}^2({\mathbb S})$, the quantile function
$F_\mu^{-1}$ coincides with $X$, and 
$F_{\mu \star \Gamma_h}^{-1}$ with 
$\Phi^{(h)}(X)$; hence, the difference 
$F_{\mu \star \Gamma_h}^{-1}- F_\mu^{-1}$ coincides 
exactly with the \textit{jump} 
$\Phi^{(h)}(X) - X$. 
\begin{lemm}\label{lemm1.11}
Let $\mu \in \mathcal{P}_2(\mathbb{R})$ and $F_{\mu}^{-1}$ be its representative in $\mathcal{U}^2(\mathbb{S})$ as defined in \eqref{inverse}. Then, the following holds for    any fixed {$h\in T/{\mathbb N}^*$}:
	\begin{itemize}
	\item[i.] The difference ${F}^{-1}_{\mu*\Gamma_h}-{F}^{-1}_{\mu}$ is monotone: \text{ if} $\phi \in \mathcal{U}^2(\mathbb{S})$, then $\langle \phi,{F}^{-1}_{\mu*\Gamma_h}-{F}^{-1}_{\mu}\rangle \geq 0$;
	\item[ii.] In addition, if $\phi$ is also in $\mathcal{C}^1(\mathbb{S})$, then the following identity holds:
	\begin{align*}
		\left\langle \phi,F^{-1}_{\mu\star\Gamma_h}-F^{-1}_{\mu}\right\rangle :=\frac{1}{4}\int_{\mathbb{R}}\int_{0}^h \phi'\bigg(\frac{(\tilde{F}_{\mu}\star\Gamma_t)(z)}{2}\bigg)\big[D_z (\tilde{F}_{\mu}\star\Gamma_t)(z)\big]^2\mathrm{d}t\mathrm{d}z.
	\end{align*}
	\end{itemize}
\end{lemm}

\begin{proof} 
Throughout, the function 
$\varphi \in \mathcal{U}^2({\mathbb S})$ is fixed. By \eqref{eq1.7}, 
we obtain (the argument is very similar to 
the derivation of 
\eqref{eq1.16})
\begin{align*}
		\|\phi - F^{-1}_{\mu\star\Gamma_h}\|_2^2 = {\mathscr W}_2^2\left({\rm Leb}_{\mathbb{S}}\circ (\varphi)^{-1}, F_{\mu\star\Gamma_h} \right) 
        &\leq \int_{\mathbb{S}\times\mathbb{R}} \left|\phi(\mathbf{x})-z-F_{\mu}(\mathbf{x})\right|^2\Gamma_h(z)\mathrm{d}\mathbf{x}\mathrm{d}z 
        \\
        & = \color{black} \|\phi - F^{-1}_{\mu}\|_2^2 + h.
	\end{align*}
 By further developing both sides of the above inequality, we obtain
	\begin{align*}
		\|\phi\|_2^2 -2\langle \phi,F^{-1}_{\mu\star\Gamma_h}\rangle +\|F^{-1}_{\mu\star\Gamma_h}\|_2^2 \leq \|\phi\|_2^2 -2\langle \phi,F^{-1}_{\mu}\rangle +  \|F^{-1}_{\mu}\|_2^2 + h.
	\end{align*}
From the identity $\|F^{-1}_{\mu\star\Gamma_h}\|_2^2 = \|F^{-1}_{\mu}\|_2^2 +h$ (see \eqref{eq1.16}), we  derive the first point in the statement.

Let us now turn to the second point (assuming that $\varphi \in {\mathcal C}^1({\mathbb S})$). 
Recalling 
the notation 
\eqref{fonction:quantile:sur:R}, we first notice that, for any $y \in [0,1]$,
\begin{align*}
		(\tilde{F}^{-1}_{\mu\star\Gamma_h}-\tilde{F}^{-1}_{\mu})(y)&= \int_{\mathbb{R}}\mathrm{d}z\mathbbm{1}_{\{\tilde{F}^{-1}_{\mu}(y)\leq z < \tilde{F}^{-1}_{\mu\star\Gamma_h}(y) \}} - \int_{\mathbb{R}}\mathrm{d}z\mathbbm{1}_{\{\tilde{F}^{-1}_{\mu\star\Gamma_h}(y)\leq z < \tilde{F}^{-1}_{\mu}(y)  \}}.
        \end{align*}
        We also recall that, for any $z \in {\mathbb R}$, 
        $\tilde F_\mu^{-1}(y) \leq z \Leftrightarrow y \leq \tilde F_\mu(z)$, and (thus) 
        $\tilde F_\mu^{-1}(y) > z \Leftrightarrow y > \tilde F_\mu(z)$. Therefore, 
        \begin{align*}
		(\tilde{F}^{-1}_{\mu\star\Gamma_h}-\tilde{F}^{-1}_{\mu})(y)= \int_{\mathbb{R}}\mathrm{d}z\mathbbm{1}_{\{(\tilde{F}_{\mu}\star\Gamma_h)(z) < y\leq \tilde{F}_{\mu}(z) \}} - \int_{\mathbb{R}}\mathrm{d}z\mathbbm{1}_{\{\tilde{F}_{\mu}(z) < y\leq (\tilde{F}_{\mu}\star\Gamma_h)(z) \}},
	\end{align*}
where we used the fact, for any $z\in \mathbb{R}$, $\tilde{F}_{\mu\star \Gamma_h}(z)= \int_{\mathbb{R}}\tilde{F}_{\mu}(z-z')\Gamma_h(z')\mathrm{d}z'= (\tilde{F}_{\mu}\star \Gamma_h)(z)$.

Moreover, by returning to \eqref{inverse}, by using the fact that $\varphi$ is odd to get the second line, and then by performing a change of variable to derive the third line below, we have
	\begin{align*}
		\langle \phi,F^{-1}_{\mu\star\Gamma_h}-F^{-1}_{\mu}\rangle&=\int_{\mathbb{S}}\phi(\mathbf{x}) (F^{-1}_{\mu\star\Gamma_h}-F^{-1}_{\mu})(\mathbf{x})\mathrm{d}\mathbf{x}
        \\
		 &= \int_{-1/2}^{0}\phi(-x) (\tilde{F}^{-1}_{\mu\star\Gamma_h}-\tilde{F}^{-1}_{\mu})(-2x)\mathrm{d}x + \int_{0}^{1/2}\phi(x) (\tilde{F}^{-1}_{\mu\star\Gamma_h}-\tilde{F}^{-1}_{\mu})(2x)\mathrm{d}x \\ &=\int_{[0,1]}\phi(\frac{y}{2}) (\tilde{F}^{-1}_{\mu\star\Gamma_h}-\tilde{F}^{-1}_{\mu})(y)\mathrm{d}y.
	\end{align*}    
Combining the last two identities, and then applying Fubini's theorem, we obtain 
\begin{align*}
\langle \phi,F^{-1}_{\mu\star\Gamma_h}-F^{-1}_{\mu}\rangle &=\int_{\mathbb{R}}\mathrm{d}z\left(\int_{[0,1]}\phi(\frac{y}{2})\mathbbm{1}_{\{(\tilde{F}_{\mu}\star\Gamma_h)(z) <  y\leq \tilde{F}_{\mu}(z) \}}\mathrm{d}y - \int_{[0,1]}\phi(\frac{y}{2})\mathbbm{1}_{\{\tilde{F}_{\mu}(z) <  y\leq (\tilde{F}_{\mu}\star\Gamma_h)(z)\}}\mathrm{d}y \right)\\
		&= \int_{\mathbb{R}}\mathrm{d}z\Bigg( \int_{(\tilde{F}_{\mu}\star\Gamma_h)(z)}^{\tilde{F}_{\mu}(z)}\phi(\frac{y}{2})\mathrm{d}y\Bigg).
\end{align*}
Notice that Fubini's theorem can be applied indeed because
\begin{equation}
\label{eq:Fubini:justification}
\begin{split}
&\int_{\mathbb R}
\left\vert 
 \left( \tilde{F}_{\mu}\star\Gamma_h\right)(z)
 - \tilde{F}_{\mu}(z)  \right\vert {\mathrm d}z
 \\
 &\leq \int_{-\infty}^0 
\left( \left(\tilde{F}_{\mu}\star\Gamma_h\right)(z)+  \tilde{F}_{\mu}(z) \right){\mathrm d}z
  + \int_0^{+\infty}
\left( 
\left( 1 - 
 \tilde{F}_{\mu}\star\Gamma_h\right)(z)
 + \left( 1 - \tilde{F}_{\mu}\right) (z)  \right) {\mathrm d}z
 \\
 &= \int_{-\infty}^0 y^- {\mathrm d} \mu(y) + 
 \int_{-\infty}^0 y^- {\mathrm d} \left( \mu \star \Gamma_h \right) (y) + \int_0^{+\infty}
 y^+ {\mathrm d} \mu(y)
 +
 \int_0^{+\infty}
 y^+ {\mathrm d} \left( \mu \star \Gamma_h \right) (y) < + \infty.
\end{split}
\end{equation}
Now, by continuity 
    of $\varphi$, it holds, for any $t>0$
\begin{equation*}
\partial_t \Biggl[\int_0^{(\tilde{F}_{\mu}\star\Gamma_t)(z)}
\varphi (\frac{y}2) {\mathrm d}y \Biggr] = 
\phi\big(\frac{(\tilde{F}_{\mu}\star\Gamma_t)(z)}{2}\big)\partial_t (\tilde{F}_{\mu}\star\Gamma_t)(z),
\end{equation*}
and then, by integrating in time, we obtain, for any $\epsilon >0$ 
	\begin{align*}
		\int_{\mathbb{R}}\mathrm{d}z\int_{0}^{(\tilde{F}_{\mu}\star\Gamma_h)(z)}\phi(\frac{y}{2})\mathrm{d}y = \int_{\mathbb{R}}\mathrm{d}z\int_{0}^{(\tilde{F}_{\mu}\star\Gamma_{\epsilon})(z)}\phi(\frac{y}{2})\mathrm{d}y + \int_{\mathbb{R}}\mathrm{d}z\int_{\epsilon}^h \phi\big(\frac{(\tilde{F}_{\mu}\star\Gamma_t)(z)}{2}\big)\partial_t (\tilde{F}_{\mu}\star\Gamma_t)(z)\mathrm{d}t.
	\end{align*}
	In addition, we recall that the convolution $(\tilde{F}_{\mu}\star\Gamma_t)$  is solution to the heat equation, that is, $\partial_t (\tilde{F}_{\mu}\star\Gamma_t)(z) = \frac{1}{2} \Delta_z (\tilde{F}_{\mu}\star\Gamma_t)(z)$, for $t>0$ and $z \in {\mathbb R}$. Then, recalling that 
    $\varphi \in {\mathcal C}^1({\mathbb S})$, and using the integration by part formula and the latter identity, we arrive at
    \begin{equation}
	\label{eq:ipp:convolution:eps}
    \begin{split}
&\int_{\mathbb{R}}\mathrm{d}z\int_{(\tilde{F}_{\mu}\star\Gamma_{\epsilon})(z)}^{(\tilde{F}_{\mu}\star\Gamma_h)(z)}\phi(\frac{y}{2})\mathrm{d}y 
    =  -\frac{1}{4} \int_{\mathbb{R}}\mathrm{d}z\int_{\epsilon}^h \phi'\big(\frac{(\tilde{F}_{\mu}\star\Gamma_t)(z)}{2}\big)[{D_z} (\tilde{F}_{\mu}\star\Gamma_t)(z)]^2\mathrm{d}t.
	\end{split}
    \end{equation}
As
$(t,z) \mapsto (\tilde{F}_\mu \star \Gamma_t)(z)$
solves the heat equation,  
it satisfies, for $\varepsilon \in (0,h)$ and for any $A>0$,
\begin{equation*}
\begin{split}
&\int_{-A}^A 
\mathrm{d}z\int_{\varepsilon}
^h
[D_z (\tilde{F}_{\mu}\star\Gamma_t)(z)]^2\mathrm{d}t 
\\
&= \int_{\varepsilon}
^h \biggl[ 
 \int_{-A}^A 
[D_z (\tilde{F}_{\mu}\star\Gamma_t)(z)]^2  \mathrm{d}z \biggr] \mathrm{d}t
\\
&= \int_{\varepsilon}
^h \biggl[ D_z (\tilde{F}_{\mu}\star\Gamma_t)(z) 
(\tilde{F}_{\mu}\star\Gamma_t)(z)
\biggr]_{-A}^A \ud t 
 - 
\int_{\varepsilon}
^h \biggl[  \int_{-A}^A   \Delta_{z} (\tilde{F}_{\mu}\star\Gamma_t)(z) 
(\tilde{F}_{\mu}\star\Gamma_t)(z) \mathrm{d}z
\biggr] \ud t
\\
&= \frac12 \int_{\varepsilon}
^h \biggl[ D_z \Bigl[ (\tilde{F}_{\mu}\star\Gamma_t)^2 \Bigr](A)
- D_z \Bigl[ (\tilde{F}_{\mu}\star\Gamma_t)^2 \Bigr](-A) 
\biggr] \ud t 
-
\int_{\varepsilon}
^h    \int_{-A}^A   \partial_t \Bigl[ [ (\tilde{F}_{\mu}\star\Gamma_t)(z)
]^2 \Bigr]
\mathrm{d}z
\biggr] \ud t,
\end{split}
\end{equation*}
and then, 
\begin{equation*}
\begin{split}
 \int_{-A}^A 
\mathrm{d}z\int_{\varepsilon}
^h
[D_z (\tilde{F}_{\mu}\star\Gamma_t)(z)]^2\mathrm{d}t 
&\leq 
 \int_{-A}^A 
 \left\vert 
[(\tilde{F}_{\mu}\star\Gamma_\varepsilon)(z)]^2
-
[(\tilde{F}_{\mu}\star\Gamma_h)(z)]^2
\right\vert 
\mathrm{d}z 
\\
&\hspace{15pt} +
\left\vert \int_{\varepsilon}^h 
D_z \left[(\tilde{F}_{\mu}\star\Gamma_t)^2\right](A)
\mathrm{d}t
\right\vert 
+  \left\vert 
\int_{\varepsilon}^h 
D_z \left[ (\tilde{F}_{\mu}\star\Gamma_t)^2\right](-A)
\mathrm{d}t \right\vert. 
\end{split}
\end{equation*}
Since $\tilde F_{\mu}$ tends to $0$ in $-\infty$ and to $1$ in $+\infty$, we can easily deduce that the second line in the above display tends to $0$ as $A$ tends to $+\infty$. 
Following 
\eqref{eq:Fubini:justification}, we see that the right-hand side on the first line is bounded, uniformly in $A$ and in 
$\varepsilon \in (0,h)$.
\color{black}
We deduce that 
\begin{equation*}
\int_{\mathbb{R}}\mathrm{d}z\int_{0}^h [D_z (\tilde{F}_{\mu}\star\Gamma_t)(z)]^2\mathrm{d}t < + \infty.
\end{equation*}
Then, since $\varphi'$ is bounded, we can let $\epsilon\searrow 0$ in \eqref{eq:ipp:convolution:eps} to obtain 
	\begin{align*}
		\int_{\mathbb{R}}\mathrm{d}z\int_{(\tilde{F}_{\mu}\star\Gamma_h)(z)}^{\tilde{F}_{\mu}(z)}\phi(\frac{y}{2})\mathrm{d}y = \frac{1}{4}\int_{\mathbb{R}}\mathrm{d}z\int_{0}^h \phi'\big(\frac{(\tilde{F}_{\mu}\star\Gamma_t)(z)}{2}\big)\big[D_z (\tilde{F}_{\mu}\star\Gamma_t)(z)\big]^2\mathrm{d}t. 
	\end{align*}
	This concludes the proof.
\end{proof}

\begin{remark}
    As a consequence of 
{item \textrm{\rm i}} in the statement of Lemma \ref{lemm1.11}, 
we observe that, 
for every $\varphi \in
    \mathcal{H}^{2}_{\mathrm{Sym}}(\mathbb{S})
    \cap U^2(\mathbb{S})$
    and every index $j \in {\mathbb N}$, 
    \begin{equation}
    \label{ineq:2.5:0}
\left\langle \Phi^{(h)}(X_{t_j^-}^{(h)}) -X_{t_j^-}^{(h)},\varphi\right\rangle \geq 0.
\end{equation}
Recalling that 
the trajectories of 
$(\eta^{(h)})_{t \in [0,T]}$ are non-decreasing, 
we deduce that, for every $\varphi \in
    \mathcal{H}^{2}_{\mathrm{Sym}}(\mathbb{S})
    \cap U^2(\mathbb{S})$, 
        the path  $t\mapsto \langle \zeta^{(h)}_t, \varphi \rangle$ is 
    (almost-surely) non-decreasing, i.e., for any $0\leq s \leq t \leq T$,
\begin{equation}\label{ineq:2.5}
    \langle \zeta^{(h)}_t -\zeta^{(h)}_s, \varphi \rangle = \sum_{j=N_s+1}^{N_t}\left\langle \Phi^{(h)}(X_{t_j^-}^{(h)}) -X_{t_j^-}^{(h)},\varphi\right\rangle + \left\langle \eta^{(h)}_{t}
    - \eta^{(h)}_{s}, \varphi \right\rangle \geq 0 .
\end{equation}
Moreover, 
\begin{equation}
\label{ineq:2.5:b}
\begin{split} 
&0 \leq \sum_{j=N_s+1}^{N_t}
\left\langle \Phi^{(h)}(X_{t_j^-}^{(h)}) -X_{t_j^-}^{(h)},\varphi\right\rangle \leq \|\zeta^{(h)}_t -\zeta^{(h)}_s\|_{2,-2} \| \varphi\|_{2,2}, 
\\
&0 \leq \left\langle \eta^{(h)}_{t}
    - \eta^{(h)}_{s}, \varphi \right\rangle
\leq \|\zeta^{(h)}_t -\zeta^{(h)}_s\|_{2,-2} \| \varphi\|_{2,2}.
\end{split}
\end{equation}
\end{remark}

\section{Tightness of the laws of the schemes}
\label{se:4}
This section is devoted to establishing the tightness of the family of schemes under consideration. 
In contrast with the approach developed in \cite{DelHam25}, where tightness was analyzed in the functional space 
$\mathcal{C}([0,T],\mathcal{H}^{-1}_{\mathrm{Sym}}(\mathbb{S}))$, we shall instead work in the Skorokhod space 
$\mathbb{D}([0,T],\mathcal{H}^{-1}_{\mathrm{Sym}}(\mathbb{S}))$, consisting of all 
$\mathcal{H}^{-1}_{\mathrm{Sym}}(\mathbb{S})$-valued c\`adl\`ag functions on $[0,T]$.
This choice is dictated by the intrinsic jump structure of the 
dynamics
\eqref{scheme0}. 

For completeness, we recall some notions about the compactness and tightness criterion in the space of c\`adl\`ag functions (equipped with the $J1$ topology), which can be found in \cite{Billingsley,JoffeMetiv86}.
\subsection*{The Aldous criteria} Let $(\mathbb{A},\rho)$ be a complete separable metric space. Let $(\Omega,
\mathfrak{F},\mathbb{P})$ be a probability space that satisfies the usual conditions. The following is a consequence of 
\cite[p.34]{JoffeMetiv86} (or equivalently in  \cite[p.117]{Billingsley}) 

\begin{thm}
\label{thm:tightness:D}
Let $(X^n)_{n\in \mathbb{N}}$ be a sequence of $\mathfrak{F}\text{-adapted}$ and $\mathbb{A}\text{-valued}$ processes. The sequence $(\mathbb{P}^n)_{n\in \mathbb{N}}$ of laws of $(X^n)_{n\in \mathbb{N}}$ on $\mathbb{D}([0,T];\mathbb{A})$ is tight if and only if the following conditions are satisfied:
\begin{itemize}
    \item[{\bf[T1]}] For every $t \in [0,T]$, the laws of the random variable $(X^n_t)_{n\in \mathbb{N}}$ form a tight sequence of law in $\mathbb{A}$. 
    \item[{\bf [T2]}] {For all $\epsilon>0$ and $\delta >0$, there exists} $\gamma_0 > 0$: $$\sup_{n\in \mathbb{N}}\mathbb{P}\Bigl( \bigl\{\omega_{[0,T],\mathbb{A}}(X^n,\gamma_0) > {\delta} \bigr\}\Bigr) \leq \epsilon,$$ where, for ${\bf v} \in \mathbb{D}([0,T],\mathbb{A})$ and $\gamma >0$,\begin{align}\label{A1}
        \omega_{[0,T],\mathbb{A}}({\bf v},\gamma):= \inf_{\Pi_{\gamma}}\max_{ r_i \in \tilde{\pi} }\sup_{r_i\leq s<t<r_{i+1}} \rho({\bf v}(t),{\bf v}(s)),
    \end{align}
$\Pi_{\gamma}$ denoting the set of all increasing sequences {$\tilde{\pi}=\{0=r_0<r_1<\cdots<r_n =T\}$ such that $r_{i+1}-r_i \geq \gamma, \, i=0,1,\cdots,n-1.$}    
\end{itemize}
\end{thm}
\begin{defi} 
\label{def:aldous}
A sequence $(X^n)_{n\in\mathbb{N}}$ of $\mathbb{A}\text{-valued}$ random variables satisfies the {\bf Aldous criterion} if and only if the following holds
\begin{itemize}
    \item[{\bf [A]}]  
    {For all
    $\epsilon>0$ and $\delta >0$}, there exist $\theta_0 > 0$ and $n_0 \in \mathbb{N}$ such that if $\theta \leq \theta_0$ and $n\geq n_0$ and if $\tau$ is a discrete  $\mathfrak{F}\text{-stopping}$ time satisfying  $\tau \leq T$, one has 
    \begin{align*}
        \sup_{n\geq n_0}\sup_{0\leq \theta \leq \theta_0}\mathbb{P}\Big\{\rho\big( X^n_{(\tau+\theta)\wedge T},X^n_{\tau}\big)\geq {\delta} \Big\} \leq \epsilon.
    \end{align*}
\end{itemize}
\end{defi}
\begin{remark}
  \label{rem:aldous}
  From \cite[Theorem 2.2.2]{JoffeMetiv86}, condition {\bf [A]} implies condition {\bf [T2]}. 
\end{remark}

\subsection{%
  \texorpdfstring{%
    Tightness of the laws of
    $\{(X_t^{(h)})_{0\leq t\leq T}\}_{h\in T/\mathbb{N}^{*}}$%
  }{%
    Tightness of the laws of the approximating processes%
  }%
}
Here is the main result of this subsection:
\begin{prop}\label{Prop 3.1} The sequence $\{(X_{t}^{(h)})_{0 \le t \le T}\}_{{h \in T/{\mathbb N}^*}}$ satisfying \eqref{scheme0} induces a tight family of probability measures on $\mathbb{D}([0,T],\mathcal{H}^{-1}_{\mathrm{Sym}}(\mathbb{S}))$.
\end{prop}
\begin{proof} The proof relies on the verification of \textbf{[T1]} and \textbf{[T2]} in the statement of Theorem~\ref{thm:tightness:D}.
\vskip 4pt

\textit{First Step.}
We first prove that for all $t \in [0,T]$, the family $\{X_t^{(h)}\}_{{h \in T/{\mathbb N}^*}}$ forms a tight sequence of law on $\mathcal{H}^{-1}_{\mathrm{Sym}}(\mathbb{S})$. 
As the embedding from $L^2_{\mathrm{Sym}}({\mathbb S})$ 
into $\mathcal{H}^{-1}_{\mathrm{Sym}}(\mathbb{S})$ is compact, 
it suffices to show that
\begin{align}\label{bound 2.8}
		\sup_{h \in T/ {\mathbb N}^*}\mathbb{E}\Big[\sup_{t\in [0,T]}\|X_t^{(h)}\|_2^{2}\Big] \leq C_{T,\lambda}\left(1 +\mathbb{E}[ \|X_{0}\|_2^{2}]\right ).
\end{align}
From It\^o's formula (see Lemma \ref{Ito}), we have,
{for 
$n \in \{0,\cdots,N-1\}$
and 
$t \in [t_n,t_{n+1})$,}
  \begin{align}
	\mathrm{d}_t\|X_{t}^{(h)}\|_2^2= - 2 \|D_{\mathbf{x}}X_{t}^{(h)}\|_2^2 \mathrm{d}t
	+ \sum_{k \in \mathbb{N}} \lambda_k^2\mathrm{d}t +  2\langle X_t^{(h)},\mathrm{d}W_t\rangle.
    \label{eq:ito:square:Xth}
\end{align}
{Moreover, we know from \eqref{eq1.16} that, 
at time $t=t_{n+1}$, 
$\| X_{t_{n+1}}^{(h)} \|_2^2 =\| \Phi^{(h)}(X_{t_{n+1}^-}^{(h)})\|_2^2 = 
\| X_{t_{n+1}^-}^{(h)} \|_2^2 + h$, 
from which we deduce that
\begin{equation}
\label{eq:ito:square:Xth:45}
\begin{split}
\| X_{t_{n+1}}^{(h)}\|_2^2 &= 
\| X_{t_{n+1}^-}^{(h)}\|_2^2 
+h 
\\
&= \| X_{t_{n}}^{(h)}\|_2^2
- 2 \int_{t_n}^{t_{n+1}} \| D_{\mathbf{x}} X_s^{(h)}\|_2^2
{\mathrm d} s 
+ 
2 \int_{t_n}^{t_{n+1}}
\langle X_s^{(h)} , 
{\mathrm d} W_s 
\rangle 
+ \Bigg( 1 + \sum_{k \in {\mathbb N}} \lambda_k^2 \Bigg) h. 
\end{split}
\end{equation}
By summing the above display from $0$ to $N_t-1$, and then by integrating
\eqref{eq:ito:square:Xth}  from $t_{N_t}$ to $t$, we obtain, ${\mathbb P}$-a.s., for every $t \in [0,T]$,}
\begin{align}\label{bound 2.9}
	\|X_{t}^{(h)}\|_2^2 + 2\int_{0}^{t} \|D_{\mathbf{x}}X_{s}^{(h)}\|_2^2 \mathrm{d}s \leq \|X_0\|_2^2 +  2\int_{0}^{t}  \langle X_s^{(h)},\mathrm{d}W_s\rangle +
    \Bigg( 1 +\sum_{k \in \mathbb{N}} \lambda_k^2\Bigg) t.
\end{align}
Since $(X_t^{(h)})_{t\in[0,T]}$ belongs to $L^2_{\mathrm{Sym}}(\mathbb{S})$, the stochastic integral above vanishes under the expectation value. Then, we obtain the following bound
\begin{align}\label{bound 1.26}
		\sup_{t\in [0,T]}\mathbb{E}\Big[\|X_t^{(h)}\|_2^{2}\Big] + 2\int_{0}^{T} \mathbb{E}\|D_{\mathbf{x}}X_{s}^{(h)}\|_2^2 \mathrm{d}s &\leq C_{T,\lambda}\bigl(1 +\mathbb{E}\|X_{0}\|_2^{2}\bigr).
\end{align}
Taking now the supremum over $t \in [0,T]$ in \eqref{bound 2.9}, the expectation value, and then applying BDG's inequality leads to the desired bound \eqref{bound 2.8}.
This proves 
$\mathbf {[T1]}$ in the statement of Theorem \ref{thm:tightness:D}. 
\vskip 4pt

\textit{Second Step.}
As a second step, we will first prove the existence of a constant $C>0$ and an integrable positive-valued random variable $\Xi$, such that, for  any discrete $\mathfrak{F}\text{-stopping}$ time $\tau \leq T$ and any real $\theta >0$,
\begin{align*}
    \mathbb{E}\left[\|X^{(h)}_{(\tau +\theta) \wedge T}-X^{(h)}_{\tau}\|_{2,-1}^2/\mathfrak{F}_{\tau}\right] \leq C\theta (1+\Xi) +h.
\end{align*}
For this, we consider the following process $({Z}_{t}^{(h)}:= \mathrm{e}^{\Delta (t-\tau)}X^{(h)}_{\tau})_{\tau \leq t \leq T}$, which is an element of $\mathcal{U}^2(\mathbb{S})$ (see Subsection \ref{subse:rearrangement}).
We further notice 
from 
\eqref{scheme0}
that both processes $(X_t^{(h)})_{\tau \leq t \leq T}$ and $({Z}_t^{(h)})_{\tau \leq t \leq T}$ satisfy, respectively,
\begin{align}
	\begin{cases}
		\mathrm{d}X_t^{(h)}(\mathbf{x}) =\Delta_{\mathbf{x}} X_t^{(h)}(\mathbf{x}) \mathrm{d}t + \mathrm{d}W_t(\mathbf{x}) +\mathrm{d}\zeta_t^{(h)}(\mathbf{x})\\
		\mathrm{d}Z_t^{(h)}(\mathbf{x}) = \Delta_{\mathbf{x}} Z_t^{(h)}(\mathbf{x})\mathrm{d}t,
	\end{cases}
\end{align}
with the same initial condition 
$X_{\tau}^{(h)} = Z_{\tau}^{(h)}$.
Thus, 
the difference process $(X_t^{(h)}-Z_t^{(h)})_{\tau \leq t \leq T}$
satisfies, 
for any sufficiently regular test function $\varphi$ 
and 
any $t \in [\tau,T]$,
\begin{align}
	\langle X_t^{(h)}-Z_t^{(h)},\varphi \rangle = \int_\tau^t \langle X_s^{(h)}-Z_s^{(h)},\Delta_{\mathbf{x}} \varphi \rangle\mathrm{d}s + \langle W_t-W_{\tau},\varphi\rangle + \langle \zeta_t^{(h)}- \zeta_{\tau}^{(h)},\varphi\rangle.
\end{align}
Choosing in particular $\varphi= \mathrm{e}^{\epsilon \Delta}e_k \in \mathcal{H}^2_{\mathrm{Sym}}(\mathbb{S})$, 
for some 
$\varepsilon >0$
with  $(e_k)_{k\in \mathbb{N}}$ 
as in Subsection \ref{Notations}, we obtain, from It\^o's formula, for any 
$n \in \{0,\cdots,N-1\}$ and any $t \in [\tau \wedge t_n,\tau \wedge t_{n+1})$,
\begin{align}\label{eq:Ito:3.8}
	\mathrm{d}\langle X_t^{(h)}-Z_t^{(h)}, \mathrm{e}^{\epsilon \Delta}e_k \rangle^2 
    &= 2\langle X_t^{(h)}-Z_t^{(h)}, \mathrm{e}^{\epsilon \Delta}e_k \rangle \langle X_t^{(h)}-Z_t^{(h)}, \Delta_{\mathbf{x}} \mathrm{e}^{\epsilon \Delta}e_k \rangle 
    {\mathrm{d} t}
    + 2\langle X_t^{(h)}-Z_t^{(h)}, \mathrm{e}^{\epsilon \Delta}e_k \rangle \mathrm{d}\langle \zeta_t^{(h)}, \mathrm{e}^{\epsilon \Delta}e_k \rangle \notag
    \\
    &\quad +  
    2\langle X_t^{{(h)}}-Z_t^{{(h)}},\mathrm{e}^{\epsilon \Delta}e_k \rangle \mathrm{d}\langle W_t, \mathrm{e}^{\epsilon \Delta}e_k \rangle+ \mathrm{d}[\langle X_{{\cdot}}^{(h)}-Z_{{\cdot}}^{(h)}, \mathrm{e}^{\epsilon \Delta}e_k \rangle]_{t},
    \end{align}
{where we used the fact that the dynamics do not jump between $t_n$ and $t_{n+1}^-$, and where the last term on the right-hand side is a bracket equal to 
$\mathrm{d}[\langle X_{{\cdot}}^{(h)}-Z_{{\cdot}}^{(h)}, \mathrm{e}^{\epsilon \Delta}e_k \rangle]_{t} = \lambda_k^2 \exp(-8 \pi^2 \epsilon^2 k^2)
\ud t$.}
    
At time $t_{n+1}$ (if greater than $\tau$), 
\begin{equation}\label{eq:Ito:3.8:bis}
\begin{split}
\langle X_{t_{n+1}}^{(h)}-Z_{t_{n+1}}^{(h)}, \mathrm{e}^{\epsilon \Delta}e_k \rangle^2
&= \langle X_{t_{n+1}^-}^{(h)}-Z_{t_{n+1}}^{(h)}, \mathrm{e}^{\epsilon \Delta}e_k \rangle^2
\\
&\hspace{5pt} +
2 
\langle X_{t_{n+1}}^{(h)}-X_{t_{n+1}^-}^{(h)}, \mathrm{e}^{\epsilon \Delta}e_k \rangle
\langle X_{t_{n+1}^-}^{(h)}-Z_{t_{n+1}}^{(h)}, \mathrm{e}^{\epsilon \Delta}e_k \rangle
+
\langle X_{t_{n+1}}^{(h)}-X_{t_{n+1}^-}^{(h)}, \mathrm{e}^{\epsilon \Delta}e_k \rangle^2.
\end{split}
\end{equation}
{Above, the jump $X_{t_{n+1}}^{(h)}-X_{t_{n+1}^-}^{(h)}$ is equal to 
$\zeta_{t_{n+1}}^{(h)}-\zeta_{t_{n+1}^-}^{(h)}$, see \eqref{eq:jump}.}
By integrating both equations \eqref{eq:Ito:3.8} and \eqref{eq:Ito:3.8:bis}
from $
\tau$ to $(\tau +\theta) \wedge T$, and
then summing over $k \in  \mathbb{N}$, we arrive at the inequality
\begin{align*}
    &\|\mathrm{e}^{\epsilon \Delta}(X^{(h)}_{(\tau +\theta) \wedge T}- {Z}_{(\tau+\theta) \wedge T}^{(h)})\|_{2}^2 +2 \int_{\tau}^{(\tau +\theta) \wedge T}
    \|\mathrm{e}^{\epsilon \Delta}D_{\mathbf{x}}(X_t^{(h)}-{Z}_t^{(h)})\|_2^2\mathrm{d}t\\
    &\leq 2 \int_{\tau}^{(\tau +\theta) \wedge T}
    {\langle 
    \mathrm{e}^{2 \epsilon \Delta}(X_t^{(h)}-{Z}_t^{(h)}), \mathrm{d}W_t\rangle} +2\int_{\tau}^{(\tau +\theta) \wedge T}  {\langle \mathrm{e}^{2\epsilon \Delta}(X_{t^-}^{(h)}- {Z}_t^{(h)}),\mathrm{d}\zeta_t^{(h)} \rangle}
    \\
&\hspace{15pt}
+
\sum_{n : t_n \in (\tau,(\tau+\theta) \wedge T]}
\| \mathrm{e}^{\epsilon \Delta} ( X_{t_n}^{(h)}
- 
X_{t_n^-}^{(h)}
) \|_2^2
    + c_{\lambda,\epsilon}\theta,
\end{align*}
with $c_{\lambda,\epsilon}= \sum_{k \in \mathbb{N}} \lambda_k^2 e^{-8\pi k^2 \epsilon}$ (see \cite{DelHam25}).
The stochastic integral above, is a true martingale and therefore vanishes under the conditional expectation with respect to the $\sigma\text{-field}$ $\mathfrak{F}_{\tau}$.
As for the second term on the right-hand side, we use the property {that $(Z_{t}^{(h)})_{\tau \le t \le T}$ takes values in $  \mathcal{U}^2({\mathbb S})$ (see the last sentence in Subsection \ref{subse:rearrangement})}, and \eqref{ineq:2.5}, which says
that the path $t\mapsto \zeta^{(h)}_t$ is non-decreasing, to deduce that
\begin{equation*}
\int_{\tau}^{(\tau+ \theta) \wedge T} {\langle} \mathrm{e}^{2\epsilon \Delta} Z_t^{h}
,  {\mathrm d} \zeta_t^{(h)} \rangle  \geq 0. 
\end{equation*}
Regarding the third term on the right-hand side, we obtain from \eqref{eq:1.16:bis} (together with the contraction property of the heat semigroup) the bound:
\begin{equation*}
\sum_{n : t_n \in (\tau,(\tau+\theta) \wedge T]}
\| \mathrm{e}^{2\epsilon \Delta} ( X_{t_n}^{(h)}
- 
X_{t_n^-}^{(h)}
) \|_2^2
\leq h \left( N_{(\tau + \theta) \wedge T} - N_{\tau} + 1\right) \leq \theta + h.
\end{equation*}
Hence, by
combining the last three displays and further letting $\epsilon \searrow 0$ (by means of Fatou's lemma), we deduce thereof 
\begin{align}
    \mathbb{E}\left[\|X^{(h)}_{(\tau +\theta) \wedge T}-{Z}_{(\tau + \theta) \wedge T}^{(h)}\|_{2}^2/\mathfrak{F}_{\tau}\right]&\leq 2\lim_{\epsilon \searrow 0} \mathbb{E} \left[\int_{\tau}^{(\tau +\theta) \wedge T}\mathrm{e}^{2\epsilon \Delta}X_t^{(h)}\cdot\mathrm{d}\zeta_t^{(h)}\Big/\mathfrak{F}_{\tau}\right] + (c_{\lambda}+1) \theta + h \notag
    \\
    &=
    2 {\mathbb E}
    \Biggl[ \sum_{n : t_n \in (\tau,(\tau+\theta) \wedge T]}
    \langle X_{t_n^-}^{(h)},X_{t_n}^{(h)} - X_{t_n^-}^{(h)} \rangle
    \Big/ {\mathfrak F}_\tau
    \Biggr] + 
(c_{\lambda}+1) \theta + h \notag
    \\
    &\leq (c_{\lambda}+1) \theta + {\mathcal O}_{h}(1),\label{bound:2.16}
\end{align}
where we used the conditional version of the reflection principle \eqref{refl1} to move from the first to the second line above, and we used the bound \eqref{eq:2.7} to derive the last line,
with $c_{\lambda}:=\sum_{k \in \mathbb{N}} \lambda_k^2.$
Above, 
the Landau term 
${\mathcal O}_h(1)$ is deterministic; it tends to $0$ with $h$, uniformly with respect to the other parameters (including in particular $\theta$).
Apply now a mean inequality and the fact that for any $\varphi\in L^2(\mathbb{S})$, $\|({\rm e}^{\Delta (t-s)}-I)\varphi\|_{2,-1}^2 \leq C(t-s)\|\varphi\|_2^2$, to obtain
\begin{align*}
    \|X^{(h)}_{(\tau +\theta)\wedge T}-X^{(h)}_{\tau}\|_{2,-1}^2 &\leq C\Big(\|X^{(h)}_{(\tau +\theta)\wedge T}-{Z}_{(\tau +\theta)\wedge T}^{(h)}\|_{2,-1}^2 + \|(\mathrm{e}^{\Delta( (\tau +\theta)\wedge T -\tau)}-I)X^{(h)}_{\tau}\|_{2,-1}^2\Big)\\
    &\leq C\Big(\|X^{(h)}_{(\tau +\theta)\wedge T}-{Z}_{(\tau +\theta)\wedge T}^{(h)}\|_{2}^2 + \theta \|X^{(h)}_{\tau}\|_2^2\Big),
\end{align*}
the value of the constant $C$ being allowed to vary from line to line (as long as it remains independent of $\tau$, $\theta$ and $h$). 
Taking the conditional expectation with respect to the $\sigma$-field $\mathfrak{F}_{\tau}$ in the above inequality, 
together with the bound \eqref{bound:2.16}, we get the bound
\begin{align*}
    \mathbb{E}\left[\|X^{(h)}_{(\tau +\theta)\wedge T}-X^{(h)}_{\tau}\|_{2,-1}^2/\mathfrak{F}_{\tau}\right] \leq C\theta \Big(1+\sup_{0\leq s \leq T}\|X_s^{(h)}\|_2^2\Big) + {\mathcal O}_h(1).
\end{align*}
In particular, using  the bound \eqref{bound 2.8} and the fact that the stopping-time $\tau$ is bounded  by $T$, we deduce the existence of a constant $C$, independent of $\tau$, $\theta$ and $h$, such that
\begin{align}\label{216}
   {\mathbb E}
    \left[
    \|X^{(h)}_{(\tau +\theta)\wedge T}-X^{(h)}_{\tau}\|_{2,-1}^2 
    \right] \leq C\theta + {\mathcal O}_h(1).
\end{align}
Consequently, {for all $\epsilon>0$ and $\delta >0$}, one has 
    \begin{align*}
        \limsup_{h\searrow 0}
        \lim_{\theta_0 \searrow 0} \sup_{\theta \in (0,\theta_0]} \mathbb{P}\left\{ \|X^{(h)}_{(\tau +\theta)\wedge T}-X^{(h)}_{\tau}\|_{2,-1} \geq {\delta} \right\} \leq \epsilon.
    \end{align*}
The condition {\bf [A]} in Definition 
\ref{def:aldous}
is then satisfied. 
By 
Remark 
\ref{rem:aldous}, this implies $\mathbf{[T2]}$ in the statement of Theorem \ref{thm:tightness:D}.
\vskip 4pt

\textit{Conclusion.}
In conclusion, the law of the family $\{(X_t^{(h)})_{0 \le t \le T}\}_{{h \in T/{\mathbb N}^*}}$ forms a tight family of probability measures in the space $\mathcal{H}^{-1}_{\mathrm{Sym}}(\mathbb{S})\text{-valued}$ c\`adl\`ag processes in $[0,T].$
\end{proof}

\subsection{%
  \texorpdfstring{%
   Tightness of the laws of $\{(\zeta_t^{(h)})_{0 \le t \le T}\}_{{h \in T/{\mathbb N}^*}}$%
  }{%
    Tightness of the laws of the approximating reflection processes%
  }%
}
Using 
\eqref{eq:def:zeta:t:h} and 
\eqref{scheme0},
we can reformulate $({\zeta}_t^{(h)})_{0 \le t \le T}$ as 
\begin{equation}\label{interpo1}
	{\zeta}_t^{(h)}(\mathbf{x}) = {X}_t^{(h)}(\mathbf{x})-X_{0}(\mathbf{x}) -\int_{0}^{t} \Delta_{\mathbf{x}}{X}_s^{(h)}(\mathbf{x})\mathrm{d}s -W_t(\mathbf{x}), \quad t \in [0,T], \ \mathbf{x}\in \mathbb{S}.
\end{equation}
Here is the main result of this subsection:
\begin{prop}\label{propo 2.5} The family $\{({\zeta}_t^{(h)})_{0 \leq t \le T}\}_{{h \in T/{\mathbb N}^*}}$ given by \eqref{interpo1} induces a tight family of probability measures on the space $\mathbb{D}([0,T],\mathcal{H}^{-2}_{\mathrm{Sym}}(\mathbb{S}))$ of c\`adl\`ag functions from $[0,T]$ to $\mathcal{H}^{-2}_{\mathrm{Sym}}(\mathbb{S})$.
In addition, there exists a constant $C$ such that, for any $h \in {T/{\mathbb N}^*}$, 
    \begin{equation}\label{eq1.44}
    \begin{split}
	 \mathbb{E}\Big[ \sup_{0\leq t\leq T}\|{\zeta}_t^{(h)}\|_{2,-2}^{2} \Big]&\leq C (1+\mathbb{E}\|X_0\|_2^{2})
    \\
	\mathbb{E}\|\zeta_{t}^{(h)}-\zeta_{s}^{(h)}\|_{2,-2}^{2} &\leq  C(t-s)  + {\mathcal O}_h(1), \quad 0 \leq s < t \leq T,
    \end{split}
\end{equation}
where the Landau term 
${\mathcal O}_h(1)$ tends to $0$ with $h$, and is independent of the parameters 
$s$ and $t$. 
\end{prop}
{A careful inspection of the proof would show that the first line in 
\eqref{eq1.44} remains true with 
the $\mathcal{H}^{-2}_{\mathrm{Sym}}(\mathbb{S})$-norm replaced by 
the $\mathcal{H}^{-1}_{\mathrm{Sym}}(\mathbb{S})$-norm, but we will not make use of this improved version of 
\eqref{eq1.44}.}

Before proving the above proposition, we first establish
the following result, which provides a compactness criterion in the space $\mathcal{C}([0,T],\mathcal{H}^{-2}(\mathbb{S}))$:
\begin{lemm}\label{A3}
 Choose ${\mathbb A}= 
 L^2({\mathbb S})$ and, for a fixed $\upsilon>0$, define $$\P_{\upsilon}:=\left\{\mathcal{Y}\in \mathbb{D}([0,T],L^2(\mathbb{S}));\, 
 {\int_0^T \|D\mathcal{Y}_t\|_2^2\mathrm{d}t \leq \upsilon} \right\}.$$
    Then, $\{ t \in [0,T] \mapsto \int_0^t \Delta \mathcal{Y}_s\mathrm{d}s, \ 
    {\mathcal Y} \in \P_{\upsilon}\}$ is relatively compact in $\mathcal{C}([0,T],\mathcal{H}^{-2}(\mathbb{S})).$
\end{lemm}
\begin{proof} We  use the Arzel\`a-Ascoli theorem. As a first step, we need to prove that, for a fixed $t \in [0,T]$, the set $\{\int_0^t \Delta \mathcal{Y}_s\mathrm{d}s, \ 
    {\mathcal Y} \in \P_{\upsilon} \}$ is relatively compact in $\mathcal{H}^{-2}(\mathbb{S})$. For this, we notice 
    that it is bounded in 
    $\mathcal H^{-1}({\mathbb S})$, which indeed follows from
    the inequalities\begin{align*}
        \Big\|\int_0^t \Delta \mathcal{Y}_s\mathrm{d}s \Big\|_{2,-1} = \sup_{\substack{\varphi \in \mathcal{C}^{\infty}\\
        \|\varphi\|_{2,1}\leq 1}} \Big|\Big\langle \int_0^t \Delta \mathcal{Y}_s\mathrm{d}s,\varphi\Big\rangle \Big| \leq \sup_{\substack{\varphi \in \mathcal{C}^{\infty}\\
        \|\varphi\|_{2,1}\leq 1}} \int_0^t |\langle D\mathcal{Y}_s, D\varphi \rangle| \mathrm{d}s \leq \biggl(\int_0^t \|D\mathcal{Y}_s\|_2^2\mathrm{d}s\biggr)^{1/2}\leq \sqrt{{\upsilon}}.
    \end{align*}
   On the other hand, we also have, for any $0\leq s\leq t$, 
   \begin{align*}
       \Big\|\int_s^t \Delta \mathcal{Y}_r\mathrm{d}r \Big\|_{2,-1} = \sup_{\substack{\varphi \in \mathcal{C}^{\infty}\\
        \|\varphi\|_{2,1}\leq 1}} \Big|\langle \int_s^t \Delta \mathcal{Y}_r\mathrm{d}r,\varphi\rangle \Big| &\leq \sup_{\substack{\varphi \in \mathcal{C}^{\infty}\\
        \|\varphi\|_{2,1}\leq 1}} \int_s^t |\langle \mathcal{Y}_r,
        \Delta \varphi \rangle| \mathrm{d}r 
        \\ 
        &{\leq 
        \sqrt{t-s} \sup_{\substack{\varphi \in \mathcal{C}^{\infty}\\
        \|\varphi\|_{2,1}\leq 1}}
        \biggl( 
        \int_s^t \vert \langle D {\mathcal Y}_r, D \varphi \rangle |^2 \ud r 
        \biggr)^{1/2} 
        \leq 
\sqrt{\upsilon} \sqrt{t-s},}
   \end{align*}
which proves that the function $t \in [0,T] \rightarrow \int_0^t \Delta \mathcal{Y}_s\mathrm{d}s \in \mathcal{H}^{-1}({\mathbb S})$ is continuous, uniformly with respect to ${\mathcal Y} \in \P_{\upsilon}$.
\end{proof}

\begin{proof}[Proof of Proposition \ref{propo 2.5}] 
The proof is divided in several steps.
\vskip 4pt

\textit{First Step.} As for the first part of the statement, the goal is to prove that, for
any $\epsilon >0$, there exists a compact set $K \subset \mathbb{D}([0,T],\mathcal{H}^{-2}_{\mathrm{Sym}}(\mathbb{S}))$ such that ${\mathbb P}(\{  \zeta^{(h)} \in K\})\geq 1-\epsilon$, for all $h \in {T/{\mathbb N}^*}$.  

We first notice from Proposition \ref{Prop 3.1}
that there exists a compact set $K_1 \subset \mathbb{D}([0,T],\mathcal{H}^{-1}_{\text{Sym}}(\mathbb{S}))$ such that ${\mathbb P}(\{ X^{(h)}  \in K_1\})\geq 1-\epsilon/3$, for all {$h \in T/{\mathbb N}^*$}. Moreover, since the  noise has continuous trajectories with values in 
$L^2({\mathbb S})$, there exists a compact subset $K_2 \subset \mathcal{C}([0,T],L^2_{\mathrm{Sym}}(\mathbb{S}))$ such that $
{\mathbb P}(\{W\in K_2\}) \geq 1 - \epsilon/3$. It remains to see that there exists a 
compact subset 
$K_3 \subset \mathcal{C}([0,T],\mathcal{H}^{-2}_{\mathrm{Sym}}(\mathbb{S}))$
such that $
{\mathbb P}(\{t \in [0,T]\mapsto \int_0^t \Delta_{\mathbf{x}} X_s^{(h)}\mathrm{d}s \in K_3\}) \geq 1 - \epsilon/3$, for all {$h \in T/{\mathbb N}^*$}: this is a consequence of Lemma \ref{A3}
and of 
the estimate \eqref{bound 1.26}. 

Regarding $K_1$, $K_2$ and $K_3$ as (compact) subsets of 
$\mathbb{D}([0,T],\mathcal{H}^{-2}_{\text{Sym}}(\mathbb{S}))$
and then defining 
$K$ as the image of $K_1 \times K_2 \times K_3$ by the mapping 
\begin{equation*}
(x,y,z) \in 
\mathbb{D}([0,T],\mathcal{H}^{-2}_{\mathrm{Sym}}(\mathbb{S}))^3 \mapsto (t \in [0,T] \mapsto x_t - x_0 - y_t - z_t) \in 
\mathbb{D}([0,T],\mathcal{H}^{-2}_{\mathrm{Sym}}(\mathbb{S})),
\end{equation*}
we observe that 
$K$ is compact. 
Then, using \eqref{interpo1},
we have, for all $h \in (0,1)$,\begin{equation*}
{\mathbb P}\left(
\{\zeta^{(h)} \in K \}\right) \geq 
{\mathbb P}\left(
\{X^{(h)} \in K_1\}
\cap 
\{W \in K_2\}
\cap 
\left\{\int_0^\cdot \Delta X_s^{(h)} {\mathrm d} s\in K_3 \right\}
\right) \geq 1 - \epsilon,
\end{equation*}
from which the tightness property follows.
\vskip 4pt

\textit{Second Step.}
We now prove \eqref{eq1.44}.
From \eqref{interpo1} and  the triangle inequality, there exists a universal constant $C$ such that, for any $t \in [0,T]$,
\begin{align*}
\|	{\zeta}_t^{(h)}\|_{2,-2}^2 &\leq C\Biggl(\|{X}_t^{(h)}-X_0\|_{2,-2}^2 + \Big\|\int_0^t \Delta_{\mathbf{x}} {X}_s^{(h)}\mathrm{d}s \Big\|_{2,-2}^2 + \|W_t\|_{2,-2}^2\Biggr)\\
 &\leq C\Biggl(\|{X}_t^{(h)}-X_0\|_{2,-1}^2 + \sum_{k\in\mathbb{N}}(k\vee 1)^{-4} \int_0^t \langle\Delta_{\mathbf{x}} {X}_s^{(h)},e_k \rangle^2\mathrm{d}s + \|W_t\|_{2}^2\Biggr) \\
&\leq C \Biggl( \|{X}_t^{(h)}-X_0\|_{2,-1}^2 + \int_0^T  \sum_{k\in\mathbb{N}}\langle {X}_s^{(h)},e_k \rangle^2\mathrm{d}s + \|W_t\|_{2}^2\Biggr).
\end{align*}
Therefore, taking first the supremum over $t$, the expectation value and using \eqref{bound 2.8}, we obtain
\begin{align*}
\mathbb{E}\Big[\sup_{0\leq t\leq T}\|{\zeta}_t^{(h)}\|_{2,-2}^{2}\Big] &\leq C_{\lambda,T}\big( 1+  \mathbb{E} \sup_{s \in [0, T]} \|{X}_s^{(h)}\|_2^{2}\big)\leq C_{\lambda,T}(1+\mathbb{E}\|X_0\|_2^2),
\end{align*}
which leads to the first bound in \eqref{eq1.44}. 
\vskip 4pt

\textit{Third Step.}
To prove the second bound in \eqref{eq1.44}, consider now  $0\leq s < t \leq T$ and let $\phi \in \mathcal{H}^2_{\mathrm{Sym}}(\mathbb{S})$. The scheme \eqref{scheme0} writes:
\begin{align}\label{scheme2}
\langle X_{t}^{(h)} -X_{s}^{(h)},\phi \rangle = \langle \zeta_{t}^{(h)}-\zeta_{s}^{(h)},\phi\rangle  + \int_{s}^{t} \langle X_r^{(h)},\Delta_{\mathbf{x}} \phi \rangle \mathrm{d}r + \langle W_{t}-W_{s},\phi \rangle.
\end{align} 
Applying the Cauchy-Schwartz inequality and assuming that 
$\| \varphi \|_{2,2} \leq 1$, we obtain
\begin{align*}
	\|\zeta_t^{(h)}-\zeta_s^{(h)}\|_{2,-2}^{2} \leq C\left( \| X_t^{(h)} -X_s^{(h)}\|_{2,-1}^2 +\int_{{s}}^{t} \|X_r^{(h)}\|_{2}^2\mathrm{d}r + \|  W_{t}-W_{s}\|_{2}^2\right). 
	\end{align*}
Taking the expectation value on both sides, and combining \eqref{216} and  \eqref{bound 2.8}, we get  the result.
\end{proof}
We conclude this subsection with the following remark
\begin{remark}\label{remark:dynamics}[Dynamics of the weak limits-- A naive attempt]Here, we provide a preliminary discussion of the equation that any weak limit of the family $\{(X_t^{(h)},W_t)_{0 \le t \le T}\}_{{h \in T/{\mathbb N}^*}}$ is expected to satisfy as $h \searrow 0$. For simplicity, we denote such a limit by $(X_t,W_t)_{0 \le t \le T}$ and assume that it is constructed on the same stochastic basis $(\Omega,\mathfrak{F},\{\mathfrak{F}_t\}_{t\ge 0},\mathbb{P},(W_t)_{t \geq 0})$ as before.
The existence of a limit follows from the tightness results established so far. The challenge thus is to  characterize  the limiting dynamics.


Hence, letting $h\searrow 0$ (up to a subsequence) in \eqref{scheme2}, the dynamics of the limit process is expected to satisfy, with probability one, for any $0\leq s \leq t\leq T$ and any smooth enough $\varphi \in L^{2}_{\mathrm{Sym}}(\mathbb{S})$,
\begin{align}\label{eq2.39}
	\langle X_t-X_s, \phi \rangle = \langle \zeta_t-\zeta_s, \phi \rangle + \int_s^t \langle X_r,\Delta_{\mathbf{x}} \phi \rangle \mathrm{d}r + \langle W_t-W_s,\phi\rangle.
\end{align}
Nevertheless, at the present stage, the equation presented above does not fully reveal the precise nature of the process $(\zeta_t)_{0 \le t \le T}$, which must be viewed as a weak-limit of the family $\{(\zeta_t^{(h)})_{0 \le t \le T}\}_{{h \in T/{\mathbb N}^*}}$. 
The identification of this process as a new reflection term
(distinct from 
$\eta$ in \eqref{0.1})
is the main objective of the next section.
{Before proceeding, we need an additional tightness property on the pure jump parts of the family $(X^{(h)})_{{h \in T/{\mathbb N}^*}}$ (even though these jumps are, for each ${h \in T/{\mathbb N}^*}$, already included in the process $\zeta^{(h)}$).}
\end{remark}

\subsection{%
  \texorpdfstring{%
    Tightness of the laws of $\{(\mathscr{J}_{t}(X^{(h)}))_{0 \le t \le T}\}_{{h \in T/{\mathbb N}^*}}$
  }{%
    Tightness of the laws of the pure jump parts%
  }%
}

We recall the definition of 
$({\mathscr J}_t(X^{(h)}))_{0 \le t \le T}$ given in 
\eqref{eq:def:zeta:t:h}:
\begin{equation}
\label{237}
{\mathscr J}_t(X^{(h)}) 
= \sum_{j=1}^{N_t}
\left(\Phi^{(h)}(X_{t_j^-}^{(h)}) -X_{t_j^-}^{(h)}
    \right), \quad t \in [0,T]. 
    \end{equation}
The analysis of the jumps is carried out using Lemma \ref{lemm1.11}, which in turn suggests decomposing general test functions into non-decreasing and non-increasing parts.

\begin{lemm}\label{lemm1.12} Fix $\delta >0$.
	Let $\phi \in \mathcal{H}^{3+\delta}_{\mathrm{Sym}}(\mathbb{S})$. Then, there exist two non-decreasing functions $\phi^{\uparrow}$ and $\phi^{\downarrow}$, both in the space  $\mathcal{H}^2_{\mathrm{Sym}}(\mathbb{S})$, such that $\phi = \phi^{\uparrow} - \phi^{\downarrow}$ and $\sup(\|\phi^{\uparrow}\|_{2,2} + \|\phi^{\downarrow}\|_{2,2} ) \leq \tilde{C} \|\phi\|_{2,(3+\delta)}$, {for a constant} $\tilde{C}>0$ {independent of $\varphi$}.
\end{lemm}
To prove the above lemma, we shall rely on the fact that each element of the sequence $(e_k)_{k \in \mathbb{N}}$  can be decomposed as the difference of two symmetric, non-decreasing functions, denoted by $e_k^{+}$ and $e_k^{-}$, respectively, and defined by the following expressions (see \cite{DelHam25}) 
\begin{align}\label{decomp 33}
	e_k^{\pm}(x)= e_k(0)\iota +\int_0^x \mathbbm{1}_{[-1/2,0]}(y)\left[D_ye_k(y)\right]^{\pm}\mathrm{d}y -\int_0^x \mathbbm{1}_{[0,1/2]}(y)\left[D_ye_k(y)\right]^{\mp}\mathrm{d}y,
\end{align}
with $\iota =1 \text{ if } \pm =+$ and $0$ elsewhere, 
{where the superscripts $\pm$ and $\mp$ on the right-hand side denote the positive and negative parts of the real number at hand.}
Clearly, 
$e_k=e_k^+-e_k^-$. Moreover, once again from \cite{DelHam25}, we have $e_k^{+}, e_k^{-} \in \mathcal{H}^{2}_{\mathrm{Sym}}(\mathbb{S})\cap \mathcal{U}^2(\mathbb{S})$, and there exists a universal constant $C>0$ such that
\begin{align}\label{decomp34}
    \|e_k^{\pm}\|_{2,2} \leq C (k^2\vee 1),
\end{align}
where the symbol $e_k^{\pm}$ indicates that the above bound holds for both $e_k^+$ and $e_k^-$.
\begin{proof}[Proof of Lemma \ref{lemm1.12}] Let $\varphi \in \mathcal{H}^{3+\delta}_{\mathrm{Sym}}(\mathbb{S})$.
Denoting by 
$(\hat{\varphi}_k)_{k \in {\mathbb N}}$  the (real-valued) Fourier coefficients of $\varphi$ and using \eqref{decomp34}, we first notice that 
\begin{equation}
    \label{eq:06-07:1}
\begin{split}
 \sum_{k\in \mathbb{N}}    \big\|\hat{\varphi}_k e_k^{+}\big\|_{2,2}
    &\leq \sum_{k\in \mathbb{N}}\frac{|\langle (k\vee 1)^{(3+\delta)}{\varphi},e_k\rangle|}{(k\vee 1)^{(3+\delta)}}\|e_k^{+}\|_{2,2} \leq C\|{\varphi}\|_{2,(3+\delta)}\sum_{k\in \mathbb{N}}\frac{1}{(k\vee 1)^{(1+\delta)}}.
\end{split}
\end{equation}
Then, the two series on the right-hand side below are normally convergent and, for any ${\mathbf x} \in {\mathbb S}$, 
    \begin{equation*}
    \begin{split}
        \varphi(\mathbf{x}) &= \sum_{k\in \mathbb{N}} \hat{\varphi}_k  e_k^{+}(\mathbf{x}) - \sum_{k\in \mathbb{N}} \hat{\varphi}_k e_k^{-}(\mathbf{x})\\
        \hspace{20pt}
        &=\Big(\sum_{k\in \mathbb{N}} \hat{\varphi}_k^{+}  e_k^{+}(\mathbf{x}) +\sum_{k\in \mathbb{N}} \hat{\varphi}_k^{+} e_k^{-}(\mathbf{x}) \Big) -\Big( \sum_{k\in \mathbb{N}} \hat{\varphi}_k^{-} e_k^{+}(\mathbf{x}) + \sum_{k\in \mathbb{N}} \hat{\varphi}_k^{+} e_k^{-}(\mathbf{x}) \Big)\\
        &=: \varphi^{\uparrow}(\mathbf{x}) - \varphi^{\downarrow}(\mathbf{x}),
        \end{split}
    \end{equation*}
    where we used the decomposition $\hat\varphi_k = \hat\varphi_k^{+}-\hat\varphi_k^{-}$, with $\hat\varphi_k^{+}$ and $\hat\varphi_k^{-}$ representing, respectively, the positive and negative part of the (real-valued) Fourier coefficient $\hat\varphi_k$. It is readily seen that both $\varphi^{\uparrow}$ and $\varphi^{\downarrow}$ are symmetric and non-decreasing functions. 
    
   We conclude the proof by  establishing the bound $\sup(\|\phi^{\uparrow}\|_{2,2} + \|\phi^{\downarrow}\|_{2,2} ) \leq \tilde{C}\|\phi\|_{2,(3+\delta)}$, for $\tilde C$ possibly depending on $\delta$, but not on $\varphi$. We  only derive a bound for the first term on the left hand side of the latter display, as the second one can be handled in an entirely similar manner. 
   Following 
   \eqref{eq:06-07:1}, we obtain
\begin{equation*}
\begin{split}
    \Big\|\sum_{k\in \mathbb{N}} \hat{\varphi}_k^{+} e_k^{+}\Big\|_{2,2}
    &\leq \sum_{k\in \mathbb{N}}\frac{|\langle (k\vee 1)^{(3+\delta)}{\varphi},e_k\rangle|}{(k\vee 1)^{(3+\delta)}}\|e_k^{+}\|_{2,2} \leq C\|{\varphi}\|_{2,(3+\delta)}\sum_{k\in \mathbb{N}}\frac{1}{(k\vee 1)^{(1+\delta)}}.
\end{split}
\end{equation*}
The proof is easily completed. 
\end{proof}
Now we are in a position to establish the main result of this subsection.
\begin{prop}\label{propo:3.8} Fix $\delta>0$. Then, the family $\{(\mathscr{J}_{t}(X^{(h)}))_{0 \le t \le T}\}_{{h \in {\mathbb T}/{\mathbb N}^*}}$ given by \eqref{237}, induces a tight family of probability measures on the space $\mathbb{D}([0,T],\mathcal{H}^{-(3+\delta)}_{\mathrm{Sym}}(\mathbb{S}))$ of c\`adl\`ag functions from $[0,T]$ to $\mathcal{H}^{-(3+\delta)}_{\mathrm{Sym}}(\mathbb{S})$.  
\end{prop}
\begin{proof} 
Fix $\delta >0$ and let $\varphi \in \mathcal{H}^{3+\delta/2}_{\mathrm{sym}}(\mathbb{S})$ with $\|\varphi\|_{2,(3+\delta/2)} \leq 1$. Then, combining 
\eqref{eq:def:zeta:t:h}
and 
Lemma \ref{lemm1.12} with the fact that for all  $0\leq s\leq t \le T$, for any $\psi \in \mathcal{U}^2(\mathbb{S})\cap \mathcal{H}^{2}_{\mathrm{Sym}}(\mathbb{S})$, 
$\langle \eta_t^{(h)}-\eta_s^{(h)},\psi \rangle \geq 0$
and 
$\langle 
\mathscr{J}_{t}(X^{(h)})-\mathscr{J}_{s}(X^{(h)}),\psi \rangle \geq 0$ (see \eqref{ineq:2.5:b}), we have
\begin{align}
	&|\langle \mathscr{J}_{t}(X^{(h)})-\mathscr{J}_{s}(X^{(h)}),\phi^{\uparrow}-\phi^{\downarrow} \rangle|
    \notag
    \\
    &\leq |\langle \mathscr{J}_{t}(X^{(h)})-\mathscr{J}_{s}(X^{(h)}),\phi^{\uparrow}\rangle| + |\langle \mathscr{J}_{t}(X^{(h)})-\mathscr{J}_{s}(X^{(h)}),\phi^{\downarrow} \rangle|\notag
    \\
    &= \langle \mathscr{J}_{t}(X^{(h)})-\mathscr{J}_{s}(X^{(h)}),\phi^{\uparrow}\rangle + \langle \mathscr{J}_{t}(X^{(h)})-\mathscr{J}_{s}(X^{(h)}),\phi^{\downarrow} \rangle
    \notag
    \\
	&\color{black} \leq \|\zeta_t^{(h)}-\zeta_s^{(h)}\|_{2,-2}\left(\|\phi^{\uparrow} \|_{2,2} + \|\phi^{\downarrow} \|_{2,2} \right), \label{ineq:3.18}
\end{align}
Consequently,
using Lemma \ref{lemm1.12}, 
we rewrite\eqref{ineq:3.18}in the form
\begin{align}\label{239}
	\| \mathscr{J}_{t}(X^{(h)})-\mathscr{J}_{s}(X^{(h)})\|_{2,-(3+\delta/2)} \leq \tilde{C}\|\zeta_t^{(h)}-\zeta_s^{(h)}\|_{2,-2},
    \end{align}
where $\tilde C$ depends on $\delta$.    
Observing $\mathscr{J}_{0}(X^{(h)})= \zeta_0^{(h)} =0$, letting $s=0$ in \eqref{239}
and using the first
bound \eqref{eq1.44}, we deduce thereof
\begin{align}\label{bound 3.6}
\sup_{t \in [0,T]}\mathbb{E}\bigl[\|\mathscr{J}_{t}(X^{(h)})\|_{2,-(3+\delta/2)}^2\bigr] \leq C (1+\mathbb{E}\|X_0\|^{2}_2),
\end{align}
for $C$ depending on $\delta$. 
Using the fact that, for any $\delta>0$, bounded sets in $\mathcal{H}^{-(3+\delta/2)}_{\text{sym}}(\mathbb{S})$ are relatively compact in the space $\mathcal{H}^{-(3+\delta)}_{\text{sym}}(\mathbb{S})$  (see Proposition \ref{prop1.4}),  the bound \eqref{bound 3.6} implies that, for all $t\in [0,T]$, the laws of $\{\mathscr{J}_{t}(X^{(h)})\}_{{h \in T/{\mathbb N}^*}}$ form a tight sequence of laws in $\mathcal{H}^{-(3+{\delta})}_{\mathrm{Sym}}(\mathbb{S})$. The condition {\bf [T1]} of  Theorem \ref{thm:tightness:D} is then satisfied.

To conclude, we now need to establish that the family $\{\mathscr{J}_{t}(X^{(h)})\}_{{h \in T/{\mathbb N}^*}}$ satisfies the condition {\bf [T2]} of Theorem \ref{thm:tightness:D}. For this, we recall that the modulus of continuity \eqref{A1} of the process $t\mapsto \mathscr{J}_t(X^{(h)})$ is precisely given by
\begin{align}
        \omega_{[0,T],\mathcal{H}^{-(3+\delta)}}(\mathscr{J}(X^{(h)}),\gamma):= \inf_{\Pi_{\gamma}}\max_{{r_i \in \tilde{\pi}} }\sup_{{r_i\leq s<t<r_{i+1}}} \|\mathscr{J}_{t}(X^{(h)})-\mathscr{J}_{s}(X^{(h)})\|_{2,-(3+\delta)},
    \end{align}
where, $\Pi_{\gamma}$ stands for the set of all increasing sequences {$\tilde{\pi}=\{0=r_0<r_1<\cdots<r_n =T\}$ satisfying $r_{i+1}-r_i \geq \gamma,$ for $i=0,1,\cdots,n-1$.}
Then, for any $\gamma>0$, the bound \eqref{239}, leads to $\omega_{[0,T],\mathcal{H}^{-(3+\delta)}}(\mathscr{J}(X^{(h)}),\gamma) \leq \tilde{C}\omega_{[0,T],\mathcal{H}^{-2}}(\zeta^{(h)},\gamma)$.
Consequently, for any $\epsilon >0$
\begin{align*}
   \lim_{\gamma \searrow 0} \limsup_{h \searrow 0} \mathbb{P}\Big\{ \omega_{[0,T],\mathcal{H}^{-(3+\delta)}}(\mathscr{J}(X^{(h)}),\gamma)\geq \epsilon  \Big\} \leq \lim_{\gamma \searrow 0} \limsup_{h \searrow 0}\mathbb{P}\Big\{ \omega_{[0,T],\mathcal{H}^{-2}}(\zeta^{(h)},\gamma)\geq \frac{\epsilon}{\tilde{C}} \Big\} = 0,
\end{align*}
where the above equality follows from the tightness of $(\zeta^{(h)}_t)_{{h \in T/{\mathbb N}^*}}$ in $\mathcal{H}^{-2}_{\mathrm{Sym}}(\mathbb{S})$ (see Proposition \ref{propo 2.5}). This
completes the proof.
%
\end{proof}
\section{Weak limits}
\label{se:5}
We identify the stochastic dynamics arising in the limit of the  scheme \eqref{scheme0} as the solution of a 
rearranged stochastic heat equation
drifted by the derivative of the entropy on ${\mathcal P}({\mathbb R})$.

\subsection{Continuity of the limiting dynamics}

We start this section with the following result.
\begin{prop}\label{remark 4.1}
    For $\delta >0$, let $(X_t, W_t,\zeta_t,{\mathscr J}_t)_{0 \le t \le T}$ be a weak limit in $\mathbb{D}([0,T], \mathcal{H}^{-1}_{\mathrm{Sym}}(\mathbb{S})\times L^2_{\mathrm{Sym}}(\mathbb{S})
\times 
{\mathcal H}^{-2}_{\rm Sym}(\mathbb{S})
    \times 
    \mathcal{H}^{-(3+\delta)}_{\mathrm{Sym}}(\mathbb{S})
    )$ of the family $\{(X_t^{(h)},W_t,\zeta_t^{(h)},{\mathscr J}_t(X^{(h)})_{0 \le t \le T}\}_{{h \in T/{\mathbb N}^*}}$ when $h\searrow 0$. Assume (without any loss of generality) that it is constructed on the same probability space $(\Omega,\mathfrak{F},\mathbb{P})$ as the approximation scheme. Then: 
    \begin{itemize}
    \item[1.] With probability 1, $(X_t)_{0 \le t \le T}$ is in $\mathcal{C}([0,T], \mathcal{H}^{-1}_{\mathrm{Sym}}(\mathbb{S}))
    \cap L^\infty([0,T], 
    \mathcal{U}^2(\mathbb{S}))
    \cap L^2([0,T],\mathcal{H}^{1}_{\mathrm{Sym}}(\mathbb{S}))$. It satisfies  
\begin{equation}
    \label{260}
\mathbb{E}\bigl[\|X_{t}\|_2^2\bigr] + 
2\int_0^{t} \mathbb{E} \bigl[\|D_{\mathbf{x}}X_{r}\|_2^2\bigr] \mathrm{d}r \leq \mathbb{E}\bigl[\|X_0\|_2^2\bigr] +t  + c_{\lambda}t, \quad t \in [0,T]. 
\end{equation}
Moreover,
    $$\mathbb{E}
    \left[ \sup_{t \in [0,T]}
    \| X_t \|_2^2 + 
    \int_0^T \|D_{\mathbf{x}}X_t\|_2^2\mathrm{d}t 
    \right] \leq C(1+\mathbb{E}\|X_0\|_2^2),$$ for a constant $C$ independent of $X_0$.  
    \item[2.] The process $(W_t)_{0 \le t \le T}$ is an $L^2_{\mathrm{Sym}}(\mathbb{S})\text{-valued}$ $Q\text{-Brownian}$ motion.
    \item[3.] With probability 1, $(\zeta_t)_{0 \le t \le T}$ takes values in $\mathcal{C}([0,T], \mathcal{H}^{-2}_{\mathrm{Sym}}(\mathbb{S}))$.
Moreover, if $\varphi \in \mathcal{H}^{2}_{\mathrm{Sym}}(\mathbb{S})$ is non-decreasing, then, almost-surely, the path $(\langle \zeta_t,\varphi\rangle)_{0 \le t \le T}$ is non-decreasing on $[0,T]$.
    \item[4.] With probability 1,  for any $\varphi \in \mathcal{H}^{2}_{\mathrm{Sym}}(\mathbb{S})$, for all $0\leq s\leq t\leq T$,
    \begin{align*}
 \langle X_t -X_s,\phi \rangle = \langle \zeta_t-\zeta_s,\phi\rangle  + \int_{s}^{t} \langle X_r,\Delta_{\mathbf{x}} \phi \rangle \mathrm{d}r + \langle W_{t}-W_{s},\phi \rangle.
\end{align*}
\item[5.] With probability 1, 
$({\mathscr J}_t)_{0 \le t \le T}$ is in 
$\mathcal{C}([0,T],\mathcal{H}_{\mathrm{Sym}}^{-(3+\delta)}(\mathbb{S}))$. Moreover, if $\varphi \in \mathcal{H}^{3+\delta}_{\mathrm{Sym}}(\mathbb{S}) \cap {\mathcal U}^2({\mathbb S})$, then, almost-surely, the paths $(\langle {\mathscr J}_t,\varphi\rangle)_{0 \le t \le T}$
and $(\langle 
\eta_t ,\varphi \rangle := \langle \zeta_t - {\mathscr J}_t, \varphi \rangle)_{0 \le t \le T}$
are non-decreasing on $[0,T]$, and 
\begin{equation*} 
\forall s,t \in [0,T] \ {\rm with} \ s <t, \quad 
\max \left( \langle {\mathscr J}_t - {\mathscr J}_s,\varphi \rangle, 
\langle \eta_t - \eta_s,\varphi \rangle \right) \leq 
\| \zeta_t - \zeta_s \|_{2,-2}
\| \varphi \|_2.
\end{equation*}
\end{itemize}
\end{prop}

 Notice from a standard separability argument that the quantifiers ``with probability 1'' and ``for any 
$\varphi$'' appearing in items 3, 4 and 5, can be exchanged.

\begin{proof}
The existence of a weak limit follows from Propositions \ref{Prop 3.1} and \ref{propo 2.5}. Item 2 of the statement is obvious. 
\vskip 4pt

\textit{First Step.} We first establish item 1. 
We consider the maximal jump of the scheme \eqref{scheme0}, measured in the $\mathcal{H}^{-1}_{\mathrm{Sym}}(\mathbb{S})\text{-norm}$:  
$$
\mathfrak{J}(X^{(h)}):= \sup_{0<t\leq T}\|X_{t}^{(h)}-X_{t^-}^{(h)}\|_{2,-1}^2.$$
From \eqref{scheme0} together with \eqref{eq:1.16:bis}, we obtain the estimate: 
\begin{align}
\label{239b}
   \mathfrak{J}(X^{(h)}) \leq \sup_{0<t\leq T}\|\zeta_t^{(h)}-\zeta_{t^-}^{(h)}\|_{2}^2 = \sup_{0<t\leq T} \|\Phi^{(h)}(X_{t^-}^{(h)})-X_{t^-}^{(h)}\|_{2}^2 \leq  h.
\end{align}
In particular, the scheme \eqref{scheme0} admits uniformly bounded jumps and the functional $\mathfrak{J}(X^{(h)})$ converges to zero as $h\searrow 0$. By the continuous mapping theorem, we deduce that $\mathfrak{J}(X)= \sup_{0<t\leq T}\|X_{t}-X_{t^-}\|_{2,-1}^2 =0$ (see for instance
\cite[Example 12.1 and Theorem 13.4]{Billingsley}, 
which are stated for real-valued processes but which may be easily extended to the Hilbert-valued setting). Hence, with probability one, $(X_t)_{t\in[0,T]} \in \mathcal{C}([0,T],\mathcal{H}^{-1}_{\mathrm{Sym}}(\mathbb{S})).$ 
By a  lower semi-continuity argument (see the proof of \cite[Proposition 6]{DelHam25}), we can deduce 
from 
the bound 
\eqref{bound 2.8} (in the proof of Proposition \ref{Prop 3.1}) that 
\begin{equation*}
{\mathbb E}\bigl[ 
\sup_{t \in [0,T]}
\| X_t \|_2^2
\bigr]
\leq C  \left(1 +\mathbb{E}[ \|X_{0}\|_2^{2}]\right ).
\end{equation*}
Moreover, by Proposition \ref{prop1.4} and the portemanteau theorem, 
we see that, with probability 1, for any 
$t \in [0,T]$, $X_t \in {\mathcal U}^2({\mathbb S})$. It remains to prove that, almost-surely, $(X_t)_{0 \le t \le T}$ belongs to the space $L^2([0,T],\mathcal{H}^{1}_{\mathrm{Sym}}(\mathbb{S})).$
Using \eqref{bound 2.9} and the contraction principle of the heat operator, we arrive at
\begin{align}\label{259}
	\mathbb{E}\|\mathrm{e}^{\epsilon \Delta}X_{t}^{(h)}\|_2^2 + 2\int_{0}^{t}\mathbb{E} \|D_{\mathbf{x}}(\mathrm{e}^{\epsilon \Delta}X_{r}^{(h)})\|_2^2\mathrm{d}r \leq \mathbb{E}\|X_0\|_2^2+ t +  c_{\lambda}t .
\end{align}
Additionally, since the mapping $\mathcal{Y}\mapsto D_{\mathbf{x}} (\mathrm{e}^{\epsilon \Delta}\mathcal{Y})$ is continuous from $\mathcal{H}^{-1}(\mathbb{S})$ into $L^2(\mathbb{S})$, it follows that \eqref{259} also holds for the weak limit as $h\searrow 0$. More precisely, we have
\begin{align*}
	\mathbb{E}\|\mathrm{e}^{\epsilon \Delta}X_{t}\|_2^2 + 2\int_{0}^{t}\mathbb{E} \|D_{\mathbf{x}}(\mathrm{e}^{\epsilon \Delta}X_{r})\|_2^2\mathrm{d}r \leq \mathbb{E}\|X_0\|_2^2+ t +  c_{\lambda}t 
\end{align*}
and the above bound is 
true for any value of the
regularization parameter $\epsilon>0$.
Now, letting $\epsilon\searrow 0$, respectively, in \eqref{259}, we deduce thereof
\begin{align*}
\mathbb{E}\|X_{t}\|_2^2 + 
2\int_0^{t} \mathbb{E} \|D_{\mathbf{x}}X_{r}\|_2^2 \mathrm{d}r \leq \mathbb{E}\|X_0\|_2^2 +t  + c_{\lambda}t, \quad t \in [0,T], 
\end{align*}
which is \eqref{260}

\textit{Second Step.} 
By weak convergence, we know that, for any 
$\varphi \in {\mathcal H}_{\mathrm{Sym}}^4({\mathbb S})$, 
\eqref{eq2.39} holds true almost surely. By the first step, the duality product $\langle X_r,\Delta_{\mathbf{x}} \varphi \rangle$ appearing therein is well-defined for 
$\varphi \in {\mathcal H}_{\mathrm{Sym}}^2({\mathbb S})$, 
which permits to regard 
$(\zeta_t)_{0 \le t \le T}$ as a continuous process with values in ${\mathcal H}_{\mathrm{Sym}}^{-2}({\mathbb S})$, i.e., with probability 1,  $(\zeta_t)_{0 \le t \le T}\in \mathcal{C}([0,T],\mathcal{H}^{-2}_{\mathrm{Sym}}(\mathbb{S}))$.
When $\varphi$ is smooth and non-decreasing,  monotonicity of the path $t\mapsto \langle \zeta_t,\varphi\rangle$  follows from \eqref{ineq:2.5};
by a density argument, this remains true when $\varphi$ is just in 
$\mathcal{H}^{2}_{\mathrm{Sym}}(\mathbb{S})$.
This proves items 3 and 4. 
 
By \eqref{239}
and 
\eqref{239b}, 
jumps vanish asymptotically. We deduce that
$(\mathscr{J}_t)_{0 \le t \le T}$ admits continuous trajectories
with values in 
$\mathcal{H}_{\mathrm{Sym}}^{-(3+\delta)}(\mathbb{S})$.
This is the first part of item 5.
Monotonicity of 
$({\mathscr J}_t)_{0 \le t \le T}$ 
is proven by means of 
\eqref{237}
and
\eqref{ineq:2.5:b}, from which we know that, for any stepsize $h>0$, the trajectories of 
$({\mathscr J}_t(X^{(h)}))_{0 \le t \le T}$ are non-decreasing.
Similarly, 
monotonicity of 
$(\eta_t)_{0 \le t \le T}$ follows the fact that, for any stepsize $h>0$, 
the trajectories of 
$(\eta_t^{(h)})_{0 \le t \le T}$ are non-decreasing. The last display in the statement follows just from the decomposition:
\begin{equation*}
\langle \zeta_t - \zeta_s,\varphi \rangle 
= 
\langle {\mathscr J}_t - 
{\mathscr J}_s,\varphi\rangle
+
\langle \eta_t - \eta_s,\varphi \rangle,
\qquad 
\varphi \in {\mathcal H}_{\mathrm{Sym}}^{3+\delta}({\mathbb S}),
\end{equation*}
together with the fact that both terms on the right-hand side are non-negative. 
\end{proof}

\subsection{%
  \texorpdfstring{%
  On weak limits of  $\{(\mathscr{J}_t(X^{(h)}))_{0 \le t \le T}\}_{{h \in T/{\mathbb N}^*}}$%
  }{%
    On weak limits of the pure jump parts%
  }%
}

In order to identify the properties of the process $({\mathscr J}_t)_{0 \le t \le T}$ in the statement of 
Proposition \ref{remark 4.1}, we first prove
 the following result:

\begin{prop}
\label{prop:5.2:06-07:1}
Let 
$h \in T/{\mathbb N}^*$
and 
$\varphi \in \mathcal{C}^1(\mathbb{S}) \cap {\mathcal U}^2(\mathbb{S})$.
Then, 
for all $0\leq s\leq t\leq T$,
	\begin{align}\label{2.39}
		0\leq \langle \varphi, \mathscr{J}_{t}(X^{(h)}) -\mathscr{J}_{s}(X^{(h)})\rangle =\frac{1}{2}\sum_{j=N_s+1}^{N_t} \int_0^h\int_{\mathbb{S}} \varphi'(\mathbf{x})\frac{1}{D_{\mathbf{x}}F_{\mu_{j},r}^{-1}(\mathbf{x})}\mathrm{d}\mathbf{x}\mathrm{d}r,
	\end{align}
    where, for all $r>0$,
\begin{align}\label{3.8}
\mathbb{S}\ni \mathbf{x}\mapsto F_{\mu_j,r}^{-1}(\mathbf{x}) =(F_{\mu_j\star\Gamma_r})^{-1}(\mathbf{x}):= \tilde{F}_{\mu_{j},r}^{-1}(2{x}){\mathbbm 1}_{{(0,1/2)}}(x) + \tilde{F}_{\mu_{j},r}^{-1}(-2{x}) {\mathbbm 1}_{{(-1/2,0)}}({x}),
\end{align}
with $x$ denoting the representative of 
${\mathbf x}$ in $(-1/2,1/2]$, and 
with $\mu_{j}$ denoting the law of $X_{t_j^{-}}^{(h)}$, that is,  $\mu_{j}:= \mathrm{Leb}_{\mathbb{S}}\circ (X_{t_j^{-}}^{(h)})^{-1}$, for $j\in \{1,\ldots, N\}$,
and 
$\tilde{F}_{\mu_{j},t}(z):= (\tilde{F}_{\mu_{j}}\star\Gamma_t)(z)= (\tilde{F}_{\mu_{j}\star\Gamma_t})(z) $, for $z \in {\mathbb R}$. 

{Here, we impose 
$F_{\mu_j,r}^{-1}({\mathbf x})$ to be equal to 
$0$ when $x=0$ and $x=1/2$, but, due to the Gaussian convolution, $\mu_j \star \Gamma_r$ has the whole ${\mathbb R}$ as support. In particular, 
$\tilde{F}^{-1}_{\mu_j,r}(0)$
is equal to $-\infty$
and 
$\tilde{F}^{-1}_{\mu_j,r}(1)$
to $+ \infty$. 
Our choice of letting 
$F_{\mu_j,r}^{-1}({\mathbf x})$ be equal to $0$ when 
$x=0$ or $1/2$ is for simplicity. This does not have an impact on the proof.}
\end{prop}

\begin{proof}
Let $t>0$. 
By Lemma~\ref{lemm1.11} (summing over 
$j \in \{1,\cdots,N_t\}$), we know that
	\begin{align*}
		\Bigl\langle \varphi, \sum_{j=1}^{N_t}(F_{\mu_j\star \Gamma_h}^{-1}-F_{\mu_j}^{-1}) \Bigr\rangle = \frac{1}{4} \sum_{j=1}^{N_t} \int_{\mathbb{R}}\int_0^h \phi'\Big(\frac{(\tilde{F}_{\mu_{j}}\star\Gamma_r)(z)}{2}\Big)[D_z( \tilde{F}_{\mu_{j}}\star\Gamma_r)]^2(z)\mathrm{d}r\mathrm{d}z. 
	\end{align*} 
Recalling the notation $\tilde{F}_{\mu_{j},r}(z)= (\tilde{F}_{\mu_{j}}\star\Gamma_r)(z)= (\tilde{F}_{\mu_{j}\star\Gamma_r})(z) $, we have the identity $\tilde{F}_{\mu_{j},r}^{-1}(\tilde{F}_{\mu_{j},r}(z))=z$ for any $z\in \mathbb{R}$. Moreover,  $\tilde{F}_{\mu_{j},r}$ is continuously differentiable
and (strictly) increasing (strict monotonicity follows from the identity 
$D_z \tilde{F}_{\mu_{j},r}(z) 
= (D_z \tilde{F}_{\mu_{j},r/2} \star \Gamma_{r/2})(z)$; if the left-hand were equal to 0 at some point $z$, then $D_z \tilde{F}_{\mu_{j},r/2}$ would be identically null). And then, the converse $\tilde{F}^{-1}_{\mu_{j},r}$
is also continuously differentiable and
\begin{align*}
D_z\tilde{F}_{\mu_{j},r}(z)D_z\tilde{F}_{\mu_{j},r}^{-1}\left( \tilde{F}_{\mu_{j},r}(z)\right)=1 \Leftrightarrow  D_z\tilde{F}_{\mu_{j},r}(z)= \frac{1}{D_z\tilde{F}_{\mu_{j},r}^{-1}\left(\tilde{F}_{\mu_{j},r}(z)\right)}.
\end{align*}
 Therefore, plugging back the above identity into the previous one, we obtain:
	\begin{align*}
		\Bigl\langle \varphi, \sum_{j=1}^{N_t}(F_{\mu_j\star \Gamma_h}^{-1}-F_{\mu_j}^{-1})  \Bigr\rangle &= \frac{1}{4} \sum_{j=1}^{N_t} \int_{\mathbb{R}}\int_0^h \phi'\big(\frac{\tilde{F}_{\mu_{j},r}(z)}{2}\big)\frac{1}{\big[D_z\tilde{F}_{\mu_{j},r}^{-1}\left(\tilde{F}_{\mu_{j},r}(z)\right)\big]^2}\mathrm{d}r\mathrm{d}z \\
		&=\frac{1}{4} \sum_{j=1}^{N_t}\int_0^h\int_0^1 \phi'\big(\frac{y}{2}\big)\frac{1}{D_y\tilde{F}_{\mu_{j},r}^{-1}(y)}\mathrm{d}y\mathrm{d}r,
	\end{align*}
where we used Fubini's lemma and the change of variable $y = \tilde{F}_{\mu_{j},t}(z)$ to derive the last line above.

We now make use of 
\eqref{3.8}. 
Expanding the above expression together with another change of variable, we obtain 
\begin{align*}
\Bigl\langle \varphi, \sum_{j=1}^{N_t}(F_{\mu_j\star \Gamma_h}^{-1}-F_{\mu_j}^{-1}) \Bigr\rangle=\frac{1}{2} \sum_{j=1}^{N_t}\int_0^h\Bigg[ \int_{-1/2}^0 \frac{-\phi'(-x)}{-2D_x\tilde{F}_{\mu_{j},r}^{-1}(-2x)}\mathrm{d}x + \int_0^{1/2} \frac{\phi'(x)}{2D_x\tilde{F}_{\mu_{j},r}^{-1}(2x)}\mathrm{d}x\Bigg]\mathrm{d}r. 
\end{align*}
And then, by \eqref{3.8}, we get the equality in \eqref{2.39}. The inequality
in \eqref{2.39}
(i.e., the fact that the term in the middle is non-negative) 
follows from Lemma \ref{lemm1.11}.
\end{proof}

We deduce the following corollary: 

\begin{cor}
\label{cor:E_eps}
Let $h \in {T/{\mathbb N}^*}$
and $\varphi \in {\mathcal C}^1({\mathbb S}) \cap {\mathcal U}^2({\mathbb S})$ such that $\varphi'({\mathbf x})=0$ if,  for some $\varsigma \in (0,1/4)$, ${\mathbf x} 
\in \cup_{k \in {\mathbb Z}/2}
(k-\varsigma,k+\varsigma)$. For a given $\varepsilon >0$, let 
\begin{equation}\label{eq:varepsilon}
   E_\varepsilon(\mathbf{x}) := {\varepsilon}\begin{cases} 1& \,\,{x}\in [0,1/2)+{\mathbb Z},
   \\
   -1& \,\,  {x}\in [-1/2,0)+{\mathbb Z},
    \end{cases}
\end{equation} 
Then, for $g$, a smooth symmetric density on ${\mathbb R}$ with support included in $(-\varsigma,\varsigma)$, and for all $0 \leq s \leq t$,

\begin{equation}\label{eq1.53}
	\langle \varphi\star g, \mathscr{J}_{t}(X^{(h)}) -\mathscr{J}_{s}(X^{(h)}) \rangle\geq\frac{1}{2}\sum_{j=N_s+1}^{N_t} \int_0^h\int_{\mathbb{S}} \phi'(\mathbf{x})\frac{1}{E_\varepsilon(\mathbf{x}) +[g\star D_{\mathbf{x}}F_{\mu_{j},r}^{-1}](\mathbf{x})}\mathrm{d}\mathbf{x}\mathrm{d}r.
\end{equation}
\end{cor}

\begin{proof}\label{remk2.10}
{By Proposition 
\ref{prop:5.2:06-07:1},} 
we first observe that, for all ${\mathbf x},{\mathbf y} \in {\mathbb S}$ {and $n \in \{1,\cdots,N\}$},
$D_{\mathbf{x}}F_{\mu_{n},t}^{-1}(\mathbf{x})$ and $\varphi'(\mathbf{y})$
have the same sign if ${\mathbf x}$ and ${\mathbf y}$ both belong to $(0,1/2)+{\mathbb Z}$,
or if they both belong to $(-1/2,0)+{\mathbb Z}$.
Moreover, we notice that for ${\mathbf x} \in {\mathbb S}$ and $x' \in {\mathbb R}$
such that $g(x') \neq 0$ and $\varphi'({\mathbf x}-x') \neq 0$,
we necessarily have $x' \in (-\varsigma,\varsigma)$, and one of the following two assertions holds:
either
${\mathbf x}-x' \in (\varsigma,1/2-\varsigma)+{\mathbb Z}$,
or
${\mathbf x}-x' \in (-1/2+\varsigma,-\varsigma)+{\mathbb Z}$.
In particular, either ${\mathbf x}-x'$ and ${\mathbf x}$ both belong to $(0,1/2)+{\mathbb Z}$,
or they both belong to $(-1/2,0)+{\mathbb Z}$.
In any case,
$\varphi'({\mathbf x}-x')$ and $D_{\mathbf{x}}F_{\mu_{j},r}^{-1}(\mathbf{x})$
have the same sign. And then, 
$(\varphi \star g)'({\mathbf x})$ and 
$DF_{\mu_{j},r}^{-1}(\mathbf{x})$ have the same sign. 
Similarly, 
$\varphi'({\mathbf x})$ and $D_{\mathbf{x}}F_{\mu_{j},r}^{-1}(\mathbf{x}-x')$
have the same sign, 
and 
$\varphi'({\mathbf x})$ and 
$(D_{\mathbf{x}}F_{\mu_{j},r}^{-1}\star g)(\mathbf{x})$ have the same sign.

As 
{a consequence of the paragraph above},  $\varphi\star g$ belongs to $\mathcal{U}^2(\mathbb{S})$,
which makes it possible to apply 
\eqref{2.39}
with  
$\varphi\star g$ substituted for 
$\varphi$. We then observe that, for $j \in \{1,\cdots,N\}$,
\begin{equation*}
\begin{split}
\int_0^h \int_{{\mathbb S}}
(\varphi \star g)'({\mathbf x}) \frac1{ D_{\mathbf{x}} F_{\mu_j,r}^{-1}({\mathbf x})} {\mathrm d}  {\mathbf x} {\mathrm d}r
= \int_0^h \left[ \int_{{\mathbb S}}
\int_{{\mathbb R}}
\varphi'({\mathbf x}-x') g(x') \frac1{ D_{\mathbf{x}} F_{\mu_j,r}^{-1}({\mathbf x})} {\mathrm d}  {\mathbf x}
{\mathrm d}  x' \right] {\mathrm d}r.
\end{split}
\end{equation*} 
By applying Fubini's theorem, performing a change of variable and then using the symmetry of $g$, we get 
\begin{equation*}
\begin{split}
&\int_0^h \int_{{\mathbb S}}
(\varphi \star g)'({\mathbf x}) \frac1{ D_{\mathbf{x}} F_{\mu_j,r}^{-1}({\mathbf x})} {\mathrm d}  {\mathbf x} {\mathrm d}r
= 
\int_0^h \int_{{\mathbb R}} 
g(x') 
\left[ 
\int_{{\mathbb S}}
\varphi'({\mathbf x}-x')   \frac1{ D_{\mathbf{x}} F_{\mu_j,r}^{-1}({\mathbf x})} {\mathrm d}  {\mathbf x}
 \right] {\mathrm d}  x' {\mathrm d}r
 \\
 &= \int_0^h \int_{{\mathbb R}} 
g(x') 
\left[ 
\int_{{\mathbb S}}
\varphi'({\mathbf x})   \frac1{ D_{\mathbf{x}} F_{\mu_j,r}^{-1}({\mathbf x}+x')} {\mathrm d}  {\mathbf x}
 \right] {\mathrm d}  x' {\mathrm d}r
 =\int_0^h \int_{{\mathbb S}} 
\left[ 
\int_{{\mathbb R}}
g(x') 
  \varphi'({\mathbf x}) \frac1{ D_{\mathbf{x}} F_{\mu_j,r}^{-1}({\mathbf x}-x')} {\mathrm d}  x'
 \right] {\mathrm d} {\mathbf x} {\mathrm d}r.
\end{split}
\end{equation*} 
By the preliminary observation, 
$ \varphi'({\mathbf x})$
and $D_{\mathbf{x}} F_{\mu_j,r}^{-1}({\mathbf x}-x')$ have the same sign. We get 
\begin{equation*}
\begin{split}
\int_0^h \int_{{\mathbb S}}
(\varphi \star g)'({\mathbf x}) \frac1{ D_{\mathbf{x}} F_{\mu_j,r}^{-1}({\mathbf x})} {\mathrm d}  {\mathbf x} {\mathrm d}r
&=\int_0^h \int_{{\mathbb S}} 
\left[
\int_{{\mathbb R}}
g(x') 
\vert \varphi'({\mathbf x}) 
\vert 
  \frac1{ \vert D_{\mathbf{x}} F_{\mu_j,r}^{-1}({\mathbf x}-x') \vert} {\mathrm d}  x'
 \right] {\mathrm d} {\mathbf x} {\mathrm d}r
 \\
 &\geq \int_0^h 
\int_{\mathbb S}
\frac{\vert \varphi'({\mathbf x}) \vert}{\int_{\mathbb R} g(x') 
\vert D_{\mathbf{x}}F_{\mu_{j},r}^{-1}(\mathbf{x}-x')\vert \mathrm{d}x'}
\mathrm{d}\mathbf{x}\mathrm{d}r,
\end{split}
\end{equation*}
 we applied Jensen's inequality in the probability space $(\mathbb{R},\mathcal{B}(\mathbb{R}),g\mathrm{d}x)$ to obtain the last step.

By the preliminary step again, we can remove the absolute values 
in the last term on the last display. We obtain 
\begin{equation*}
\begin{split}
\int_0^h \int_{{\mathbb S}}
(\varphi \star g)'({\mathbf x}) \frac1{ D_{\mathbf{x}} F_{\mu_j,r}^{-1}({\mathbf x})} {\mathrm d}  {\mathbf x} {\mathrm d}r
&\geq \int_0^h 
\int_{\mathbb S}
\frac{ \varphi'({\mathbf x})  }{\int_{\mathbb R} g(x') 
 D_{\mathbf{x}}F_{\mu_{j},r}^{-1}(\mathbf{x}-x') \mathrm{d}x'}
\mathrm{d}\mathbf{x}\mathrm{d}r
\\
&= \int_0^h 
\int_{\mathbb S}
\frac{ \varphi'({\mathbf x})  }{[ g \star  
 D_{\mathbf{x}}F_{\mu_{j},r}^{-1}](\mathbf{x})}
\mathrm{d}\mathbf{x}\mathrm{d}r.
\end{split}
\end{equation*}
Since $\varphi'({\mathbf x}) $, $[ g \star  
 D_{\mathbf{x}}F_{\mu_{j},r}^{-1}](\mathbf{x})$ and $E_{\varepsilon}({\mathbf x})$ (as defined in the statement) have the same sign, we deduce
\begin{equation*}
\begin{split}
\int_0^h \int_{{\mathbb S}}
(\varphi \star g)'({\mathbf x}) \frac1{ D_{\mathbf{x}} F_{\mu_j,r}^{-1}({\mathbf x})} {\mathrm d}  {\mathbf x} {\mathrm d}r
&\geq \int_0^h 
\int_{\mathbb S}
\frac{ \varphi'({\mathbf x})  }{E_{\varepsilon}({\mathbf x}) + [ g \star  
 D_{\mathbf{x}}F_{\mu_{j},r}^{-1}](\mathbf{x})}
\mathrm{d}\mathbf{x}\mathrm{d}r.
\end{split}
\end{equation*}
By \eqref{2.39}, we obtain
	\begin{equation}
    \begin{split}
		\langle \varphi\star g, \mathscr{J}_{t}(X^{(h)}) -\mathscr{J}_{s}(X^{(h)}) \rangle \geq\frac{1}{2}\sum_{j=N_s+1}^{N_t} \int_0^h\int_{\mathbb{S}} \phi'(\mathbf{x})\frac{1}{E_\varepsilon(\mathbf{x})+ [g \star D_{\mathbf{x}}F_{\mu_{j},r}^{-1}]({\mathbf x})}\mathrm{d}\mathbf{x}\mathrm{d}r.
        \end{split}
	\end{equation}
This completes the proof.
\end{proof}
Letting $h$ tend to $0$, we derive the following statement:
\begin{prop}\label{lemm2.12} For $\delta >0$, 
let $(X_t, \mathscr{J}_t)_{0 \le t \le T}$ be a weak limit in $\mathcal{C}([0,T], \mathcal{H}^{-1}_{\mathrm{Sym}}(\mathbb{S})\times \mathcal{H}^{-(3+\delta)}_{\mathrm{Sym}}(\mathbb{S}))$  of the couple of processes $\{(X_t^{(h)},\mathscr{J}_t(X^{(h)}))_{0 \le t \le T}\}_{ h \in T/{\mathbb N}^*}$ when $h\searrow 0$. 
Let $\phi\in \mathcal{H}^{(3+\delta)}_{\mathrm{Sym}}(\mathbb{S})\cap \mathcal{U}^2(\mathbb{S})$
and $\varepsilon >0$. 
Then, 
with probability one, 
for all $0\leq s<t\leq T$, 
\begin{align}\label{57}
	\langle \phi, \mathscr{J}_{t}- \mathscr{J}_{s}\rangle \geq  \frac{1}{2}\int_s^t\int_{\mathbb{S}}\frac{\phi'(\mathbf{x})}{E_{\varepsilon}(\mathbf{x})+ D_{\mathbf{x}}{X}_{r}(\mathbf{x})}\mathrm{d}\mathbf{x}\mathrm{d}r,
\end{align}
where $E_\varepsilon(\mathbf{x})$ is given by \eqref{eq:varepsilon}.
\end{prop}

Following Theorem \ref{remark 4.1}, the quantifiers ``with probability 1'' and ``for any 
$\varphi$'' appearing in the statement can be exchanged.
In the proof of the above proposition, we shall use the following convergence result.
\begin{lemm}\label{lemm2.8} For fixed {$h \in T/{\mathbb N}^*$} and 
$\varepsilon >0$, let $E_\varepsilon$ be given by \eqref{eq:varepsilon}.
 Then,  for any $\delta >0$, 
\begin{align*}
	\lim_{h \searrow 0} \mathbb{P}\left\{\sup_{j\in\{1,\cdots, N\}}\sup_{\mathbf{x}\in \mathbb{S}} \sup_{0\leq t\leq h}\bigg| \frac{1}{E_\varepsilon(\mathbf{x})+[g\star D_{\mathbf{x}}F_{\mu_{j},t}^{-1}](\mathbf{x})} - \frac{1}{E_\varepsilon(\mathbf{x})+[g\star D_{\mathbf{x}}F_{\mu_{j,h}}^{-1}](\mathbf{x})} \bigg|\geq \delta\right\}  =0, 
\end{align*}
where $\mu_{j,h}$ is the value of $\mu_{j,t}$ at $t=h$ and $\mu_{j,t}=\mathrm{Leb}_{\mathbb{S}}\circ (X_{t_{j}^-}^{(h)})^{-1}\star \Gamma_t$.
\end{lemm}
\begin{proof}[Proof of Lemma \ref{lemm2.8}.] 
Using first the triangle inequality, the {smoothness of the two functions $x \in (0,+ \infty) \mapsto \frac{1}{(\varepsilon+x)}$ and 
$x \in (- \infty,0) \mapsto \frac{1}{(-\varepsilon+x)}$} and then the fact
that 
$E_{\varepsilon}({\mathbf x})$ and $[g\star D_{\mathbf{x}} F_{\mu_j,t}^{-1}]({\mathbf x})$ have the same sign
(see the preliminary step in the proof of Corollary \ref{cor:E_eps}) , we obtain the existence of a constant  $C_{\varepsilon}>0$, depending on $\varepsilon$, such that, for any 
${\mathbf x} \in {\mathbb S}$,
\begin{align*}
	\bigg| \frac{1}{E_\varepsilon(\mathbf{x})+[g\star D_{\mathbf{x}}F_{\mu_{j},t}^{-1}](\mathbf{x})} - \frac{1}{E_\varepsilon(\mathbf{x})+[g\star D_{\mathbf{x}}F_{\mu_{j,h}}^{-1}](\mathbf{x})} \bigg|
	&\leq C_{\varepsilon} \Big\vert 
[g\star D_{\mathbf{x}}F_{\mu_{j,h}}^{-1}](\mathbf{x})
-
[g\star D_{\mathbf{x}}F_{\mu_{j,t}}^{-1}](\mathbf{x})
    \Big\vert
\\  	
	&\leq C_{\varepsilon} \int_{\mathbb{R}} |F_{\mu_j,t}^{-1}(\mathbf{x}-x')-F_{\mu_j,h}^{-1}(\mathbf{x}-x')| \vert g'(x')\vert \mathrm{d}x'.
\end{align*}
Using H\"older's inequality, 
and following the proof of the bound \eqref{eq:1.16:bis}, 
we obtain:
\begin{align*}
	&\bigg| \frac{1}{E_\varepsilon(\mathbf{x})+[g\star D_{\mathbf{x}}F_{\mu_{j},t}^{-1}](\mathbf{x})} - \frac{1}{E_\varepsilon(\mathbf{x})+[g\star D_{\mathbf{x}}F_{\mu_{j,h}}^{-1}](\mathbf{x})} \bigg|\\ 
	&\leq C_{\varepsilon}\|g'\|_{2}  \mathscr{W}_2\Big({\rm Leb}_{\mathbb{S}}\circ(X_{t_j^{-}}^{(h)})^{-1}\star\Gamma_t,{\rm Leb}_{\mathbb{S}}\circ(X_{t_j^{-}}^{(h)})^{-1}\star\Gamma_h\Big)
    \\
    &\leq C_{\varepsilon} \|g'\|_{2} \sqrt{h}.
\end{align*}
By Markov inequality, we get, for any ${\delta}>0$
\begin{align*}
	\mathbb{P}\left\{\sup_{j\in\{1,\cdots, N\}}\sup_{\mathbf{x}\in \mathbb{S}} \sup_{0\leq t\leq h}\bigg| \frac{1}{E_\varepsilon(\mathbf{x})+[g\star D_{\mathbf{x}}F_{\mu_{j},t}^{-1}](\mathbf{x})} - \frac{1}{E_\varepsilon(\mathbf{x})+[g\star D_{\mathbf{x}}F_{\mu_{j,h}}^{-1}](\mathbf{x})} \bigg| \geq {\delta}\right\}\leq \frac{C\|g'\|_{2}^2}{{\delta}^2} h.
\end{align*}
Letting $h\searrow 0$, we conclude the proof.
\end{proof}


We are now in a position to establish Proposition \ref{lemm2.12}.
\begin{proof}[Proof of Proposition \ref{lemm2.12}]
In the first four steps of the proof, we consider $\phi\in \mathcal{H}^{(3+\delta)}_{\mathrm{Sym}}(\mathbb{S})\cap \mathcal{U}^2(\mathbb{S})$ such that $\varphi'({\mathbf x})=0$ if,  for some $\varsigma \in (0,1/4)$, ${\mathbf x} 
\in \cup_{k \in {\mathbb Z}/2}
(k-\varsigma,k+\varsigma)$, and let $\varepsilon >0$.
\color{black}
\vskip 4pt

\textit{First Step.}
Consider $g$, a smooth symmetric density on ${\mathbb R}$ with support included in $(-\varsigma,\varsigma)$.
Then, from \eqref{eq1.53}
and from the fact that 
$F_{\mu_j,h}^{-1}({\mathbf x}) = 
X_{t_j}^{(h)}({\mathbf x})$, we know that the following holds for all $(s,t) \in [0,T]^2$:
\begin{equation}\label{245}
	\langle \varphi\star g, \mathscr{J}_{t}(X^{(h)}) -\mathscr{J}_{s}(X^{(h)})\rangle -\frac{1}{2}\int_s^t \int_{\mathbb{S}}  \frac{\varphi'(\mathbf{x})}{E_\varepsilon(\mathbf{x})+ D_{\mathbf{x}}(X_{r}^{(h)}\star g)(\mathbf{x})}  \mathrm{d}\mathbf{x} \mathrm{d}r \geq {\mathcal T}^{(h)}(s,t) + \mathcal{R}^{(h)}(s,t),
\end{equation}
where:
\begin{equation*}
\begin{split}
\mathcal{R}^{(h)}(s,t)&= \frac{1}{2}\sum_{j=N_s+1}^{N_t}\int_0^h\int_{\mathbb{S}} 
\biggl[ \frac{\phi'(\mathbf{x})}{E_\varepsilon(\mathbf{x})+ D_{\mathbf{x}}(F^{-1}_{\mu_j,r}\star g)(\mathbf{x})}- \frac{\phi'(\mathbf{x})}{E_\varepsilon(\mathbf{x})+ D_{\mathbf{x}}(F^{-1}_{\mu_j,h}\star g)(\mathbf{x})}\biggr] \mathrm{d}\mathbf{x}\mathrm{d}r,\\
\hspace{15pt}
{\mathcal T}^{(h)}(s,t) &= \frac{1}{2}\sum_{j=N_s+1}^{N_t}  h\int_{\mathbb{S}} 
 \frac{\phi'(\mathbf{x})}{E_\varepsilon(\mathbf{x})+ D_{\mathbf{x}}(X_{t_{j}}^{(h)}\star g)(\mathbf{x})}\mathrm{d}\mathbf{x} -\frac{1}{2}\int_s^t \int_{\mathbb{S}}  \frac{\varphi'(\mathbf{x})}{E_\varepsilon(\mathbf{x})+ D_{\mathbf{x}}(X_r^{(h)}\star g)(\mathbf{x})}
 \mathrm{d}\mathbf{x} \mathrm{d}r.
\end{split}
\end{equation*}

\textit{Second Step.}
We first study the two terms on the right-hand side of \eqref{245}. 
Since $\phi \in \mathcal{H}^{3+\delta}(\mathbb{S})$, we deduce that
\begin{equation*}
    |\mathcal{R}^{(h)}(s,t)| \leq  \frac{\|\phi'\|_{\infty}}{2}T\sup_{j\in\{1,\cdots,N\}}\sup_{\mathbf{x} \in \mathbb{S}}\sup_{r\in [0,h]} \bigg| \frac{1}{E_\varepsilon(\mathbf{x})+ D_{\mathbf{x}}(F^{-1}_{\mu_j,r}\star g)(\mathbf{x})}- \frac{1}{E_\varepsilon(\mathbf{x})+ D_{\mathbf{x}}(F^{-1}_{\mu_j,h}\star g)(\mathbf{x})} \bigg|.
\end{equation*}
Hence, from Lemma \ref{lemm2.8} we have, for any ${\delta}>0$, 
\begin{equation}\label{eq:claim}
  \lim_{h\searrow 0}  \mathbb{P}\Big(\sup_{0\leq s< t\leq T}|\mathcal{R}^{(h)}(s,t)| > {\delta} \Big)=0 .
\end{equation}
We now claim that
\begin{align}\label{246}
	\forall \delta >0, \quad \lim_{h\searrow 0 }\mathbb{P}\Big(\sup_{0\leq s< t\leq T}|{\mathcal T}^{(h)}(s,t)| > {\delta} \Big) =0.
\end{align}
In fact, the triangle inequality leads to 
\begin{equation}\label{eq.4.11}
\begin{split}
&|{\mathcal T}^{(h)}(s,t)| \\
&\leq \frac{1}{2}\bigg|\sum_{j=N_s+1}^{N_t}\bigg(h\int_{\mathbb{S}} \frac{\phi'(\mathbf{x})}{E_\varepsilon(\mathbf{x})+ D_{\mathbf{x}}(X_{t_{j}}^{(h)}\star g)(\mathbf{x})}\mathrm{d}\mathbf{x} -\int_{t_{j-1}}^{t_j} \int_{\mathbb{S}}  \frac{\varphi'(\mathbf{x})}{E_\varepsilon(\mathbf{x})+ D_{\mathbf{x}}(X_r^{(h)}\star g)(\mathbf{x})}  \mathrm{d}\mathbf{x} \mathrm{d}r\bigg)\bigg|
\\
&\hspace{5pt} + \frac{1}{2}\int_{t_{N_t}}^t\int_{\mathbb{S}} \bigg| \frac{\varphi'(\mathbf{x})}{E_\varepsilon(\mathbf{x})+ D_{\mathbf{x}}(X_r^{(h)}\star g)(\mathbf{x})}\bigg| \mathrm{d}\mathbf{x} \mathrm{d}r + \frac{1}{2}\int_{t_{N_s}}^s\int_{\mathbb{S}} \bigg| \frac{\varphi'(\mathbf{x})}{E_\varepsilon(\mathbf{x})+ D_{\mathbf{x}}(X_r^{(h)}\star g)(\mathbf{x})}\bigg| \mathrm{d}\mathbf{x} \mathrm{d}r.
\end{split}
\end{equation}
Following the proof of Lemma \ref{lemm2.8}, there exists a constant $C_{\varepsilon}$, depending on $\varepsilon$, such that 
the first term on the right hand side of the above inequality, denoted by ${\mathcal T}_1^{(h)}(s,t)$, satisfies
\begin{align*}
    |{\mathcal T}_1^{(h)}(s,t)| 
    &\leq  C_\varepsilon{\|\phi'\|_{\infty}}\sum_{j=1}^{N_t}\int_{t_{j-1}}^{t_j}\int_{\mathbb{S}} \big|[(X_{t_{j}}^{(h)}-{X}_{r}^{(h)})\star D_{\mathbf{x}}g](\mathbf{x})\big|\mathrm{d}\mathbf{x}\mathrm{d}r.
\end{align*}
Applying Holder's inequality together with the bound \eqref{216} we obtain
\begin{equation*}
\begin{split}
    \mathbb{E} \Big[\sup_{0\leq s< t\leq T}|{\mathcal T}_1^{(h)}(s,t)| \Big]
    &\leq C\|g\|_{2,2}{\|\phi'\|_{\infty}}\sum_{j=1}^{N_t} \int_{t_{j-1}}^{t_j} \mathbb{E}\bigl[ \|X_r^{(h)}-X_{t_{j}}^{(h)}\|_{2,-1} \bigr] \mathrm{d}r\\
    \hspace{15pt}
    &\leq {C}{\|\phi'\|_{\infty}}\|g\|_{2,2} T\left(\sqrt{h} +\mathcal{O}_h(1)\right),
    \end{split}
\end{equation*}
where the Landau symbol $\mathcal{O}_h(1)$ is independent of the parameters $s,t$
and tends to 0 with $h$.

In order to handle the last two terms on the right hand side of \eqref{eq.4.11}, denoted respectively, by ${\mathcal T}_2^{(h)}(t)$ and ${\mathcal T}_3^{(h)}(s)$, 
we use the preliminary step established in the proof {of Corollary \ref{cor:E_eps}}. 
It says that 
$\varphi'({\mathbf x})$ and 
$E_{\varepsilon}({\mathbf x}) + D_{\mathbf{x}}(X_r^{(h)} \star g)({\mathbf x})$ have the same sign, from which we deduce that 
\begin{equation*}
\biggl\vert 
\frac{\varphi'({\mathbf x})}{E_{\varepsilon}({\mathbf x})+ D_{\mathbf{x}} (X_r^{(h)} \star g)({\mathbf x})}
\biggr\vert \leq \frac{\| \varphi' \|_\infty}{\varepsilon}. 
\end{equation*}
And then, 
\begin{equation*}
\begin{split}
  {\mathcal T}_2^{(h)}(t) :=      \frac{1}{2}\int_{t_{N_t}}^t\int_{\mathbb{S}} \bigg| \frac{\varphi'(\mathbf{x})}{E_\varepsilon(\mathbf{x})+ D_{\mathbf{x}}(X_r^{(h)}\star g)(\mathbf{x})}\bigg| \mathrm{d}\mathbf{x} \mathrm{d}r
   \leq 
\frac{\| \varphi' \|_\infty}{2 \varepsilon}h,   
   \end{split}
\end{equation*}
and similarly 
for 
$  {\mathcal T}_2^{(h)}(s)$. 
By combining the above displays, the {claim} \eqref{246} is then proved. 
\vskip 4pt

\textit{Third Step.}
We now address the second term on the left-hand side of \eqref{245}. 
Furthermore, let space $\mathcal{U}^2(\mathbb{S})$ be equipped with the topology induced by the norm in space $\mathcal{H}^{-1}_{\mathrm{Sym}}(\mathbb{S})$. Then, the function defined by: $$\mathbb{D}\left([0,T], \mathcal{U}^2(\mathbb{S})\right)\ni (\mathcal{Y}_r)_{0\leq r\leq t}\mapsto \Biggl(\int_0^t \int_{\mathbb{S}} \phi'(\mathbf{x})\frac{1}{E_\varepsilon(\mathbf{x})+(\mathcal{Y}_r\star D_{\mathbf{x}}g)(\mathbf{x})}\mathrm{d}\mathbf{x}\mathrm{d}r\Biggr) \in \mathcal{C}([0,T],\mathbb{R})$$ 
is continuous with respect to this topology. Consider indeed a sequence $(\mathcal{Y}^{(n)})_{n \in {\mathbb N}}$ of elements of $\mathbb{D}\left([0,T], \mathcal{U}^2(\mathbb{S})\right)$ that converges to $\mathcal{Y}$ in $\mathbb{D}\left([0,T], \mathcal{U}^2(\mathbb{S})\right)$, with $\mathcal{U}^2(\mathbb{S})$ equipped with the $\mathcal{H}^{-1}_{\mathrm{Sym}}(\mathbb{S})$ topology. Then, proceeding as in the derivation of the preceding estimates and as in the proof of Lemma \ref{lemm2.8},  one readily checks that 
\begin{align*}
    &\sup_{0\leq t\leq T}\left| \int_0^t \int_{\mathbb{S}} \phi'(\mathbf{x})\frac{1}{E_\varepsilon(\mathbf{x})+(\mathcal{Y}_r^{(n)}\star D_{\mathbf{x}}g)(\mathbf{x})}\mathrm{d}\mathbf{x}\mathrm{d}r- \int_0^t \int_{\mathbb{S}} \phi'(\mathbf{x})\frac{1}{E_\varepsilon(\mathbf{x})+(\mathcal{Y}_r\star D_{\mathbf{x}}g)(\mathbf{x})}\mathrm{d}\mathbf{x}\mathrm{d}r \right|\\
    &\leq C \int_0^T \left(\|\mathcal{Y}_r^{(n)}-\mathcal{Y}_r\|_{2,-1}\wedge 1 \right)\mathrm{d}r,
\end{align*}
{with $C$ depending on $\varphi$ and 
$\varepsilon$.}
Then, the claim follows from the dominated convergence theorem and from the fact that, 
for almost every 
$r \in [0,T]$, 
$\lim_{n \rightarrow \infty} \|\mathcal{Y}_r^{(n)}-\mathcal{Y}_r\|_{2,-1}
= 0$ {(see (12.14) in \cite{Billingsley} with an obvious adaptation to the Hilbert valued setting, and use the fact that discontinuity points are at most countable)}. Therefore, by
combining the latter claim and the tightness of the family $\{(X_t^{(h)})_{0 \le t \le T}\}_{{h \in T/{\mathbb N}^*}}$ in $\mathbb{D}([0,T],\mathcal{H}^{-1}_{\mathrm{Sym}}(\mathbb{S}))$
together with the fact that 
any weak limit takes values in 
${\mathcal U}^2({\mathbb S})$ (see Proposition 
\ref{remark 4.1}), we deduce that, along any subsequence (still indexed by 
$h$) 
$\{(X_t^{(h)},{\mathscr J}_t(X^{(h)})_{0 \le t \le T}\}_{{h \in T/{\mathbb N}^*}}$  
weakly converges
in $\mathbb{D}([0,T], \mathcal{H}^{-1}_{\mathrm{Sym}}(\mathbb{S}))
\times 
{\mathcal H}^{-(3+\delta)}_{\rm Sym}(\mathbb{S}))$ to some  $(X_t, {\mathscr J}_t)_{0 \le t \le T}$, it also holds 
\begin{equation}\label{340}
\begin{split}
&\biggl( 
\Bigl( X_t^{(h)}, {\mathscr J}_t(X^{(h)})\Bigr)_{0 \le t \le T}, 
\Bigl(
    \frac{1}{2}\int_0^t\int_{\mathbb{S}}\frac{\phi'(\mathbf{x})}{E_\varepsilon(\mathbf{x})+ ({X}_{r}^{(h)}\star D_{\mathbf{x}}g)(\mathbf{x})}\mathrm{d}\mathbf{x}\mathrm{d}r \Bigr)_{0\le t \le T}
    \biggr) 
    \\
&\hspace{15pt}    \Rightarrow 
    \biggl(
    \Bigl( X_t, {{\mathscr J_t}}\Bigr)_{0 \le t \le T}, \Bigl( \frac{1}{2}\int_0^t\int_{\mathbb{S}}\frac{\phi'(\mathbf{x})}{E_\varepsilon(\mathbf{x})+ ({X}_{r}\star D_{\mathbf{x}}g)(\mathbf{x})}\mathrm{d}\mathbf{x}\mathrm{d}r \Bigr)_{0 \le t \le T} \biggr),
    \end{split}
\end{equation}
with 
$\Rightarrow$ here denoting convergence in law in 
$\mathbb{D}([0,T], \mathcal{H}^{-1}_{\mathrm{Sym}}(\mathbb{S}))
\times 
{\mathcal H}^{-(3+\delta)}_{\rm Sym}(\mathbb{S}))
\times {\mathcal  C}([0,T],{\mathbb R})
$ (with the last factor denoting the space of continuous functions from $[0,T]$ to ${\mathbb R}$). 
\vskip 4pt

\textit{Fourth Step.} We let $h$ go to $0$ in \eqref{245}. Following the second step, we let $(X_t,{\mathscr{J}_{t}} )_{0\le t \le T}$ be a weak limit of the {pairs of} c\`adl\`ag processes $\{(X_t^{(h)},\mathscr{J}_{t}(X^{(h)}))_{0 \le t \le T}\}_{{h \in T/{\mathbb N}^*}}$; by 
Proposition \ref{remark 4.1}, 
it takes values 
in $\mathcal{C}([0,T], \mathcal{H}^{-1}_{\mathrm{Sym}}(\mathbb{S})\times \mathcal{H}^{-(3+\delta)}_{\mathrm{Sym}}(\mathbb{S})),$ for some $\delta >0$. Combining \eqref{eq:claim}, \eqref{246} and 
\eqref{340} with the porte-manteau theorem, we obtain 
\begin{equation*}
	\forall (s,t) \in [0,T]^2 \ {\rm with} \ s < t, \quad \langle \varphi \star g, \mathscr{J}_{t} -\mathscr{J}_{s}\rangle \geq  \frac{1}{2}\int_s^t \int_{\mathbb{S}}  \frac{\varphi'(\mathbf{x})}{E_\varepsilon(\mathbf{x})+ ({X}_{r}\star D_{\mathbf{x}}g)(\mathbf{x})}\mathrm{d}\mathbf{x} \mathrm{d}r 
\end{equation*}
with probability 1.

By Proposition \ref{remark 4.1}, the weak-limit process $({X}_t)_{0 \le t \le T}$ {takes values in}  $L^2([0,T],\mathcal{H}_{\mathrm{Sym}}^{1}(\mathbb{S}))$. In particular, we can write, 
for almost every 
$r \in [0,T]$, 
\begin{equation*}\int_{\mathbb{S}}\frac{\varphi'(\mathbf{x})}{E_\varepsilon(\mathbf{x})+ ({X}_{r}\star D_{\mathbf{x}}g)(\mathbf{x})}\mathrm{d}\mathbf{x} 
=
\int_{\mathbb S}
\frac{\varphi'({\mathbf x})}{E_{\varepsilon}({\mathbf x}) + 
(D_{\mathbf{x}} X_r \star g)({\mathbf x})}
{\mathrm d} {\mathbf x}.
\end{equation*}
As $g$ tends to the delta mass in $0$ (while maintaining the support 
    of $g$ disjoint from that of $\varphi'$), 
$D_{\mathbf{x}} X_r \star g({\mathbf x})$
tends to $D_{\mathbf{x}}X_r({\mathbf x})$ in Lebesgue measure on 
${\mathbb S}$. 
Meanwhile, the term
$\langle \varphi \star g , {\mathscr{J}_{t} -\mathscr{J}_{s}}\rangle$ 
converges to $\langle \varphi ,{\mathscr{J}_{t} -\mathscr{J}_{s}}\rangle$ as $g$ converges to the delta mass in $0$, uniformly in $(s,t)$. The proof is then complete when 
$\varphi$ satisfies the additional condition that $\varphi'({\mathbf x})=0$ if,  for some $\varsigma \in (0,1/4)$, ${\mathbf x} 
\in \cup_{k \in {\mathbb Z}/2}
(k-\varsigma,k+\varsigma)$, and let $\varepsilon >0$.
\vskip 4pt

\textit{Fifth Step.}
We now consider a generic element 
$\varphi  \in \mathcal{H}^{3+\delta}_{\mathrm{Sym}}(\mathbb{S})\cap \mathcal{U}^2(\mathbb{S})$. 
For a given $\varsigma \in (0,1/4)$, we consider 
a smooth, periodic, even and non-negative function 
$\psi$ such that 
$\psi({\mathbf x})=0$ if 
${\mathbf x} \in \cup_{k \in {\mathbb Z}/2} (k-\varsigma,k+\varsigma)$.
We then let:
\begin{equation*}
\varphi^{\psi}(x) 
:= 
\int_0^x 
\psi(y) {\varphi}^\prime(y) {\mathrm d}y, \quad x \in {\mathbb R}.
\end{equation*}
The function $\varphi^{\psi}$ is obviously periodic 
{(by symmetry arguments, 
$\int_{-1/2}^{1/2} \psi(y) \varphi'(y) \ud y=0$)}
and can be regarded as a function on ${\mathbb S}$. We then have 
$\varphi^{\psi} \in {\mathcal H}^{3+\delta} \cap {\mathcal U}^2({\mathbb S})$, and 
$\varphi^{\psi}$ satisfies the assumption used in the fourth step. 
We then have
\begin{equation*}
\forall s,t \in [0,T]
\quad {\rm with} \ s <t, \quad 
\langle {\mathscr J}_t - {\mathscr J}_s, 
\varphi^{\psi} \rangle 
\geq \frac12 
\int_s^t 
\int_{{\mathbb S}}
\frac{(\varphi^{\psi})'({\mathbf x})}{E_{\varepsilon}({\mathbf x})+D_{\mathbf{x}} X_r({\mathbf x})} {\mathrm d}{\mathbf x} {\mathbf d}r.
\end{equation*}
Assuming that $\psi$ takes values in [0,1], we also have 
$\varphi - \varphi^{\psi} \in {\mathcal H}^{3+\delta}({\mathbb S}) \cap {\mathcal U}^2({\mathbb S})$, from which we get 
\begin{equation*}
\forall s,t \in [0,T] \quad {\rm with} \ s < t, 
\quad 
\langle {\mathscr J}_t - 
{\mathscr J}_s , 
\varphi - \varphi^{\psi}
\rangle 
\geq 0.
\end{equation*}
This gives 
\begin{equation*}
\begin{split}
\forall s,t \in [0,T] \quad {\rm with} \ s < t, 
\quad 
\langle {\mathscr J}_t - 
{\mathscr J}_s , 
\varphi 
\rangle 
\geq 
\langle {\mathscr J}_t - 
{\mathscr J}_s , 
\varphi^\psi \rangle 
&\geq 
\int_s^t 
\int_{{\mathbb S}}
\frac{(\varphi^{\psi})'({\mathbf x})}{E_{\varepsilon}({\mathbf x})+D_{\mathbf{x}} X_r({\mathbf x})} {\mathrm d}{\mathbf x} {\mathbf d}r
\\
&=
\int_s^t 
\int_{{\mathbb S}}
\frac{\psi({\mathbf x}) \varphi'({\mathbf x})}{E_{\varepsilon}({\mathbf x})+D_{\mathbf{x}} X_r({\mathbf x})} {\mathrm d}{\mathbf x} {\mathbf d}r.
\end{split}
\end{equation*}
Letting 
$\psi$ tend pointwise to 1, we complete the proof. 
\color{black}
\end{proof}

As a consequence of Proposition 
\ref{lemm2.12}, we get: 

\begin{cor}\label{cor 3.9}
For $\delta >0$, 
let $(X_t, \mathscr{J}_t)_{0 \le t \le T}$ be a weak limit in $\mathcal{C}([0,T], \mathcal{H}^{-1}_{\mathrm{Sym}}(\mathbb{S})\times \mathcal{H}^{-(3+\delta)}_{\mathrm{Sym}}(\mathbb{S}))$  of the sequence of pair processes $\{(X_t^{(h)},\mathscr{J}_t(X^{(h)}))_{0 \le t \le T}\}_{{h \in T/{\mathbb N}^*}}$ when $h\searrow 0$. 
Then, for any $t\in [0,T]$,
\begin{align*}
\frac{1}{2}\mathbb{E}\int_0^T\int_{\mathbb{S}}\frac{1}{|D_{\mathbf{x}}{X}_{t}(\mathbf{x})|}\mathrm{d}\mathbf{x}\mathrm{d}t \leq C_{(T,\mathbb{E}\|X_0\|_2)},
\end{align*}
where $C_{(T,\mathbb{E}\|X_0\|_2)}$ is a constant only depending on $T,\lambda$ and  
${\mathbb E}[ \| X_0 \|_2]$. 
\end{cor}
\begin{proof} 
By \eqref{57}, we know that, for any 
$\varphi \in 
 \mathcal{H}^{(3+\delta)}_{\mathrm{Sym}}(\mathbb{S})\cap \mathcal{U}^2(\mathbb{S})$,
any $\varepsilon >0$
 and all $0\leq s<t\leq T$, 
\begin{align}
	\langle \phi, \mathscr{J}_{t}- \mathscr{J}_{s}\rangle \geq  \frac{1}{2}\int_s^t\int_{\mathbb{S}}\frac{\phi'(\mathbf{x})}{E_{\varepsilon}(\mathbf{x})+ D_{\mathbf{x}}{X}_{r}(\mathbf{x})}\mathrm{d}\mathbf{x}\mathrm{d}r.
    \label{330:bbb}
\end{align}
We recall from the item 1 of Proposition \ref{remark 4.1} that, almost surely, for almost every $r \in [s,t]$, $X_r \in {\mathcal H}^1({\mathbb S})$. Having said this, we have, almost surely, for almost every $r \in [s,t]$, 
$D_{\mathbf{x}} X_r({\mathbf x}) \geq 0$ for almost every ${\mathbf x} \in (0,1/2)+{\mathbb Z}$, and 
$D X_r({\mathbf x}) \leq 0$ for almost every 
${\mathbf x} \in (-1/2,0)+{\mathbb Z}$. In particular, 
$\varphi'({\mathbf x})$ and 
$E_{\varepsilon}({\mathbf x}) + D_{\mathbf{x}} X_r({\mathbf x})$ have the same sign (for almost every 
${\mathbf x}$). We rewrite \eqref{330:bbb}
as
\begin{align}
	\langle \phi, \mathscr{J}_{t}- \mathscr{J}_{s}\rangle \geq  \frac{1}{2}\int_s^t\int_{\mathbb{S}}\frac{\vert \phi'(\mathbf{x})\vert}{\vert E_{\varepsilon}(\mathbf{x}) \vert + \vert D_{\mathbf{x}}{X}_{r}(\mathbf{x})\vert}\mathrm{d}\mathbf{x}\mathrm{d}r.
    \label{330:bb}
\end{align}
By items 5 and then 4 in the statement of Proposition \ref{remark 4.1}, 
\begin{equation*}
\begin{split}
0 \leq \langle \phi, \mathscr{J}_{t}- \mathscr{J}_{s}\rangle &\leq 
\langle \phi, \zeta_{t}- \zeta_{s}\rangle
\\
&=\langle X_t - X_s, \varphi \rangle 
+ \int_s^t \langle D_{\mathbf{x}} X_r,  \varphi' \rangle 
{\mathrm d}r + 
\langle W_t - W_s, \varphi \rangle.
\end{split}
\end{equation*}
And then, 
inserting the above in \eqref{330:bb}, we deduce that 
\begin{equation*}
\begin{split}
 \frac{1}{2} \int_{0}^{T} \int_{\mathbb{S}}  \frac{\vert \phi'(\mathbf{x})\vert}{\varepsilon+
\vert
D_{\mathbf{x}}X_r(\mathbf{x})\vert} \mathrm{d}\mathbf{x} \mathrm{d}r 
&\leq \|\varphi 
\|_{1,\infty}
\biggl[   \|X_{T}-X_{0}\|_{2} +   \int_0^T\|D_{\mathbf{x}}X_r\|_2\mathrm{d}r + \| W_T - W_0 \|_2
\biggr].
\end{split}
\end{equation*}
Taking expectation and using item 1 in Proposition \ref{remark 4.1}, we deduce that 
there exists a constant $C_T$, independent of 
$\varepsilon$, such that 
\begin{equation}
\label{ineq.4.20} \mathbb{E}\int_{0}^{T} \int_{\mathbb{S}} \frac{\vert \phi'(\mathbf{x}) \vert}{\varepsilon+
\vert D_{\mathbf{x}} X_r({\mathbf x})
\vert} \mathrm{d}\mathbf{x} \mathrm{d}r \leq C_T\|\varphi\|_{1,\infty}\left(1+\mathbb{E}[\|X_0\|_{2}^2]^{1/2} \right).
\end{equation}

Setting now $\varphi^{\epsilon}:= \Xi\star \Gamma_{\epsilon}$, where 
$\Xi$ is the periodic function equal to 
$\Xi(\mathbf{x}) = |\mathbf{x}|$ for $\mathbf{x}\in [-1/2,1/2]$ and $\Gamma_{\epsilon}$ is the Gaussian density with fixed variance $\epsilon >0$ {(which should be distinguished from $\varepsilon$)}. We observe that the bound \eqref{ineq.4.20} still holds for $\varphi^{\epsilon}$ in lieu of $\varphi$, since $\varphi^{\epsilon}$ belongs to $\mathcal{H}^{3+\delta}_{\mathrm{Sym}}(\mathbb{S})\cap \mathcal{U}^2(\mathbb{S})$. More precisely, the inequality \eqref{ineq.4.20} writes: 
\begin{equation*}
 \mathbb{E}\int_{0}^{T} \int_{\mathbb{S}}  \frac{\vert (\varphi^{\epsilon})'(\mathbf{x})\vert}{\varepsilon+ \vert D_{\mathbf{x}} X_r(\mathbf{x})\vert} \mathrm{d}\mathbf{x} \mathrm{d}r \leq C_T \left(1+\mathbb{E}[\|X_0\|_{2}^2]^{1/2}\right)
\end{equation*}
where the constant $C_T$ above is independent of $\epsilon$ and $\varepsilon$. 
Letting 
$\epsilon$ {(first)} and
{then}
$\varepsilon$
tend to $0$, we obtain from Fatou's lemma:
\begin{align}\label{ineq:420}
 \mathbb{E}\int_{0}^{T} \int_{\mathbb{S}}  \frac{1}{\vert D_{\mathbf{x}} X_r({\mathbf x}) \vert} \mathrm{d}\mathbf{x} \mathrm{d}r \leq C_T \left(1+\mathbb{E}[\|X_0\|_{2}^2]^{1/2}\right)
\end{align}
This completes the proof.
\end{proof}

\subsection{%
  \texorpdfstring{%
 Integration with respect to
$({\mathscr J}_t)_{t \geq 0}$, 
$(\eta_t)_{t \geq 0}$
and 
$(\zeta_t)_{t\geq0}$%
  }{%
    Integration with respect to the reflection processes%
  }%
}

\label{subsection 4.2} 

The purpose of this subsection is to provide a clear meaning to 
integrals driven by the processes 
$({\mathscr J}_t)_{0 \le t \le T}$, 
$(\eta_t)_{0 \le t \le T}$
and 
$(\zeta_t)_{0 \le t \le T}$ defined in 
the statement of Proposition \ref{remark 4.1}. 
The construction of the subsequent integrals is directly inspired from 
\cite{DelHam25},  and is carried out
in a pathwise sense (see for instance 
\cite[Remark 4]{DelHam25}). 
Still from Proposition \ref{remark 4.1}, 
we know that 
\begin{align}\label{eq:limit:zeta}
\zeta_t = \mathscr{J}_t  +  \eta_t,\quad t \in [0,T].
\end{align}
We deduce from 
the decomposition \eqref{eq:limit:zeta}
that
it suffices to construct integrals with respect to the trajectories of 
$({\mathscr J}_t)_{0 \le t \le T}$ and 
$(\eta_t)_{0 \le t \le T}$. 
Both satisfy the same properties, which we summarize in the following (pathwise) way. 
{For the same $T>0$ as before, and for}
$\delta >0$, a deterministic path 
$\varsigma : [0,T] \ni t \mapsto \mathcal{H}^{-(3+\delta)}_{\mathrm{Sym}}(\mathbb{S})
$ is said to satisfy {\bf C1}, 
{\bf C2} and {\bf C3} if the following three conditions are satisfied: 
\begin{itemize}
\item[{\fbox{\bf C1}}] The map $[0,T] \ni t \mapsto \varsigma_t$ is a continuous function from $[0,T]$ to $\mathcal{H}^{-(3+\delta)}_{\mathrm{Sym}}(\mathbb{S})$, with respect to the norm $\|\cdot\|_{2,-(3+\delta)}$;
\item[{\fbox{\bf C2}}] For any function $\varphi \in \mathcal{H}^{3+\delta}_{\mathrm{Sym}}(\mathbb{S})\cap \mathcal{U}^2(\mathbb{S})$, the path $[0,T]\ni t \mapsto \langle  \varsigma_t,\varphi\rangle$ is non-decreasing;
\item[{\fbox{\bf C3}}]
There exists a constant $C \geq 0$ such that, for any $\varphi \in \mathcal{H}^{3+\delta}_{\mathrm{Sym}}(\mathbb{S})\cap \mathcal{U}^2(\mathbb{S})$,
\begin{equation*}
\sup_{t \in [0,T]}
\bigl\vert \langle \varsigma_t,\varphi\rangle
\bigr\vert \leq C \| \varphi \|_{2,2}. 
\end{equation*}
\end{itemize}

Once again, the construction of an integral with $\varsigma$ as integrator 
is mainly inspired by
\cite{DelHam25}. 
One main change in the analysis consists in mollifying the decomposition 
\eqref{decomp 33}.
Recall from the latter that
\begin{equation*}
e_k = e_k^+ - e_k^{-}.
\end{equation*}
We deduce that, for every $t>0$,
\begin{equation*}
{\mathrm e}^{t \Delta} e_k = 
{\mathrm e}^{t \Delta}
e_k^+ -
{\mathrm e}^{t \Delta}
e_k^-, \quad t \geq 0. 
\end{equation*} 
The left-hand side is equal to 
${\rm e}^{- {4} \pi^2 k^2 t^2} e_k$.
By Subsection \ref{subse:rearrangement}, the two functions 
appearing in 
the right-hand side 
are still non-decreasing. We thus obtain the following decomposition
for $e_k$ into 
non-decreasing and non-increasing smooth parts:
\begin{equation*}
e_k = {\rm e}^{{4} \pi^2 k^2 t^2} 
\left( 
{{\mathrm e}^{t \Delta}}
e_k^+ - 
{{\mathrm e}^{t \Delta}}
e_k^-  \right). 
\end{equation*}
Choosing $t=1/k$, we get 
\begin{equation*}
e_k =  e_k^{\star,+} -  e_k^{\star,-}, 
\end{equation*}
with 
\begin{equation*}
e_k^{\star,+} = {\rm e}^{{4} \pi^2}
{{\rm e}^{\tfrac1k \Delta}}
e_k^{+}
\quad 
e_k^{\star,-} = {\rm e}^{{4} \pi^2}
{{\rm e}^{\tfrac1k \Delta}} e_k^- .
\end{equation*}
Since $\textrm{\rm e}^{\frac1k \Delta}$ is a contraction in ${\mathcal H}^2({\mathbb S})$, we observe that, up to a new constant 
$C$ therein, 
the bound \eqref{decomp34}
remains true
for $e_k^{\star,\pm}$. 

With these notations, we claim:

\begin{thm}
\label{thm:integral:varsigma}
Let 
$\varsigma=(\varsigma_t)_{t \in [0,T]}$ satisfy \textbf{\bf C1} and \textbf{\bf C2}, and 
$(\varphi : [0,T] \ni t \mapsto \varphi_t)_{t \in [0,T]}$ 
be a continuous curve from $[0,T]$
to  $\mathcal{H}^4_{\mathrm{Sym}}(\mathbb{S})$. Then, there exists an integral curve, denoted 
\begin{equation*}
\biggl( \int_0^t \varphi_s \cdot 
{\mathrm d} \varsigma_s 
\biggr)_{0 \leq t \leq T},
\end{equation*}
satisfying the following properties:
\begin{enumerate}
    \item the integral curve is continuous in $t \in [0,T]$;
    \item 
\begin{equation*}
\lim_{N \rightarrow \infty}
\sup_{t \in [0,T]}
\biggl\vert 
\int_0^t 
\varphi_s  \cdot 
{\mathrm d} \varsigma_s 
- 
\sum_{k=0}^N 
\biggl( 
\int_0^t \langle \varphi_s,e_k\rangle 
{\mathrm d} \langle 
\varsigma_s,e_k^{\star,+}\rangle
- 
\int_0^t \langle \varphi_s,e_k\rangle 
{\mathrm d} \langle 
\varsigma_s,e_k^{\star,-}\rangle
\biggr) 
\biggr\vert
=0,
\end{equation*}
with
the integral curves
$(\int_0^t \langle \varphi_s,e_k\rangle 
{\mathrm d} \langle 
\varsigma_s,e_k^{\star,\pm}\rangle)_{t \in [0,T]}$ being understood as 
Riemann-Stieltjes integrals;
\item there exists a universal constant $c>0$ such that, 
\begin{equation*}
\sup_{t \in [0,T]}
\biggl\vert 
\int_0^t
\varphi_s \cdot {\mathrm d} 
\varsigma_s 
\biggr\vert
\leq c \sup_{t \in [0,T]}
\| \varphi_t \|_{2,4}
\sup_{t \in [0,T]} \sup_{\psi \in {\mathcal H}^{3+\delta}({\mathbb S}) : \| \psi \|_{2,2} \leq 1}
\bigl\vert \langle \varsigma_t, \psi \rangle 
\bigr\vert;
\end{equation*}
\item 
\begin{equation*}
\lim_{N \rightarrow \infty}
\sup_{t \in [0,T]}
\biggl\vert 
\int_0^t 
\varphi_s  \cdot 
{\mathrm d} \varsigma_s 
- 
\sum_{k \in {\mathbb N} : k  \leq N t} 
\langle 
\varphi_{k/N}, \varsigma_{t \wedge (k+1)/N}
- 
\varsigma_{k/N}
\rangle  
\biggr\vert
=0;
\end{equation*}
\item if, in addition, $\varphi$ takes values in 
${\mathcal U}^2({\mathbb S})$, then 
$[0,T] \ni t \mapsto 
\int_0^t \varphi_s \cdot {\mathrm d} \varsigma_s$
is non-decreasing. 
    \end{enumerate}
\end{thm}
In item (2), 
the fact that integral 
$\int_0^t \langle \varphi_s,e_k \rangle {\mathrm d} \langle \varsigma_s,e_k^{\star,\pm}\rangle$
can be defined as a Riemann-Stieltjes integral follows from the fact that 
$e_k^+$ and $e_k^-$
 both belong to $\mathcal{U}^2(\mathbb{S})$, and then 
 from 
 {\bf C2}. 

 \begin{proof}
We just provide
a sketch of the proof as most of the ingredients are similar to 
\cite{DelHam25}. 

Item (1) in the statement is a consequence of item (2). 

As for item (2), we can follow the derivation of \cite[Definition 2]{DelHam25}, which relies itself on 
\cite[Lemma 9]{DelHam25}. The main difference lies in 
\cite[(4.22)]{DelHam25}: 
\begin{enumerate}[(i)]
\item 
in 
\cite{DelHam25}, the authors use the decomposition 
$e_k=e_k^+ - e_k^{-}$; here  
we use  
$e_k = e_k^{\star,+} - e_k^{\star,-}$; 
\item in 
\cite{DelHam25}, the test function $\varphi$ is analytic in space (as it is obtained via convolution by the heat kernel); in comparison, the test function $\varphi$ used in the statement of Theorem \ref{thm:integral:varsigma}
just takes values in the space
$\mathcal{H}^4_{\mathrm{Sym}}(\mathbb{S})$; 
\item in 
\cite{DelHam25},
the integrator $\varsigma$ takes values in 
${\mathcal H}^{-2}_{\mathrm{Sym}}(\mathbb{S})$; here 
the integrator $\varsigma$ takes values in ${\mathcal H}^{-(3+\delta)}_{\mathrm{Sym}}({\mathbb S})$. 
\end{enumerate}
We thus rewrite \cite[(4.22)]{DelHam25} in the following form (using the fact that the integral in the left-hand side is a Riemann-Stieltjes integral): 
\begin{equation*}
\begin{split}
\sup_{t \in [0,T]}
\biggl\vert 
\int_0^t \langle \varphi_s,e_n\rangle 
{\mathrm d} \langle \varsigma_s,e_n^{\star,\pm}
\rangle 
\biggr\vert 
&\leq 
\sup_{t \in [0,T]}
\bigl\vert   \langle \varphi_t,e_n\rangle
\bigr\vert 
\times \langle \varsigma_T,e_n^{\star,\pm}
\rangle 
 \leq 
C (1+n)^{-4} 
\sup_{t \in [0,T]}
\| \varphi_t \|_{2,4}
\|e_n^{\star,\pm} \|_{2,2},
\end{split}
\end{equation*}
where $n$ is a fixed integer in ${\mathbb N}$, and with the last line following from 
\textbf{C3}. 

By \eqref{decomp34} (for $e_n^{\star,\pm}$), we get (for a new value of $C >0$)
\begin{equation*}
\begin{split}
\sup_{t \in [0,T]}
\biggl\vert 
\int_0^t \langle \varphi_s,e_n\rangle 
{\mathrm d} \langle \varsigma_s,e_n^{\star,\pm}
\rangle 
\biggr\vert 
&\leq 
C (1+n)^{-2} 
\sup_{t \in [0,T]}
\| \varphi_t \|_{2,4}.
\end{split}
\end{equation*}
This shows that the series with the left-hand side as generic term is summable over $n \in {\mathbb N}$. This permits to prove item (2) in the statement, following \cite[Definition 2]{DelHam25}.
Item (3) is a direct consequence of the above bound. 

Items (4) and   (5) are consequences of 
item (3), the proof following the arguments used to 
prove \cite[Lemma 10 \& Corollary 2]{DelHam25}. 
\end{proof}

By Proposition \ref{remark 4.1}, we can apply 
Theorem \ref{thm:integral:varsigma} to the trajectories of $(\zeta_t)_{0 \le t \le T}$, 
$(\eta_t)_{{0 \le t \le T}}$ and 
$({\mathscr J}_t)_{{0 \le t \le T}}$. 
For any progressively-measurable process $(z_t)_{{0 \le t \le T}}$ with continuous trajectories with values in ${\mathcal H}^{4}({\mathbb S}) \cap {\mathcal U}^2({\mathbb S})$, we can define 
the progressively-measurable processes 
$(\int_0^t z_r \cdot {\mathrm d} \zeta_r)_{{0 \le t \le T}}$, 
$(\int_0^t z_r \cdot {\mathrm d} \eta_r)_{{0 \le t \le T}}$
and 
$(\int_0^t z_r \cdot {\mathrm d} {\mathscr J}_r)_{{0 \le t \le T}}$. 
It is readily seen that 
\begin{align*}
\int_0^t z_r\cdot\mathrm{d}\zeta_r:=  \int_0^t z_r\cdot\mathrm{d}\mathscr{J}_r + \int_0^t z_r\cdot\mathrm{d}\eta_r. 
\end{align*}
As for the integral 
driven by $({\mathscr J}_t)_{{0 \le t \le T}}$, we have the following lower bound.
\begin{lemm}
Let $(z_t)_{{0 \le t \le T}}$ be a process with continuous trajectories with values in ${\mathcal H}^{4}({\mathbb S}) \cap {\mathcal U}^2({\mathbb S})$. 
Then, 
for all $\varepsilon >0$,
with probability 1, for 
all $s,t\in [0,T]$ with $s<t$,
\begin{align}\label{ineq:4.18}
    \int_s^t z_r\cdot\mathrm{d}\mathscr{J}_r \geq  \frac{1}{2}\int_s^t \int_{\mathbb{S}}\frac{z'_{r}(\mathbf{x})}{E_\varepsilon(\mathbf{x})+D_{\mathbf{x}}X_r(\mathbf{x})}\mathrm{d}\mathbf{x}\mathrm{d}r,
\end{align}
where $E_\varepsilon(\mathbf{x})$ is given by \eqref{eq:varepsilon}.
\end{lemm}
\begin{proof} 
By item (4) in the proof of Theorem \ref{thm:integral:varsigma},
we know that, almost surely, for all $s,t \in [0,T]$ with $s<t$,
\begin{equation*}
\begin{split}
\int_s^t z_r \cdot {\mathrm d} {\mathscr J}_r = \lim_{N \rightarrow \infty}
\biggl[
\sum_{k \in {\mathbb N} : 
N s < k \leq Nt}
\left\langle z_{k/N},{\mathscr J}_{t \wedge (k+1)/N} -  {\mathscr J}_{k/N} \right\rangle 
+ 
\left\langle z_{\lfloor N s \rfloor/N }, {\mathscr J}_{(\lfloor N s \rfloor+1)/N }
-
{\mathscr J}_{s} \right\rangle \biggr].
\end{split}
\end{equation*}
By Lemma \ref{lemm2.12}, we know that, almost surely, for each $k \in {\mathbb N}$ with $k \leq Nt$,
\begin{equation*}
\left\langle z_{k/N},{\mathscr J}_{t \wedge (k+1)/N} -  {\mathscr J}_{k/N} \right\rangle \geq 
\int_{k/N}^{t \wedge (k+1)/N}
\int_{{\mathbb S}}
\frac{z_{k/N}'({\mathbf x})}{E_{\varepsilon}({\mathbf x}) 
+ 
D_{\mathbf{x}} X_r({\mathbf x})
} {\mathrm d}
{\mathbf x} {\mathrm d} r.
\end{equation*}
Similarly, 
\begin{equation*}
\left\langle z_{\lfloor Ns \rfloor /N},
{\mathscr J}_{(\lfloor N s \rfloor+1)/N}
- 
{\mathscr J}_{s} 
\right\rangle \geq 
\int_{s}^{(\lfloor N s \rfloor+1)/N}
\int_{{\mathbb S}}
\frac{z_{\lfloor Ns \rfloor /N}'({\mathbf x})}{E_{\varepsilon}({\mathbf x}) 
+ 
D_{\mathbf{x}} X_r({\mathbf x})
} {\mathrm d}
{\mathbf x} {\mathrm d} r.
\end{equation*}
And then, 
\begin{equation*}
\begin{split}
\sum_{k \in {\mathbb N} : 
N s < k \leq Nt}
\left\langle z_{k/N},{\mathscr J}_{t \wedge (k+1)/N} -  {\mathscr J}_{k/N} \right\rangle 
+ 
\left\langle z_{\lfloor N s \rfloor/N }, {\mathscr J}_{(\lfloor N s \rfloor+1)/N }
-
{\mathscr J}_{s} \right\rangle 
\geq 
\int_s^t 
\int_{{\mathbb S}}
\frac{z_{\lfloor N r \rfloor/N}'({\mathbf x})}{E_{\varepsilon}({\mathbf x}) + D_{\mathbf{x}} X_r({\mathbf x})}
{\mathrm d} {\mathbf x}
{\mathbf d}r.
\end{split}
\end{equation*}
Using the fact that, almost surely, the trajectory
$[0,T] \ni r \mapsto z_r' \in {\mathcal H}^3({\mathbb S})$ is continuous, we get, for all $s,t \in [0,T]$ with $s < t$, 
with probability 1, 
\begin{equation*}
\lim_{N \rightarrow \infty}
\int_s^t 
\int_{{\mathbb S}}
\frac{z_{\lfloor N r \rfloor/N}'({\mathbf x})}{E_{\varepsilon}({\mathbf x}) + D_{\mathbf{x}} X_r({\mathbf x})}
{\mathrm d} {\mathbf x}
{\mathbf d}r
= 
\int_s^t 
\int_{{\mathbb S}}
\frac{z_{r}'({\mathbf x})}{E_{\varepsilon}({\mathbf x}) + D_{\mathbf{x}} X_r({\mathbf x})}
{\mathrm d} {\mathbf x}
{\mathbf d}r.
\end{equation*}
And then, for all $s,t \in [0,T]$ with 
$s<t$, with probability 1, 
\begin{equation*}
\int_s^t z_r \cdot {\mathrm d} {\mathscr J}_r \geq 
\int_s^t 
\int_{{\mathbb S}}
\frac{z_{r}'({\mathbf x})}{E_{\varepsilon}({\mathbf x}) + D_{\mathbf{x}} X_r({\mathbf x})}
{\mathrm d} {\mathbf x}
{\mathbf d}r.
\end{equation*}
Using continuity 
in $(s,t)$ of both sides, we deduce that the result holds true, with probability 1, for all 
$s,t \in [0,T]$ with $s < t$. This completes the proof. 
\end{proof}

\subsection{%
  \texorpdfstring{%
    The dynamics of the weak-limit process $(X_t)_{0 \le t \le T}$
  }{%
    The dynamics of the weak-limit process%
  }%
}

In this subsection, we finally characterize the dynamics of any weak-limit process $(X_t,W_t)_{0 \le t \le T}$ 
arising from 
Proposition \ref{remark 4.1}.

\begin{prop}\label{Propo 5.13}    For $T>0$ and $\delta >0$, let $(X_t, W_t)_{0 \le t \le T}$ be a weak limit in $\mathbb{D}([0,T], \mathcal{H}^{-1}_{\mathrm{Sym}}(\mathbb{S})\times L^2_{\mathrm{Sym}}(\mathbb{S})
    )$ of the family $\{(X_t^{(h)},W_t)_{0 \le t \le T}\}_{{h \in T/{\mathbb N}^*}}$ when $h\searrow 0$. Then, almost surely, the following equation holds in a weak sense:
\begin{align}\label{RHSE2}
\mathrm{d}X_t(\mathbf{x}) = -\frac{1}{2}D_{\mathbf{x}}\Big(\frac{1}{D_{\mathbf{x}}X_t(\mathbf{x})}\Big)\mathrm{d}t + \Delta_{\mathbf{x}} X_t(\mathbf{x})\mathrm{d}t +\mathrm{d}\nu_t(\mathbf{x}) + \mathrm{d}W_t(\mathbf{x}), \quad \forall \mathbf{x}\in \mathbb{S}, \ t \in  [0,T], 
\end{align}
where $(\nu_t)_{0 \le t \le T}$ stands for a process with values in ${\mathcal H}^{-2}_{\rm Sym}({\mathbb S})$ that is non-decreasing in the sense that, for any function $\varphi \in {\mathcal H}_{\rm Sym}^2({\mathbb S})$, 
$(\langle \varphi,\nu_t\rangle)_{0 \le t \le T}$ is non-decreasing. Moreover, for all $t \in [0,T]$,
\begin{equation*}
\lim_{\epsilon \searrow 0}
\mathbb{E}\int_{0}^{t} \mathrm{e}^{\epsilon \Delta}  X_r\cdot\mathrm{d}{\nu}_r =0,
\quad 
\lim_{\epsilon \searrow 0}
\mathbb{E} \biggl[ \biggl\vert \int_{0}^{t} \int_{\mathbb S} \frac{ \mathrm{e}^{\epsilon \Delta} 
D_{\mathbf{x}} X_r({\mathbf x})}{D_{\mathbf{x}} X_r({\mathbf x})} \mathrm{d}{\mathbf x} - t
\biggr\vert \biggr]=0.
\end{equation*}

In addition, the process $(X_t)_{0 \le t \le T}$ has time-continuous trajectories with respect to the norm $\|\cdot\|_2$.
\end{prop}
\begin{proof}
Let $(X_t, W_t,\zeta_t,{\mathscr J}_t)_{0 \le t \le T}$
be as in the statement of Proposition \ref{remark 4.1}. 
\vskip 4pt

\textit{First Step.} We know from 
item 4 in the statement of Proposition \ref{remark 4.1} that, for any 
$\varphi \in {\mathcal H}^2_{\rm Sym}({\mathbb S})$, 
for all 
$0 \leq s \leq t \leq T$, 
\begin{align*}
	\langle X_t-X_s, \phi \rangle = \langle \zeta_t-\zeta_s, \phi \rangle + \int_s^t \langle X_r,\Delta_{\mathbf{x}} \phi\rangle \mathrm{d}r + \langle W_t-W_s,\phi\rangle.
\end{align*}
The goal is to provide a closed-form expression of the component process $t\mapsto \langle \zeta_t, \varphi \rangle$ that appears in the above equation. To this end, we will compute the $L^2\text{-norm}$ of $X_t$ in two different ways and then identify the result by comparison. 
For this, choose $\varphi = \mathrm{e}^{\epsilon \Delta}e_k$, with $k\in \mathbb{N}$ and then apply It\^{o}'s rule to obtain:
\begin{align*}
\mathrm{d}\langle X_t, \mathrm{e}^{\epsilon \Delta}e_k \rangle^2 &= 2\langle  X_t, \mathrm{e}^{\epsilon \Delta}e_k \rangle \mathrm{d}\langle  X_t, \Delta_{\mathbf{x}} \mathrm{e}^{\epsilon \Delta}e_k \rangle + 2\langle  X_t, \mathrm{e}^{\epsilon \Delta}e_k \rangle \mathrm{d}\langle \zeta_t, \mathrm{e}^{\epsilon \Delta}e_k \rangle\\&\quad +  2\langle  X_t, \mathrm{e}^{\epsilon \Delta}e_k \rangle \mathrm{d}\langle W_t, \mathrm{e}^{\epsilon \Delta}e_k \rangle+ \mathrm{d}[\langle  X_{\cdot}, \mathrm{e}^{\epsilon \Delta}e_k \rangle]_t,
\end{align*}
where the symbol $[\cdot]_t$ represents the bracket and $\epsilon >0$, here represents a regularization parameter, different from the parameter $\varepsilon$
introduced in \eqref{eq:varepsilon}.
Now, integrating from $0$ to $t$ and summing over $k \in \mathbb{N}$ the latter identity, we obtain the following
\begin{align}\label{254}
	\|\mathrm{e}^{\epsilon \Delta}X_t\|_2^2 + 2\int_{0}^{t} \|D_{\mathbf{x}}(\mathrm{e}^{\epsilon \Delta}X_r)\|_2^2 \mathrm{d}r 
	= \|\mathrm{e}^{\epsilon \Delta}X_0\|_2^2 + 2\int_{0}^{t} \mathrm{e}^{2\epsilon \Delta} X_r\cdot\mathrm{d}\zeta_r 
	+  2\int_{0}^{t} \mathrm{e}^{2\epsilon \Delta}X_r\cdot\mathrm{d}W_r
	+c_{\lambda,\epsilon} t,
\end{align}
where $c_{\lambda,\epsilon}:= \sum_{k \in \mathbb{N}} \lambda_k^2 \mathrm{e}^{-8 {\pi^2} k^2 \epsilon}$(see \cite{DelHam25}). The integrand $\psi_r:= \mathrm{e}^{2\epsilon \Delta} X_r $ fulfills the requirements of Theorem \ref{thm:integral:varsigma}, which in turn implies the well-posedness of the integral $\int_{0}^{t} \mathrm{e}^{2\epsilon \Delta} X_r\cdot\mathrm{d}\zeta_r$.
By taking the expectation on both sides of the latter, 
{we can get rid of} 
the stochastic integral. Therefore, using the 
fact that 
${\mathbb E} \int_0^T \| D_{\mathbf{x}} X_r \|^2_2 {\mathrm d} r <+\infty$, 
and then
letting $\epsilon\searrow 0$, we get 
\begin{align}\label{262}
	\mathbb{E} \bigl[ \|X_t\|_2^2 \bigr]
	=\mathbb{E} \bigl[ \|X_0\|_2^2 \bigr] +  2\lim_{\epsilon\searrow 0}\mathbb{E}\int_{0}^{t} \mathrm{e}^{2\epsilon \Delta}  X_r\cdot\mathrm{d}\zeta_r- 2\int_{0}^{t} \mathbb{E} \bigl[\|D_{\mathbf{x}}X_r\|_2^2 
    \bigr]\mathrm{d}r +c_{\lambda} t,
\end{align}
where $c_{\lambda}:=\sum_{k \in \mathbb{N}} \lambda_k^2$. 

Let us return to the expression \eqref{254}. From \eqref{eq:limit:zeta},  the integral with respect to $(\zeta_t)_{t \in [0,T]}$ is also understood as the sum of the integrals with respect to $(\mathscr{J}_t)_{t \in [0,T]}$ and $(\eta_t)_{t \in [0,T]}$, respectively. In addition, from item 5 in the statement of Proposition \ref{remark 4.1}, the integral with respect to the latter process is also non-negative, i.e. $\int_{0}^{t} {\rm e}^{2\epsilon \Delta} X_r\cdot\mathrm{d}\eta_r \geq 0$. Hence, \eqref{ineq:4.18} leads to: 
\begin{equation}\label{ineq.4.26}
\begin{split}
	\|\mathrm{e}^{\epsilon \Delta}X_t\|_2^2 &+ 2\int_{0}^{t} \|\mathrm{e}^{\epsilon \Delta}D_{\mathbf{x}}X_r\|_2^2 \mathrm{d}r 
	\\
    &\geq  \|\mathrm{e}^{\epsilon \Delta}X_0\|_2^2 + \int_0^t\!\!\!\int_{\mathbb{S}}\frac{\mathrm{e}^{2\epsilon \Delta}D_{\mathbf{x}}X_r(\mathbf{x})}{E_\varepsilon(\mathbf{x})+ D_{\mathbf{x}}{X}_{r}(\mathbf{x})}\mathrm{d}\mathbf{x}\mathrm{d}r
	+  2\int_{0}^{t} \mathrm{e}^{2\epsilon \Delta}X_r\cdot\mathrm{d}W_r
	+c_{\lambda,\epsilon} t.
    \end{split}
\end{equation}
Let us focus for a moment on the second term on the right-hand side of the above inequality. 
Since $X_r$ takes values in 
${\mathcal U}^2({\mathbb S})$, $D_{\mathbf{x}}X_r({\mathbf x})$ and $E_{\varepsilon}({\mathbf x})$ have the same sign for almost every ${\mathbf x} \in {\mathbb S}$.
In particular,
\begin{equation}
\label{ineq4.27:-1}
\begin{split}
\left\vert 
\frac{{\rm e}^{2 \epsilon \Delta} D_{\mathbf{x}} X_r ({\mathbf x})}{E_{\varepsilon}({\mathbf x}) 
+
D_{\mathbf{x}} X_r({\mathbf x})
}
\right\vert
= 
\frac{ \vert 
{\rm e}^{2 \epsilon \Delta} 
D_{\mathbf{x}} X_r({\mathbf x}) 
\vert}{\varepsilon + 
\vert D_{\mathbf{x}} X_r({\mathbf x})\vert}
\leq 
\frac{ \vert 
{\rm e}^{2 \epsilon \Delta} 
D_{\mathbf{x}} X_r({\mathbf x}) 
\vert}{\varepsilon}. 
\end{split}
\end{equation}
Using the fact that $\int_0^t \int_{\mathbb S} \vert D_{\mathbf{x}} X_r({\mathbf x}) \vert^2 {\mathrm d} {\mathbf x} 
{\mathrm d} r < + \infty$, we deduce from a standard uniform integrability argument that 
\begin{align*}
\lim_{\epsilon\searrow 0}\mathbb{E}\int_0^t\int_{\mathbb{S}}\frac{\mathrm{e}^{2\epsilon \Delta}D_{\mathbf{x}}X_r(\mathbf{x})}{E_\varepsilon(\mathbf{x})+ D_{\mathbf{x}}{X}_{r}(\mathbf{x})}\mathrm{d}\mathbf{x}\mathrm{d}r = 
\mathbb{E} \int_0^t \int_{\mathbb{S}}
\frac{ D_{\mathbf{x}}X_r(\mathbf{x})}{E_\varepsilon({\mathbf x}) + D_{\mathbf{x}} X_r({\mathbf x})}\mathrm{d}r\mathrm{d}\mathbf{x}.
\end{align*}
Observe that 
\begin{equation}
\label{ineq4.27:-2}
\left\vert \frac{D_{\mathbf{x}} X_r({\mathbf x})}{E_{\varepsilon}({\mathbf x}) + 
D_{\mathbf{x}} X_r({\mathbf x})
}\right\vert = \frac{\vert D_{\mathbf{x}} X_r({\mathbf x}) \vert}{\varepsilon + 
\vert  D_{\mathbf{x}} X_r({\mathbf x})\vert} \leq 1,
\end{equation}
we deduce from Lebesgue dominated convergence theorem that 
\begin{align}\label{ineq.4.27}
\lim_{\varepsilon \searrow 
0} \lim_{\epsilon\searrow 0}\mathbb{E}\int_0^t\int_{\mathbb{S}}\frac{\mathrm{e}^{2\epsilon \Delta}D_{\mathbf{x}} X_r(\mathbf{x})}{E_\varepsilon(\mathbf{x})+ D_{\mathbf{x}} {X}_{r}(\mathbf{x})}\mathrm{d}\mathbf{x}\mathrm{d}r 
= \mathbb{E} \int_0^t \int_{\mathbb{S}}{\bf 1}_{\{D_{\mathbf{x}} X_r(\mathbf{x})\not = 0\}}\mathrm{d}r\mathrm{d}\mathbf{x}.
\end{align}
Additionally, Corollary \ref{cor 3.9} says that, for almost every $r \in [0,T]$, $\mathbb{P}{\{D_{\mathbf{x}} X_r(\mathbf{x}) \not = 0\}}=1$. Therefore, the quantity on the right-hand side of \eqref{ineq.4.27} is equal to $t$.

Taking now the expectations in \eqref{ineq.4.26}, letting $\varepsilon, \epsilon \searrow 0$
{(using \eqref{ineq.4.27} to do so)}, and then combining the upper bound obtained with the lower bound  \eqref{260}, we finally arrive at the identity
\begin{align}\label{261}
\mathbb{E}\bigl[\| X_{t}\|_2^2\bigr] = \mathbb{E}\bigl[\|X_0\|_2^2\bigr] +t -2\int_0^{t} \mathbb{E} \bigl[ \|D_{\mathbf{x}} X_{r}\|_2^2
\bigr] \mathrm{d}r + c_{\lambda}t.
\end{align}
We now proceed to compare the two $L^2\text{-expansions}$ of the same limiting-process  $(X_t)_{t\in [0,T]}$, given respectively by \eqref{261} and \eqref{262}, to obtain the desired characterization of the weak limit process $(\zeta_t)_{t\geq 0} \in \mathcal{C}([0,T],\mathcal{H}^{-2}_{\mathrm{Sym}}(\mathbb{S}))$. This reads as follows:  for all 
$t \in [0,T]$,
\begin{align}\label{eq2.53}
\lim_{\epsilon\searrow 0}\mathbb{E}\int_{0}^{t} \mathrm{e}^{2\epsilon \Delta}  X_r\cdot\mathrm{d}\zeta_r =  \frac{1}{2}t.
\end{align}
We now prove that the limit, previously shown to hold in the mean, also holds in $L^1$. By \eqref{ineq:4.18} , we have, for any 
$\varepsilon >0$,
\begin{equation*}
\int_{0}^{t} \mathrm{e}^{2\epsilon \Delta}  X_r\cdot\mathrm{d}\zeta_r
\geq \frac12 \int_0^t \int_{\mathbb{S}}\frac{\mathrm{e}^{2\epsilon \Delta}D_{\mathbf{x}} X_r(\mathbf{x})}{E_\varepsilon(\mathbf{x})+ D_{\mathbf{x}} {X}_{r}(\mathbf{x})}\mathrm{d}\mathbf{x}\mathrm{d}r.
\end{equation*}
Denoting by $x_+=\max(x,0)$ the positive part of $x \in {\mathbb R}$, we get
\begin{equation*}
{\mathbb E} \biggl[ \biggl( \frac{t}2 - \int_0^t {\rm e}^{2 \epsilon \Delta} X_r \cdot \ud \zeta_r 
\biggr)_+ \biggr]
\leq {\mathbb E} \biggl[ \biggl( \frac{t}2 - \frac12 \int_0^t \int_{\mathbb S}
\frac{{\rm e}^{2 \epsilon \Delta} D_{\mathbf{x}} X_r({\mathbf x})}{E_{\varepsilon}({\mathbf x})+D_{\mathbf{x}} X_r({\mathbf x})} \ud {\mathbf x} \ud r \biggr)_+ \biggr],
\end{equation*}
from which we deduce that 
\begin{equation*}
\limsup_{\epsilon \searrow 0}
{\mathbb E} \biggl[ \biggl( \frac{t}2 - \int_0^t {\rm e}^{2 \epsilon \Delta} X_r \cdot \ud \zeta_r 
\biggr)_+ \biggr]
\leq \limsup_{\epsilon \searrow 0}
{\mathbb E} \biggl[ \biggl( \frac{t}2 - \frac12 \int_0^t \int_{\mathbb S}
\frac{{\rm e}^{2 \epsilon \Delta} D_{\mathbf{x}} X_r({\mathbf x})}{E_{\varepsilon}({\mathbf x})+D_{\mathbf{x}} X_r({\mathbf x})} \ud {\mathbf x} \ud r \biggr)_+ \biggr].
\end{equation*}
By \eqref{ineq4.27:-1} and the dominated convergence, this gives
\begin{equation*}
\limsup_{\epsilon \searrow 0}
{\mathbb E} \biggl[ \biggl( \frac{t}2 - \frac12 \int_0^t \int_{\mathbb S}
\frac{{\rm e}^{2 \epsilon \Delta} D_{\mathbf{x}} X_r({\mathbf x})}{E_{\varepsilon}({\mathbf x})+D_{\mathbf{x}} X_r({\mathbf x})} \ud {\mathbf x} \ud r \biggr)_+ \biggr]
= {\mathbb E} \biggl[ \biggl( \frac{t}2 - \frac12 \int_0^t \int_{\mathbb S}
\frac{D_{\mathbf{x}} X_r({\mathbf x})}{E_{\varepsilon}({\mathbf x})+D_{\mathbf{x}} X_r({\mathbf x})} \ud {\mathbf x} \ud r \biggr)_+ \biggr].
\end{equation*}
Using  \eqref{ineq4.27:-2}, 
we deduce that the right-hand side tends to $0$ as $\varepsilon$ tends to $0$. Combining the last two displays, we obtain 
\begin{equation*}
\limsup_{\epsilon \searrow 0}
{\mathbb E} \biggl[ \biggl( \frac{t}2 - \int_0^t {\rm e}^{2 \epsilon \Delta} X_r \cdot \ud \zeta_r 
\biggr)_+ \biggr]=0. 
\end{equation*}
Thanks to \eqref{eq2.53}, we deduce that
\begin{equation*}
\limsup_{\epsilon \searrow 0}
{\mathbb E} \biggl[ \biggl\vert \frac{t}2 - \int_0^t {\rm e}^{2 \epsilon \Delta} X_r \cdot \ud \zeta_r 
\biggr\vert \biggr]=0. 
\end{equation*}
\color{black}

\textit{Second Step.} We now address the construction of the process 
$(\nu_t)_{0 \le t \le T}$ in the statement. 
Letting $\varepsilon$ tend to $0$ in 
\eqref{ineq:4.18}, we know that, for any process $(z_t)_{0 \le t \le T}$ with continuous trajectories in 
${\mathcal H}^4({\mathbb S}) \cap {\mathcal U}^2({\mathbb S})$, it holds, with probability 1, for all $s,t \in [0,T]$ with $s<t$, 
\begin{equation}
\label{eq:identification:1}
\int_s^t z_r \cdot \mathrm{d} {\mathscr J}_r  \geq \frac{1}{2}
\int_s^t \int_{\mathbb{S}}\frac{z_r'(\mathbf{x})}{D_{\mathbf{x}} X_r(\mathbf{x})}\mathrm{d}\mathbf{x}\mathrm{d}r.
\end{equation}
Letting
\begin{equation}
\label{eq:def:nut}
\nu_t := \zeta_t + \frac12 \int_0^t D_{\mathbf{x}}\left( \frac1{D_{\mathbf{x}} X_r} \right) {\mathrm d} r, \quad t \in [0,T],
\end{equation}
we have, for any $\varphi \in {\mathcal H}^2_{\rm Sym}({\mathbb S})$,

\begin{equation*}
\langle \varphi, \nu_t \rangle 
= 
\langle \varphi,\zeta_t \rangle 
-
\frac12 
\int_0^t \int_{\mathbb S}
\frac{\varphi'({\mathbf x})}{D_{\mathbf{x}} X_r({\mathbf x})} {\mathrm d} {\mathbf x}
{\mathrm d} r, \quad t \in [0,T],
\end{equation*}
the integral on the right-hand side being well-defined since 
$\varphi'$ is bounded and $\int_0^T\int_{\mathbb S} 1/\vert D X_r({\mathbf x}) \vert {\mathrm d} {\mathbf x} {\mathrm d} r$ is 
(almost surely) finite. In particular, 
$(\nu_t)_{t \in [0,T]}$ is a continuous process with values in ${\mathcal H}^{-2}_{\rm Sym}({\mathbb S})$. By \eqref{eq:identification:1}, it satisfies, for all $\varphi \in {\mathcal H}^4({\mathbb S}) \cap {\mathcal U}^2({\mathbb S})$, with probability 1, for all $s,t \in [0,T]$ with $s< t$, 
\begin{equation*}
\begin{split}
\langle \varphi, \nu_t - \nu_s \rangle 
&= \langle \varphi, \zeta_t - \zeta_s \rangle - 
\frac12
\int_s^t \int_{\mathbb S}
\frac{\varphi'({\mathbf x})}{D_{\mathbf{x}} X_r({\mathbf x})} {\mathrm d} {\mathbf x} {\mathrm d}r
\\
&\geq \langle \varphi, {\mathscr J}_t - {\mathscr J}_s \rangle - 
\frac12
\int_s^t \int_{\mathbb S}
\frac{\varphi'({\mathbf x})}{D_{\mathbf{x}} X_r({\mathbf x})} {\mathrm d} {\mathbf x} {\mathrm d}r
\geq 0,
\end{split}
\end{equation*}
which proves that the process $(\nu_t)_{t \in [0,T]}$ is  non-decreasing in the sense indicated in the statement (by a standard separability argument, we easily get, with probability 1, for all $\varphi \in {\mathcal H}^2({\mathbb S}) \cap {\mathcal U}^2({\mathbb S})$, for all $s,t \in [0,T]$ with $s< t$,  $\langle \varphi, \nu_t - \nu_s \rangle \geq 0$). By \eqref{eq2.53} and \eqref{eq:identification:1}, we deduce that, with probability 1, for all $t \in [0,T]$,
\begin{equation}
\label{eq:identification:22}
\begin{split}
  \frac{1}{2}t &= \lim_{\epsilon \searrow 0}
{\mathbb E} \int_0^t 
    \mathrm{e}^{2\epsilon \Delta}  X_r\cdot\mathrm{d}\zeta_r
    \\
    &\geq  \limsup_{\epsilon \searrow 0}
    \frac12 {\mathbb E}
    \int_0^t \int_{\mathbb S} \frac{{\rm e}^{2 \epsilon \Delta} D_{\mathbf{x}} X_r({\mathbf x})}{D_{\mathbf{x}} X_r({\mathbf x})} {\mathrm d} {\mathbf x} {\mathrm d}r \geq 
    \liminf_{\epsilon \searrow 0}
    \frac12 {\mathbb E}
    \int_0^t \int_{\mathbb S} \frac{{\rm e}^{2 \epsilon \Delta} D_{\mathbf{x}} X_r({\mathbf x})}{D_{\mathbf{x}} X_r({\mathbf x})} {\mathrm d} {\mathbf x} {\mathrm d}r.
\end{split}
\end{equation}
Following the proof of \eqref{ineq4.27:-1}, 
the integrand in the last term is non-negative. 
And then, using the same notation as in \eqref{ineq.4.27}, we get, for any $\varepsilon >0$,
\begin{equation}
\label{eq:cv:to:t/2}
\begin{split}
\frac12 t 
&\geq  \limsup_{\epsilon \searrow 0}
    \frac12 {\mathbb E}
    \int_0^t \int_{\mathbb S} \frac{{\rm e}^{2 \epsilon \Delta} D_{\mathbf{x}} X_r({\mathbf x})}{D_{\mathbf{x}} X_r({\mathbf x})} {\mathrm d} {\mathbf x} {\mathrm d}r
    \\
    &\geq 
    \liminf_{\epsilon \searrow 0}
    \frac12 {\mathbb E}
    \int_0^t \int_{\mathbb S} \frac{{\rm e}^{2 \epsilon \Delta} D_{\mathbf{x}} X_r({\mathbf x})}{D_{\mathbf{x}} X_r({\mathbf x})} {\mathrm d} {\mathbf x} {\mathrm d}r
    \\
&\geq \liminf_{\varepsilon \searrow 0} \liminf_{\epsilon \searrow 0} \frac12 {\mathbb E}
\int_0^t \int_{\mathbb S}
\frac{{\rm e}^{2 \varepsilon \Delta} D_{\mathbf{x}} X_r({\mathbf x})}{E_{\varepsilon}({\mathbf x}) + D_{\mathbf{x}} X_r({\mathbf x})}
 \ud {\mathbf x} \ud r = \frac12 t.
 \end{split}
 \end{equation}
 This shows that 
\begin{equation}
\label{eq:mean:E:t/2}
\lim_{\varepsilon \searrow 0} \frac12 {\mathbb E}
\int_0^t \int_{\mathbb S} \frac{{\rm e}^{2 \varepsilon \Delta} D_{\mathbf{x}} X_r({\mathbf x})}{D_{\mathbf{x}} X_r({\mathbf x})} {\mathrm d} {\mathbf x} {\mathrm d}r 
= \frac{t}2. 
\end{equation}
Following the end of the first step, we prove that the limit, which has been proved to hold in the mean, also holds in $L^1$.
By the same argument as in the derivation of \eqref{eq:identification:22},  we have, for any $\epsilon>0$ and $\varepsilon >0$, 
\begin{equation*}
\begin{split}
{\mathbb E} \biggl[ \biggl( 
\int_0^t 
    \mathrm{e}^{2\epsilon \Delta}  X_r\cdot\mathrm{d}\zeta_r
- \frac{t}2
\biggr)_+ \biggr]
&\geq {\mathbb E} \biggl[ \biggl( 
\frac12 \int_0^t\int_{\mathbb S} \frac{{\rm e}^{2 \epsilon \Delta} D_{\mathbf{x}} X_r({\mathbf x})}{ D_{\mathbf{x}} X_r({\mathbf x})}
    \ud {\mathbf x}\mathrm{d} r
- \frac{t}2
\biggr)_+ \biggr]
\\
&\geq {\mathbb E} \biggl[ \biggl( 
\frac12 \int_0^t \int_{\mathbb S} \frac{{\rm e}^{2 \epsilon \Delta} D_{\mathbf{x}} X_r({\mathbf x})}{E_{\varepsilon}({\mathbf x})+ D_{\mathbf{x}} X_r({\mathbf x})}
    \ud {\mathbf x}\mathrm{d} r
- \frac{t}2
\biggr)_+ \biggr]
\end{split}
\end{equation*}
By \eqref{ineq4.27:-1} and \eqref{ineq4.27:-2}, and similar to \eqref{ineq.4.27} , we get 
\begin{equation*}
\lim_{\varepsilon \searrow 0}
\lim_{\epsilon \searrow 0}
{\mathbb E} \biggl[ \biggl( 
\frac12 \int_0^t \int_{\mathbb S} \frac{{\rm e}^{2 \epsilon \Delta} D_{\mathbf{x}} X_r({\mathbf x})}{E_{\varepsilon}({\mathbf x})+ D_{\mathbf{x}} X_r({\mathbf x})}
    \ud {\mathbf x}\mathrm{d} r
- \frac{t}2
\biggr)_+ \biggr] =
\lim_{\varepsilon \searrow 0}
{\mathbb E} \biggl[ \biggl( 
\frac12 \int_0^t\int_{\mathbb S} \frac{D_{\mathbf{x}} X_r({\mathbf x})}{E_{\varepsilon}({\mathbf x})+ D_{\mathbf{x}} X_r({\mathbf x})}
    \ud {\mathbf x}\mathrm{d} r
- \frac{t}2
\biggr)_+ \biggr] =0,
\end{equation*}
which proves that 
\begin{equation*}
\lim_{\epsilon \searrow 0}
{\mathbb E} \biggl[ \biggl( 
\frac12 \int_0^t\int_{\mathbb S} \frac{{\rm e}^{2 \epsilon \Delta} D_{\mathbf{x}} X_r({\mathbf x})}{  D_{\mathbf{x}} X_r({\mathbf x})}
    \ud {\mathbf x}\mathrm{d} r
- \frac{t}2
\biggr)_+ \biggr] =0.
\end{equation*}
Combining with \eqref{eq:mean:E:t/2}, this leads to 
\begin{equation}
\label{eq:cv:L1:t/2}
\lim_{\epsilon \searrow 0}
{\mathbb E} \biggl[ \biggl\vert 
\frac12 \int_0^t\int_{\mathbb S} \frac{{\rm e}^{2 \epsilon \Delta} D_{\mathbf{x}} X_r({\mathbf x})}{  D_{\mathbf{x}} X_r({\mathbf x})}
    \ud {\mathbf x}\mathrm{d} r
- \frac{t}2
\biggr\vert \biggr] =0.
\end{equation}
\color{black}
Returning to \eqref{eq:identification:22}, we deduce that, with probability 1, for all $t \in [0,T]$, 
\begin{equation*}
\lim_{\varepsilon \searrow 0} 
{\mathbb E} \int_0^t {\rm e}^{2 \epsilon \Delta} X_r \cdot\ud \nu_r =0.
\end{equation*}

\textit{Third Step.} It remains to establish the continuity of the trajectories of $(X_t)_{t \in [0,T]}$ in $L^2({\mathbb S})$. Returning once again to \eqref{254}, now with another time parameter $s \in [0,T]$, for $s\leq t$ we have
\begin{align*}
	\left| \|\mathrm{e}^{\epsilon \Delta}X_t\|_2^2 - \|\mathrm{e}^{\epsilon \Delta}X_s\|_2^2 \right| &\leq  2\int_{s}^{t} \|D_{\mathbf{x}} X_r\|_2^2 \mathrm{d}r  
 + 2\int_{s}^{t} \mathrm{e}^{2\epsilon \Delta} X_r\cdot\mathrm{d}\zeta_r 
	+  2\left|\int_{s}^{t} \mathrm{e}^{2\epsilon \Delta}X_r\cdot\mathrm{d}W_r\right|
	+c_{\lambda,\epsilon} (t-s),
\end{align*}
Using almost the same arguments as the ones displayed in \cite[Proof of Proposition 7, Pg 74-75]{DelHam25} and further letting $\epsilon \searrow 0$ in the above inequality along with \eqref{eq2.53} we deduce thereof
\begin{align*}
	\Big| \Vert X_t\Vert_2^2 - \Vert X_s\Vert_2^2 \Big| &\leq  2\int_{s}^{t} \|D_{\mathbf{x}} X_r\|_2^2 \mathrm{d}r  
 + (t-s)+  2\left|\int_{s}^{t} X_r\cdot\mathrm{d}W_r\right|
	+c_{\lambda} (t-s).
\end{align*}
As $s \rightarrow t$, the terms on the right-hand side converge to 0. The latter implies the existence of an event of probability one, under which the trajectory $[0,T]\ni t\mapsto \|X_t\|_2^2$ is continuous. Combining this with the fact that the limiting-process $(X_t)_{t\in [0,T]}$ is continuous with respect to the norm $\|\cdot\|_{2,-1}$~(see Remark \ref{remark 4.1}), we then deduce that, the process $(X_t)_{t\in [0,T]}$ has time-continuous trajectories with respect to the norm $\|\cdot\|_2$.
\end{proof}

\section{\textbf{Rearranged SHE with Entropy-Driven Drift}}
\label{se:6}
In this section, we rely on Proposition
\ref{Propo 5.13}
to rigorously define a version of the RSHE driven by the derivative of entropy. Proposition \ref{Propo 5.13} ensures the existence of weak solutions. Here, we prove strong uniqueness. We also establish the Feller regularity properties of the semigroup generated by the equation.

\subsection{Definition and Uniqueness of the solution} 
The following definition is an analog of Definition \ref{Defi 0.1}, but adapted to the scheme \eqref{scheme1}:

\begin{defi}\label{Defi 317} 
Given a probabilistic set-up $(\Omega,\mathfrak{F},{\mathbb F}= \{\mathfrak{F}_t\}_{t\geq 0},\mathbb{P},(W_t)_{t\geq 0})$ as defined in Subsection \ref{Notations}, and an ${\mathcal F}_0$--measurable initial condition $X_0$ with values in ${\mathcal U}^2({\mathbb S} )$ satisfying ${\mathbb E}[\| X_0\|_2^p] < + \infty$ for all $p \geq 1$, we say that a pair of  processes $(X_t,{\eta}_t)$ solves the equation \eqref{RHSE2}, initialized at $X_0$, if
\begin{enumerate}
\item[(i)] $(X_t)_{t\geq 0}$ is continuous $\mathbb{F}-$adapted with values in $\mathcal{U}^2(\mathbb{S})$; almost surely, the path $(X_t)_{t \geq 0}$ belongs to $L^2_{\rm loc}([0,+\infty),{\mathcal H}^1_{\rm Sym}({\mathbb S}))$ (i.e., $(X_t)_{0 \le t \le T} \in L^2([0,T],{\mathcal H}^1_{\rm Sym}({\mathbb S}))$). Moreover, 
\begin{equation*}
\forall T>0, \quad {\mathbb E} \int_0^T \| D_{\mathbf{x}} X_t \|_2^2 {\mathrm d} t < + \infty, \quad 
 {\mathbb E} \int_0^T \int_{\mathbb S} \frac1{\vert D_{\mathbf{x}} X_t({\mathbf x})\vert} {\mathrm d}{\mathbf x} < + \infty;
\end{equation*}
\item[(ii)] $({\eta}_t)_{t\geq 0}$ is a continuous $\mathbb{F}\text{-adapted}$ process that takes values in $\mathcal{H}_{\mathrm{Sym}}^{-2}(\mathbb{S})$, starting from $0$ at $0$ such that, with probability 1, for any $\varphi \in \mathcal{H}_{\mathrm{Sym}}^{2}(\mathbb{S})\cap \mathcal{U}^2(\mathbb{S})$ the path $t\mapsto\langle {\eta}_t, \varphi\rangle$ is non-decreasing;
\item[(iii)]  with probability one, for any $\varphi \in \mathcal{H}_{\mathrm{Sym}}^{2}(\mathbb{S})$, for all $t \geq 0$
\begin{align}\label{eq:4.28}
\langle X_t, \phi \rangle =
\langle X_0,\phi \rangle + \frac{1}{2}\int_0^t \Big\langle \frac{1}{D_{\mathbf{x}} X_r},D_{\mathbf{x}} \phi \Big\rangle \mathrm{d}r  + \int_0^t \langle X_r,\Delta_{\mathbf{x}} \phi \rangle \mathrm{d}r + \langle {\eta}_t,\phi\rangle +\langle W_t,\phi\rangle, \quad t \geq 0;
\end{align}
\item[(iv)] for any $t\geq 0$,
\[\lim_{\epsilon\searrow 0}\mathbb{E}\int_{0}^{t} \mathrm{e}^{\epsilon \Delta} X_r\cdot\mathrm{d}{\eta}_r =0, 
\quad
\lim_{\epsilon \searrow 0}
\mathbb{E} \biggl[\biggl\vert  \int_{0}^{t}  \int_{\mathbb S}\frac{ \mathrm{e}^{\epsilon \Delta} 
D_{\mathbf{x}} X_r({\mathbf x})}{D_{\mathbf{x}} X_r({\mathbf x})} \mathrm{d}{\mathbf x} -t  \biggr\vert \biggr]=0,
\]
with the first integral being interpreted as in Theorem \ref{thm:integral:varsigma}.\end{enumerate}
\end{defi}
Observe that, compared to the notation used in the statement of Proposition \ref{Propo 5.13}, we here use the letter $\eta$ to denote the reflection process in equation \eqref{eq:4.28}, whereas in Proposition \ref{Propo 5.13} the same process is denoted by $\nu$. This choice is made for consistency with the notation adopted in Definition \ref{Defi 0.1}. However, we emphasize that the process 
$\eta$ in the above definition does not necessarily coincide with the process 
$\eta$ appearing in  \ref{eq:limit:zeta}; rather, it corresponds to the process $\nu$ defined in \eqref{eq:def:nut}.

\begin{remark}
\label{rem:item (iv):Def:Solution}
The second condition in item \textit{(iv)} of the definition is required to ensure the uniqueness of the solutions. Of course, this property holds whenever $D_{\mathbf{x}}X$ is a regular function of $\mathbf{x}$ that remains bounded away from $0$. However, under the standing assumptions, we cannot guaranty this, and therefore it must be imposed.
\end{remark}

\begin{remark}[Distributional drift as the derivative of the entropy]\label{Remark 5.16}
We now clarify the connection between the drift that appears in \eqref{eq:4.28} 
and the notion of entropy. Let 
$\mu \in {\mathcal P}({\mathbb R})$ be a probability measure with a density $p$ that is strictly positive. Denote by $F_{\mu}^{-1}$ the corresponding quantile function, as defined by \eqref{inverse}. Then

\begin{equation*}
p(y) = \frac{2}{(D_y F_\mu^{-1})(\tilde F_{\mu}(y)/2)} , \quad y \in {\mathbb R},
\end{equation*}
where $\tilde F_{\mu}$ denotes the standard cumulative distribution function of $\mu$.
In particular, the relative entropy of 
$p$ with respect to the Lebesgue measure on 
${\mathbb R}$ reads
\begin{equation*}
\begin{split}
\int_{\mathbb R} p(y) \ln(p(y)) {\mathrm d} y 
= 
\int_{\mathbb S} \ln \bigl(  \frac{2}{D_{\mathbf{x}} F_{\mu}^{-1}(\tilde F_{\mu}(F_\mu^{-1}({\mathbf x}))/2)}
\bigr) {\mathrm d} {\mathbf x}
&=\int_{\mathbb S} \ln \bigl(  \frac{2}{\vert D_{\mathbf{x}} F_{\mu}^{-1}({\mathbf x}) \vert}
\bigr) {\mathrm d} {\mathbf x}
\\
&= \ln(2) - \int_{\mathbb S}
\ln\bigl( \vert D_{\mathbf{x}} F_{\mu}^{-1}({\mathbf x}) \vert\bigr) {\mathrm d} {\mathbf x}. 
\end{split}
\end{equation*}
Take now a smooth symmetric function $\varphi : {\mathbb S} \rightarrow {\mathbb R}$ such that $\varphi({\mathbf x})=0$ for ${\mathbf x}$ in a neighborhood of ${\mathbb Z}/2$. 
For $\varepsilon >0$ small enough, $F_{\mu}^{-1} + \varepsilon \varphi$ remains in 
${\mathcal U}^2({\mathbb S})$ and is in fact strictly increasing; as such, it can be regarded as the quantile of a probability measure
with a strictly positive density. The derivative of the entropy along the direction $\varphi$ at the measure $\mu$ is thus given by the following derivative (if it exists):
\begin{equation*}
{\frac{\mathrm d}{\mathrm d \varepsilon}}_{\vert \varepsilon =0}
\left( - \int_{\mathbb S}
\ln \bigl( \vert D_{\mathbf{x}} F_{\mu}^{-1}({\mathbf x}) 
+ \varepsilon D_{\mathbf{x}} \varphi({\mathbf x}) 
\vert
\bigr) {\mathrm d} {\mathbf x} \right)
\end{equation*}
Since $\varphi$ is equal to $0$ in the neighborhood of $ {\mathbb Z}/2$, and 
$D F_{\mu}^{-1}$ is smooth away from ${\mathbb Z}/2$ (as given by the inverse of $p(y)$, for 
$y$ in a compact set), we easily deduce that 
\begin{equation*}
{\frac{\mathrm d}{\mathrm d \varepsilon}}_{\vert \varepsilon =0}
\left( - \int_{\mathbb S}
\ln \bigl( \vert D_{\mathbf{x}} F_{\mu}^{-1}({\mathbf x}) 
+ \varepsilon D_{\mathbf{x}} \varphi({\mathbf x}) 
\vert
\bigr) {\mathrm d} {\mathbf x} \right)
=
- \int_{\mathbb S} \frac{D_{\mathbf{x}} \varphi({\mathbf x})}{D_{\mathbf{x}} F_{\mu}^{-1}({\mathbf x})} {\mathrm d} {\mathbf x}.
\end{equation*}
 Obviously, we can interpret 
 the right-hand side as the action of the function $\varphi$ onto the distribution $D_{\mathbf{x}}(1/ D_{\mathbf{x}} F_{\mu}^{-1})$.
 This means that we should regard the second term on the right-hand side of \eqref{eq:4.28} as (half) the opposite of
 the derivative of the entropy, calculated at the distribution ${\rm Leb}_{\mathbb S} \circ X_r^{-1}$.
 Of course, this interpretation should be the right one: as the step size $h$ goes to zero, the effect of 
the convolution appearing on the second line of 
\eqref{scheme1} is to force the system to follow the opposite of the gradient of the entropy perturbed by the original rearranged dynamics. 
\end{remark}

In the following result, we establish the pathwise uniqueness of \eqref{eq:4.28}.

\begin{prop}\label{propo:4.17}
Given a probabilistic set-up $(\Omega,\mathfrak{F},{\mathbb F}= \{\mathfrak{F}_t\}_{t\geq 0},\mathbb{P},(W_t)_{t\geq 0})$ as defined in 
Subsection \ref{Notations}, and an ${\mathcal F}_0-$measurable initial condition $X_0$ with values in $\mathcal{U}^2(\mathbb{S})$ satisfying $\mathbb{E}[\|X_0\|_2^{p}]< + \infty$ for all $p \geq 1$, the RSHE~\eqref{RHSE2}  admits at most one solution $(X_t,{\eta}_t)_{t\geq 0}$ that satisfies Definition \ref{Defi 317}.

Moreover, if $(X^x,\eta^x)$ and $(X^y,\eta^y)$ are two solutions to \eqref{eq:4.28} starting at $x \text{ and } y \in \mathcal{U}^2(\mathbb{S})$ respectively, then, $\mathbb{P}\mathrm{-a.s.}$,
\begin{align}\label{ineq:4.32}
	\|X_{t}^x-X_t^y\|_2^2 + 2\int_{0}^{t} \|D_{\mathbf{x}} (X_s^x-X_s^y)\|_2^2 \mathrm{d}s 
    + \int_0^t \int_{\mathbb S}
    \biggl\vert  \sqrt{\frac{D_{\mathbf{x}} X^x_s({\mathbf x})}{D_{\mathbf{x}} X^y_s({\mathbf x})}}- \sqrt{\frac{D_{\mathbf{x}} X^y_s({\mathbf x})}{D_{\mathbf{x}} X^x_s({\mathbf x})}} \biggr\vert^2 {\mathrm d}{\mathbf x}  {\mathrm d}s 
    \leq \|x-y\|_2^2.
\end{align}
In particular, the pair solution $(X_t,\eta_t)_{t\in [0,T]}$ has a version that is measurable, continuous in time and Lipschitz continuous in $x \in {\mathcal U}^2(\mathbb{S})$.
\end{prop}
\begin{proof}
On the same probabilistic set-up as in the statement, let us consider two solutions $(X_t,\eta_t)_{t \geq 0}$ and $(X_t',\eta_t')_{t \geq  0}$ to the equation \eqref{eq:4.28}. We then let 
\begin{equation*}
\begin{split}
\zeta_t := \eta_t - \frac12 
D_{\mathbf{x}} \Bigl( \frac{1}{D_{\mathbf{x}} X_t} \Bigr), 
\quad
\zeta_t' := \eta_t' - \frac12 
D_{\mathbf{x}} \Bigl( \frac1{D_{\mathbf{x}} X_t'} \Bigr)
\quad t \geq 0.
\end{split}
\end{equation*}
By performing essentially the same computations as in the derivation of equation \eqref{254}, It\^o's formula applied to the continuous process $(\mathrm{e}^{\epsilon \Delta}X_t)_{t \geq 0}$, for a given $\epsilon >0$, yields:
\begin{align*}
	\|\mathrm{e}^{\epsilon \Delta}(X_{t}-X_t')\|_2^2 &+ 2\int_{0}^{t} \|D_{\mathbf{x}}\mathrm{e}^{\epsilon \Delta}(X_{s}-X_s')\|_2^2 \mathrm{d}s = 2\int_{0}^{t} \mathrm{e}^{2\epsilon \Delta} (X_{s}-X_s')\cdot\mathrm{d}(\zeta_s-\zeta_s'), \quad t \geq 0,
\end{align*}
the integral on the right-hand side being interpreted by means of Theorem 
\ref{thm:integral:varsigma}. And then, expanding the right-hand side, we obtain 
\begin{equation}
\label{eq:stability:000}
\begin{split}
&\|\mathrm{e}^{\epsilon \Delta}(X_{t}-X_t')\|_2^2 + 2\int_{0}^{t} \|D_{\mathbf{x}} \mathrm{e}^{\epsilon \Delta}(X_{s}-X_s')\|_2^2 \mathrm{d}s 
\\
&\leq 2\int_{0}^{t} \mathrm{e}^{2\epsilon \Delta} ( X_{s} - X _s') \cdot 
\mathrm{d}(\eta_s - \eta_s')  + \int_{0}^{t}\int_{\mathbb{S}}
\bigl[ \mathrm{e}^{2\epsilon \Delta}(D_{\mathbf{x}} X_s - D_{\mathbf{x}} X_s')\bigr](\mathbf{x})\left(\frac{1}{D_{\mathbf{x}} X_s(\mathbf{x})} - \frac{1}{D_{\mathbf{x}} X_s'(\mathbf{x})} \right) \mathrm{d}\mathbf{x}\mathrm{d}s.
\end{split}
\end{equation}
The first term on the right-hand side can be rewritten as 
\begin{equation*}
\begin{split}
&\int_0^t {\rm e}^{2 \epsilon \Delta} (X_s - X_s') \cdot \ud (\eta_s  - \eta_s')
 =
\int_0^t {\rm e}^{2 \epsilon \Delta} X_s \cdot \ud \eta_s + \int_0^t {\rm e}^{2 \varepsilon} X_s'\cdot \ud \eta_s' 
-
\int_0^t {\rm e}^{2 \epsilon \Delta } X_s 
\cdot \ud \eta_s'
- \int_0^t {\rm e}^{2 \epsilon \Delta } X_s' \cdot \ud \eta_s.
\end{split}
\end{equation*}
By Theorem \ref{thm:integral:varsigma}, the last two terms give a negative contribution. And, by item \textit{(iv)} in Definition \ref{Defi 317}, the first two ones tend to $0$ in probability as $\epsilon$ tend to $0$. This says that
\begin{equation}
\label{eq:conv:proba:1}
    \forall \delta >0, 
    \quad \lim_{\epsilon \searrow 0}
    {\mathbb P}
    \left( \left\{ \int_0^t {\rm e}^{2 \epsilon \Delta} (X_s - X_s') \cdot \ud (\eta_s  - \eta_s') \geq \delta \right\} 
    \right) =0.
\end{equation}
By item $(iv)$ in Definition \ref{Defi 317}, we also know that 
\begin{equation}
\label{eq:conv:proba:2}
\forall \delta >0, \quad \lim_{\epsilon \searrow 0}
    {\mathbb P}
    \left( \left\{ 
    \int_{0}^{t}\int_{\mathbb{S}}
    \biggl( 
\frac{ [\mathrm{e}^{2\epsilon \Delta}D_{\mathbf{x}} X_s ](\mathbf{x})}{D_{\mathbf{x}}X_s({\mathbf x})} + 
\frac{ [\mathrm{e}^{2\epsilon \Delta}D_{\mathbf{x}} X_s' ](\mathbf{x})}{D_{\mathbf{x}} X_s'({\mathbf x})}
\biggr)
\mathrm{d}\mathbf{x}\mathrm{d}s
    \geq 2t+ \delta \right\} 
    \right) =0.
    \end{equation}
Finally, using the fact that $[{\rm e}^{2 \epsilon \Delta} D_{\mathbf{x}} X_s]({\mathbf x})/D_{\mathbf{x}} X_s'({\mathbf x})$ and $[{\rm e}^{2 \epsilon \Delta} D_{\mathbf{x}} X_s']({\mathbf x})/D_{\mathbf{x}} X_s({\mathbf x})$ are non-negative together with Fatou's lemma, we get, up to a subsequence $(\epsilon_n)_{n \geq 1}$
{(converging to $0$)} along which $[{\rm e}^{2 \epsilon_n \Delta} D_{\mathbf{x}} X_s]({\mathbf x})/D_{\mathbf{x}} X_s'({\mathbf x})$ and $[{\rm e}^{2 \epsilon_n \Delta} D_{\mathbf{x}} X_s']({\mathbf x})/D_{\mathbf{x}} X_s({\mathbf x})$ converge (almost everywhere in $(\omega,s,{\mathbf x})$ to $D_{\mathbf{x}} X_s({\mathbf x})/D_{\mathbf{x}} X_s'({\mathbf x})$ and $D_{\mathbf{x}} X_s'({\mathbf x})/D_{\mathbf{x}} X_s({\mathbf x})$ respectively, for all $t \geq 0$, ${\mathbb P}-$almost surely, 
\begin{equation*}
\begin{split}
&\int_{0}^{t}\int_{\mathbb{S}}
    \biggl( 
\frac{  D_{\mathbf{x}} X_s' (\mathbf{x})}{D_{\mathbf{x}} X_s({\mathbf x})} + 
\frac{ D_{\mathbf{x}} X_s (\mathbf{x})}{D_{\mathbf{x}} X_s'({\mathbf x})}
\biggr)
\mathrm{d}\mathbf{x}\mathrm{d}s
\\
&=\int_{0}^{t}\int_{\mathbb{S}}
\lim_{n \rightarrow  \infty}
    \biggl( 
\frac{ [\mathrm{e}^{2\epsilon_n\Delta}D_{\mathbf{x}} X_s' ](\mathbf{x})}{D_{\mathbf{x}} X_s({\mathbf x})} + 
\frac{ [\mathrm{e}^{2\epsilon_n \Delta}D_{\mathbf{x}} X_s ](\mathbf{x})}{D_{\mathbf{x}} X_s'({\mathbf x})}
\biggr)
\mathrm{d}\mathbf{x}\mathrm{d}s
\\
&\leq \liminf_{n \rightarrow  \infty} \int_{0}^{t}\int_{\mathbb{S}}
    \biggl( 
\frac{ [\mathrm{e}^{2\epsilon_n\Delta}D_{\mathbf{x}} X_s' ](\mathbf{x})}{D_{\mathbf{x}} X_s({\mathbf x})} + 
\frac{ [\mathrm{e}^{2\epsilon_n \Delta}D_{\mathbf{x}} X_s ](\mathbf{x})}{D_{\mathbf{x}} X_s'({\mathbf x})}
\biggr)
\mathrm{d}\mathbf{x}\mathrm{d}s.
\end{split}
\end{equation*}
By \eqref{eq:conv:proba:1} and \eqref{eq:conv:proba:2}, we can change the definition of the sequence $(\epsilon_n)_{n \geq 1}$ in such a way that, ${\mathbb P}-$almost surely, 
\begin{equation*}
\begin{split}
&\limsup_{n \rightarrow + \infty} 
\biggl[ 
2 \int_0^t \mathrm{e}^{2 \epsilon_n \Delta} (X_s - X_s') \cdot \ud (\eta_s  - \eta_s')
+     \int_{0}^{t}\int_{\mathbb{S}}
    \biggl( 
\frac{ [\mathrm{e}^{2\epsilon_n \Delta}D_{\mathbf{x}} X_s ](\mathbf{x})}{D_{\mathbf{x}} X_s({\mathbf x})} + 
\frac{ [\mathrm{e}^{2\epsilon_n \Delta}D_{\mathbf{x}} X_s' ](\mathbf{x})}{D_{\mathbf{x}} X_s'({\mathbf x})}
\biggr)
\mathrm{d}\mathbf{x}\mathrm{d}s
\biggr] \leq 2 t.
\end{split}
\end{equation*}
Returning back to \eqref{eq:stability:000} with $\epsilon=\epsilon_n$ therein, and then letting 
$n$ tend to $+\infty$, we get, for all $t \geq 0$, ${\mathbb P}-$almost surely, 
\begin{equation*}
\begin{split}
&\|X_{t}-X_t'\|_2^2 + 2\int_{0}^{t} \|D_{\mathbf{x}} (X_{s}-X_s')\|_2^2 \mathrm{d}s 
+
\int_{0}^{t}\int_{\mathbb{S}}
    \biggl( 
\frac{  D_{\mathbf{x}} X_s' (\mathbf{x})}{D_{\mathbf{x}} X_s({\mathbf x})} + 
\frac{ D_{\mathbf{x}} X_s (\mathbf{x})}{D_{\mathbf{x}} X_s'({\mathbf x})}
\biggr)
\mathrm{d}\mathbf{x}\mathrm{d}s
 \leq  2t,
\end{split}
\end{equation*}
which we rewrite as 
\begin{equation*}
\begin{split}
&\|X_{t}-X_t'\|_2^2 + 2\int_{0}^{t} \|D_{\mathbf{x}} (X_{s}-X_s')\|_2^2 \mathrm{d}s 
+
\int_0^t \int_{\mathbb S}
    \biggl\vert  \sqrt{\frac{D_{\mathbf{x}} X_s({\mathbf x})}{D_{\mathbf{x}} X_s'({\mathbf x})}}- \sqrt{\frac{D_{\mathbf{x}} X_s'({\mathbf x})}{D_{\mathbf{x}} X_s({\mathbf x})}} \biggr\vert^2 {\mathrm d}{\mathbf x}  {\mathrm d}s 
 \leq  0.
\end{split}
\end{equation*}
This holds ${\mathbb P}-$ almost surely, for all $t \geq 0$. This proves uniqueness. 
    Inequality \eqref{ineq:4.32} can be established in the same way.
This completes the proof.
\end{proof}`

Combining Propositions \ref{Propo 5.13} and \ref{propo:4.17}
(and on the model of the proof of Theorem 1 in \cite{DelHam25}, based on Yamada-Watanabe's principle), we obtain the first main result of this paper:

\begin{thm}\label{Main 1}
    Given a probabilistic set-up $(\Omega,\mathfrak{F},{\mathbb F}= \{\mathfrak{F}_t\}_{t\geq 0},\mathbb{P},(W_t)_{t\geq 0})$ as defined in Subsection \ref{Notations}, and an ${\mathcal F}_0-$measurable initial condition $X_0$ with values in ${\mathcal U}^2({\mathbb S})$ satisfying ${\mathbb E}[\| X_0\|^p_2] < + \infty$ for all $p \geq 1$, there exists a unique solution $(X_t,{\eta}_t)_{t\geq 0}$ to the RSHE~\eqref{RHSE2}  in the sense of Definition \ref{Defi 317}.
For every $T>0$, the law of $(X_t,W_t)_{0 \le t \le T}$ over $\mathcal{C}([0,T],\mathcal{H}_{\mathrm{Sym}}^{-1}(\mathbb{S})\times L^2_{\rm Sym}(\mathbb{S}))$  is equal to the limit of the laws of the processes  $\{(X_t^{(h)},W_t)_{0 \le t \le T}\}_{{h \in T/{\mathbb N}^*}}$ given by \eqref{scheme0}, the space $\mathcal{C}([0,T],\mathcal{H}_{\mathrm{Sym}}^{-1}(\mathbb{S})\times L^2_{\rm Sym}(\mathbb{S}))$  being equipped with the topology of uniform convergence.
\end{thm}

\begin{remark}
\label{rem:lp:integrability}
By \eqref{254} and item $(iv)$
in Definition \ref{Defi 317} and with the same notations as in the above statement, we deduce that, with probability 1, for all $t \geq 0$, 
\begin{equation*}
\| X_t \|_2^2 + \int_0^t \|D_{\mathbf{x}} X_s \|^2_2 \ud s \leq \| X_0 \|_2^2
+ (1+ c_\lambda) t + 2 \int_0^t X_r \cdot \ud W_r. 
\end{equation*}
By a straightforward induction on the exponent, we get, for any $p \geq 1$,
\begin{equation*}
\forall t \geq 0, \quad {\mathbb E}
\bigl[ \sup_{s \in [0,t]} \| X_s \|_2^{2p} \bigr] 
+ {\mathbb E}
\biggl[ \biggl( \int_0^t \| D_{\mathbf{x}} X_s \|_2^2 \ud s \biggr)^p \biggr]
< + \infty. 
\end{equation*}
\end{remark}

We end up this subsection with an alternative characterization of the solutions to 
\eqref{RHSE2}:

\begin{prop}
\label{prop:RSHE:altenative:charac}
    Given a probabilistic set-up $(\Omega,\mathfrak{F},{\mathbb F}= \{\mathfrak{F}_t\}_{t\geq 0},\mathbb{P},(W_t)_{t\geq 0})$ as defined in Subsection \ref{Notations}, and an ${\mathcal F}_0-$measurable initial condition $X_0$ with values in ${\mathcal U}^2({\mathbb S} )$ satisfying ${\mathbb E}[\| X_0\|_2^p] < + \infty$ for all $p \geq 1$, a pair of  processes $(X_t,{\eta}_t)_{t \geq 0}$ solves the equation \eqref{RHSE2} initialized at $X_0$, if and only if
    item $(i)$ in Definition \ref{Defi 317} holds true and there exists a continuous 
    ${\mathbb F}-$adapted process $(\zeta_t)_{t \geq 0}$ with values in ${\mathcal H}_{\rm Sym}^{-2}({\mathbb S})$ such that 
\begin{enumerate}
\item[(i)] $({\zeta}_t)_{t\geq 0}$ starts from $0$ at $0$ and, with probability 1, for any $\varphi \in \mathcal{H}_{\mathrm{Sym}}^{2}(\mathbb{S})\cap \mathcal{U}^2(\mathbb{S})$ the path $t\mapsto\langle {\zeta}_t, \varphi\rangle$ is non-decreasing;
\item[(ii)]  with probability 1, for any $\varphi \in 
\mathcal{H}_{\mathrm{Sym}}^{2}(\mathbb{S})$, for all $t \geq 0$
\begin{align*}
\langle X_t, \phi \rangle =
\langle X_0,\phi \rangle   + \int_0^t \langle X_r,\Delta_{\mathbf{x}} \phi \rangle \mathrm{d}r + \langle {\zeta}_t,\phi\rangle +\langle W_t,\phi\rangle, \quad t \geq 0;
\end{align*}
\item[(iii)] for any $t\geq 0$,
\[ 
\lim_{\epsilon \searrow 0}
\mathbb{E} \biggl[   \int_{0}^{t}  {\rm e}^{\epsilon \Delta} X_r \cdot 
\mathrm{d} \zeta_r    \biggr]=\frac{t}2 ,
\]
with the integral being interpreted as in Theorem \ref{thm:integral:varsigma};
\item[(iv)] with probability 1, for any $\varphi \in {\mathcal{H}_{\mathrm{Sym}}^{2}(\mathbb{S}) \cap \mathcal{U}^2(\mathbb{S})}$, the process 
\begin{equation}
\label{eq:item:v:def:zeta}
\biggl( \langle \varphi,\zeta_t \rangle -
\frac12 \int_0^t \int_{\mathbb S} \frac{D_{\mathbf{x}} \varphi({\mathbf x})}{D_{\mathbf{x}} X_r({\mathbf x})} 
 \ud r{\mathrm d} {\mathbf x} \biggr)_{t \geq 0}
\end{equation}
is non-decreasing. 
\end{enumerate}
    And then, $(\eta_t)_{t \geq 0}$ and $(\zeta_t)_{t \geq 0}$ are related through the formula
    \begin{equation}
    \label{eq:relation:between:eta:zeta}
\eta_t = \zeta_t +\frac12 D_{\mathbf{x}} \Bigl( \frac1{D_{\mathbf{x}} X_t} \Bigr), \quad t \geq 0. 
    \end{equation}
    
    \end{prop}
\begin{proof}
{\ \ }
\vskip 4pt

\textit{First Step.} The necessary condition is established quite easily. Given a solution $(X_t,\eta_t)_{t \geq 0}$ in the sense of Definition \ref{Defi 317}, we can define $(\zeta_t)_{t \geq 0}$ through the relationship \eqref{eq:relation:between:eta:zeta}.  
Observing that, for any 
$\varphi \in {\mathcal H}^2_{\rm Sym}({\mathbb S})$
and any $0 \leq s \leq t$, 
\begin{equation*}
\biggl\vert \int_s^t \int_{\mathbb S} \frac{D_{\mathbf{x}} \varphi({\mathbf x})}{D_{\mathbf{x}} X_r({\mathbf x})} 
\ud {\mathbf x}\ud r 
\biggr\vert \leq \sup_{{\mathbf x} \in {\mathbb S}}
\vert D_{\mathbf{x}} \varphi({\mathbf x}) \vert \int_s^t \int_{\mathbb S} \frac1{\vert D_{\mathbf{x}} X_r({\mathbf x})\vert} \ud r \ud {\mathbf x} \leq c \| \varphi \|_{2,2} 
\int_s^t \int_{\mathbb S} \frac1{\vert D_{\mathbf{x}} X_r({\mathbf x})\vert}
 \ud r \ud {\mathbf x},
 \end{equation*}
for $c$ a universal constant, and using item $(i)$ in 
Definition \ref{Defi 317}, we deduce that $(\zeta_t)_{t \geq 0}$ has continuous trajectories with values in 
${\mathcal H}^{-2}_{\rm Sym}({\mathbb S})$.

Item $(iii)$ in the statement can be readily checked by means of item $(iv)$ in Definition \ref{Defi 317} (notice that, since
$\int_0^t {\rm e}^{\epsilon \Delta} X_r \cdot \ud \zeta_r \geq 0$, it holds
$\lim_{\epsilon \searrow 0} {\mathbb E}[ \vert \int_0^t {\rm e}^{\epsilon \Delta} X_r \cdot \ud \zeta_r \vert]=0$). Item $(iv)$ is obvious, since 
the generic item in \eqref{eq:item:v:def:zeta}
is equal to $(\langle \varphi,\eta_t \rangle)_{t \geq 0}$, with $(\eta_t)_{t \geq 0}$ as in Definition\ref{Defi 317}.
\vskip 4pt

\textit{Second Step.}
The proof of the converse is slightly more involved. Assume that $(X_t,\zeta_t)_{t \geq 0}$ satisfies item 
$(i)$ in Definition \ref{Defi 317} together with the
items $(i)$, $(ii)$, $(iii)$ and $(iv)$ in the statement. Define $(\eta_t)_{t \geq 0}$ as in \eqref{eq:relation:between:eta:zeta}. Proceeding as in the first step, we can prove that 
$(\eta_t)_{t \geq 0}$ has continuous trajectories with values in ${\mathcal H}^{-2}({\mathbb S})$. 
By $(iv)$ in the statement, $(\eta_t)_{t \geq 0}$ is non-decreasing. 

It remains to check that $(iv)$ in Definition \ref{Defi 317} is satisfied. We write 
\begin{equation*}
\begin{split}
\int_0^t {\rm e}^{\epsilon \Delta} X_r \cdot \ud \zeta_r &=
\int_0^t {\rm e}^{\epsilon \Delta} X_r \cdot \ud \eta_r 
+ 
\frac12 \int_0^t
\int_{\mathbb S}\frac{{\rm e}^{\epsilon \Delta} D_{\mathbf{x}} X_r({\mathbf x})}{D_{\mathbf{x}} X_r({\mathbf x})}
\ud {\mathbf x} \ud r
 \geq \frac12 \int_0^t
\int_{\mathbb S}\frac{{\rm e}^{\epsilon \Delta} D_{\mathbf{x}} X_r({\mathbf x})}{D_{\mathbf{x}} X_r({\mathbf x})}
\ud {\mathbf x} \ud r.
\end{split}
\end{equation*}
Using $(iii)$ in the statement and following
\eqref{eq:cv:to:t/2}, we deduce
\begin{equation*}
\lim_{\epsilon \searrow 0} {\mathbb E} \int_0^t
\int_{\mathbb S}
\frac{{\rm e}^{\epsilon \Delta} D_{\mathbf{x}} X_r({\mathbf x})}{D_{\mathbf{x}} X_r({\mathbf x})}
\ud {\mathbf x} \ud r = t,
\end{equation*}
which yields 
\begin{equation*}
{
\lim_{\epsilon \searrow 0}
{\mathbb E}
\int_0^t \textrm{\rm e}^{\epsilon \Delta} X_r \cdot \ud \eta_r 
=0.}
\end{equation*}
This is the first claim in item $(iv)$ of 
Definition \ref{Defi 317}. And then, following the derivation of \eqref{eq:cv:L1:t/2} in the proof of 
Proposition 
\ref{Propo 5.13}, we complete the proof of item $(iv)$ in Definition \ref{Defi 317}.
\end{proof}

\begin{remark}
\label{rem:30-06:1}
Similarly to \cite[Theorem 2]{DelHam25}, we could prove that the semigroup generated by \eqref{RHSE2}
is strong Feller and maps bounded functions, in positive time, onto Lipschitz continuous functions, with a Lipschitz constant blowing up at the same rate as time tends to zero.
The result could be stated in the same way. For this reason, we omit it here, and we only provide the main ideas that would allow us to adapt the proof.

One of the main points in \cite{DelHam25} is to translate, in a dynamic way, the initial condition of \eqref{RHSE2}. Here, for an integer $K \in {\mathbb N}^*$, we denote by $E^K$ the vector subspace of $L^2_{\rm Sym}({\mathbb S})$ spanned by $e_0,\ldots,e_K$; we identify it with ${\mathbb R}^{K+1}$ and equip it with the Lebesgue measure $\textrm{\rm Leb}_{K+1}$. For a fixed time horizon $T>0$, a parameter $\vartheta \in {\mathbb R}$, and a direction $u \in {\mathbb R}^K$, we define the following translation and its composition with the rearrangement:
\begin{equation}
    {\mathbf y} \in E^K \mapsto \mathbf{y}_t(u,\vartheta):= \mathbf{y} +\vartheta \frac{T-t}{T}u,\quad \text{ and } \quad  {\mathbf y} \in E^K \mapsto \mathbf{y}_t^{*}(u,\vartheta):= \left(\mathbf{y} +\vartheta \frac{T-t}{T}u\right)^{*}.
\end{equation}
A key point in the proof is to express the dynamics of $(X_t^{\mathbf{y}_t^{*}(u,\vartheta)})_{t\in [0,T]}$, which we do here by using the reformulation of \eqref{RHSE2} provided by Proposition \ref{prop:RSHE:altenative:charac}. The main idea is therefore to work with the process $(\zeta_t)_{0 \le t \le T}$ rather than with the process $(\eta_t)_{t \in [0,T]}$. The key advantage is that the regularity properties of $(\zeta_t)_{0 \le t \le T}$ are here the same as those of the reflection process $(\eta_t)_{t \in [0,T]}$ in \cite{DelHam25}. The difficulty is that $(\zeta_t)_{t \in [0,T]}$ is not a reflected process. The purpose of this remark is therefore to explain how this affects the proof.

Following \cite[(5.7)]{DelHam25}, we let
\begin{equation}\label{eq:4.47:zeta}
    \tilde{\zeta}_t^{{\mathbf y},(u,\vartheta)} :=  {\zeta}_t^{\mathbf{y}_t^{*}(u,\vartheta)}  + \frac{\vartheta}{T} \int_0^t \partial_{\mathbf{y}} \zeta_s^{\mathbf{y}^*}\mathrm{d}s, \quad t \in [0,T].
\end{equation}
This makes it possible to derive the analogue of \cite[(5.13)]{DelHam25} but
with $(\eta_t)_{t \in [0,T]}$ replaced by $(\zeta_t)_{ t \in [0,T]}$. It now remains to prove that 
$\tilde{\zeta}^{{\mathbf y},(u,\vartheta)}$ satisfies the properties listed in Proposition 
\ref{prop:RSHE:altenative:charac}.
We start with the verification of item $(iv)$ therein. 

Following \cite[(5.20)]{DelHam25} and 
letting 
$(\tilde \zeta_t^{\Upsilon,(u,\vartheta)}:=\int_{E^K} 
\tilde \zeta_t^{{\mathbf y},(u,\vartheta)}
\ud \Upsilon(y))_{t \in [0,T]}$, 
for a smooth $\Upsilon$ on $E^K$ with compact support, 
we get, for any process $(z_t)_{t \in [0,T]}$ with values in 
$L^2_{\rm Sym}({\mathbb S})$, with probability 1, 
\begin{equation*}
\begin{split}
\int_0^t \langle {\rm e}^{\varepsilon \Delta} z_s, 
\ud \tilde{\zeta}_s^{\Upsilon,(u,\vartheta)}\rangle
&= \int_{E^K}  \Upsilon({\mathbf y})\int_0^t 
\int_{\mathbb S} 
 \frac{[{\rm e}^{\epsilon \Delta} D_{\mathbf{x}}   z_s]({\mathbf x}) }{D_{\mathbf{x}} X_s^{{\mathbf y},(u,\vartheta)}({\mathbf x})} \ud {\mathbf x} \ud s   \ud {\mathbf y}
+ \int_E  \int_0^t  \Upsilon({\mathbf y}_s^*(u,-\vartheta))
\langle {\rm e}^{\epsilon \Delta}   z_s, 
\ud \eta_s^{{\mathbf y}^*} \rangle \ud {\mathbf y}.
\end{split}
\end{equation*}
Using the fact that $\eta^{{\mathbf y}^*}$ is non-decreasing, and following the second and third steps in the proof of Proposition 12 in \cite{DelHam25}, we deduce that, for almost every ${\mathbf y} \in E^K$, 
$(\tilde \zeta^{{\mathbf y},(u,\vartheta)}_t + \int_0^t D_{\mathbf{x}}(1/ D_{\mathbf{x}}X_s^{{\mathbf y}^*(u,\vartheta)}) \ud s)_{0 \le t \le T}$ is non-decreasing, which corresponds to item $(iv)$ in Proposition \ref{prop:RSHE:altenative:charac}. 
It remains to address item $(iii)$ in Proposition \ref{prop:RSHE:altenative:charac}
. In fact, it suffices to prove that there exists a sequence $(\epsilon_n)_{n \geq 1}$, converging to $0$, such 
for almost every $({\mathbf y},\omega)  \in E^K \times \Omega$, for all 
$t \in [0,T]$,
\begin{equation*}
\lim_{n \rightarrow + \infty}
\int_0^t 
({\rm e}^{\epsilon_n \Delta} X_s^{{\mathbf y}_s^*(u,\vartheta)}) \cdot 
\ud \tilde{\zeta}_s^{{\mathbf y},(u,\vartheta)} = \frac{t}2.
\end{equation*}
By Dini's theorem, convergence is necessarily uniform (in $t$). In fact, since the left-hand side is non-decreasing (in time), it suffices to prove that there exists a sequence $(\epsilon_n)_{n \geq 1}$ converging to $0$ such that, for all rational $t \in [0,T]$, the above holds true for almost every 
$({\mathbf y},\omega) \in E^K \times \Omega$. And then, by diagonal extraction, one can fix a rational $t \in [0,T]$ and then prove that there exists a sequence {$(\epsilon_n)_{n \geq 1}$}, converging to $0$, such that the above holds true for almost every $({\mathbf y},\omega)\in E^K \times \Omega$. 
Following the proof of \cite[(5.20)]{DelHam25} and for $\Upsilon$ as above,
\begin{equation}
\label{eq:smoothing:1}
\begin{split}
\int_{E^K} \Upsilon(\mathbf{y}) \biggl( \int_0^t 
({\rm e}^{\epsilon \Delta} X_s^{{\mathbf y}_s^*(u,\vartheta)}) \cdot 
\ud \tilde{\zeta}_s^{{\mathbf y},(u,\vartheta)} \biggr) \ud {\mathbf y}
&= \int_{E^K} \biggl( \int_0^t \int_{\mathbb S} \Upsilon ({\mathbf y} ) \frac{[{\rm e}^{\epsilon \Delta}  D_{\mathbf{x}} X_s^{{\mathbf y}^*(u,\vartheta)}]({\mathbf x})}{D_{\mathbf{x}} X_s^{\mathbf{y}^*(u,\vartheta)}({\mathbf x})} \ud {\mathbf x} \ud s \biggr) \ud {\mathbf y}
\\
&\hspace{15pt}
+
\int_{E^K} \biggl( \int_0^t \biggl[ \Upsilon\Bigl(\mathbf{y}_s(u,-\vartheta)\Bigr) {\rm e}^{\epsilon \Delta}  X_s^{{\mathbf y}^*}\biggr] \cdot 
\ud {\eta}_s^{{\mathbf y}^*} \biggr) \ud {\mathbf y}.
\end{split}
\end{equation} 
Arguing as the proof of \cite[Proposition 13]{DelHam25},
\begin{equation*}
\lim_{\epsilon \rightarrow 0} {\mathbb E} \int_{E^K} \biggl( \int_0^T \biggl[ \Upsilon\Bigl(\mathbf{y}_s(u,-\vartheta)\Bigr) {\rm e}^{\epsilon \Delta}  X_s^{{\mathbf y}^*}\biggr] \cdot 
\ud {\eta}_s^{{\mathbf y}*} \biggr) \ud {\mathbf y}
=0,
\end{equation*}
which permits us to handle the last term in 
\eqref{eq:smoothing:1}.
In order to handle the first term on the right-hand side therein, we proceed as follows. Assuming that $\Upsilon$ 
is non-negative valued, we get, for 
an arbitrary mesh $(0=s_0<s_1<\cdots<s_m=T)$ of $[0,t]$ (for a certain integer 
$m \in {\mathbb N}^*$), 
\begin{equation*}
\begin{split}
\int_{E^K} \biggl( \int_0^t \biggl[ \Upsilon\Bigl(\mathbf{y}_s(u,-\vartheta)\Bigr) {\rm e}^{\epsilon \Delta}  X_s^{{\mathbf y}^*}\biggr] \cdot 
\ud {\zeta}_s^{{\mathbf y}^*} \biggr) \ud {\mathbf y}
&= \sum_{i=0}^{m-1} \int_{E^K}
\biggl( 
\int_{s_i}^{s_{i+1}}
\biggl[ \Upsilon\Bigl(\mathbf{y}_r(u,-\vartheta)\Bigr) {\rm e}^{\epsilon \Delta}  X_r^{{\mathbf y}^*}\biggr] \cdot 
\ud {\zeta}_r^{{\mathbf y}^*} \biggr) \ud {\mathbf y}
\\
&\leq \sum_{i=0}^{m-1} \int_{E^K} 
\sup_{r \in [s_i,s_{i+1}]}
\Upsilon\Bigl(\mathbf{y}_r(u,-\vartheta)\Bigr) 
\biggl(
\int_{s_i}^{s_{i+1}}
  {\rm e}^{\epsilon \Delta}  X_r^{{{\mathbf y}^*}} \cdot 
\ud {\zeta}_r^{{\mathbf y}^*} \biggr) \ud {\mathbf y}
\end{split}
\end{equation*}
Using (iii) in Proposition \ref{prop:RSHE:altenative:charac}, we get 
\begin{equation*}
\begin{split}
&\limsup_{\epsilon \searrow 0}
{\mathbb E} \int_{E^K} \biggl( \int_0^t \biggl[ \Upsilon\Bigl(\mathbf{y}_s(u,-\vartheta)\Bigr) {\rm e}^{\epsilon \Delta}  X_s^{{\mathbf y}^*}\biggr] \cdot 
\ud {\zeta}_s^{{\mathbf y}^*}\biggr) \ud {\mathbf y} \leq \frac12 \sum_{i=0}^{m-1} (s_{i+1}-s_i)
{\mathbb E}
\int_{E^K} 
 \sup_{r \in [s_i,s_{i+1}]} \Upsilon\Bigl(\mathbf{y}_r(u,-\vartheta)\Bigr) \ud {\mathbf y}.
\end{split}
\end{equation*}
Letting the step of the mesh tend to $0$
and following \cite[(5.20)]{DelHam25},  we get 
\begin{equation*}
\limsup_{\epsilon \searrow 0}
\int_{E^K} \Upsilon(\mathbf{y}) {\mathbb E}  \biggl( \int_0^t 
({\rm e}^{\epsilon \Delta} X_s^{{\mathbf y}^*,(u,\vartheta)}) \cdot 
\ud \tilde{\zeta}_s^{{\mathbf y},(u,\vartheta)} \biggr) \ud {\mathbf y} \leq \frac{t}2 \int_{E^K} \Upsilon({\mathbf y}) \ud {\mathbf y}.
\end{equation*}
This implies 
\begin{equation}
\label{eq:smoothing:2}
\forall \delta >0, \quad 
\lim_{\epsilon \searrow 0}
\textrm{\rm Leb}_{K}
\left\{ {\mathbf y} \in E^K : {\mathbb E} \int_0^t 
({\rm e}^{\epsilon \Delta} X_s^{{\mathbf y}_s^*(u,\vartheta)}) \cdot 
\ud \tilde{\zeta}_s^{{\mathbf y},(u,\vartheta)} \geq \frac{t}{2}
+\delta 
\right\}=0.
\end{equation}
Moreover,
from \eqref{eq:smoothing:1} we recall that  
\begin{equation*}
\int_0^t 
({\rm e}^{\varepsilon \Delta} X_s^{{\mathbf y}_s^*(u,\vartheta)}) \cdot 
\ud \tilde{\zeta}_s^{{\mathbf y},(u,\vartheta)} \geq \frac12 
\int_0^t \int_{\mathbb S}  \frac{{\rm e}^{\varepsilon \Delta}  D_{\mathbf{x}} X_s^{{\mathbf y}^*(u,\vartheta)}}{D_{\mathbf{x}} X_s^{{\mathbf y}^*(u,\vartheta)}} \ud {\mathbf x} \ud s.
\end{equation*}
And then, following the proof of Proposition \ref{prop:RSHE:altenative:charac}, we get for almost every ${\mathbf y} \in E^K$,
\begin{equation*}
{\liminf_{\epsilon \rightarrow 0} 
{\mathbb E} \biggl[ \biggl(
\int_0^t 
({\rm e}^{\epsilon \Delta} X_s^{{\mathbf y}_s^*(u,\vartheta)}) \cdot 
\ud \tilde{\zeta}_s^{{\mathbf y},(u,\vartheta)} -
 \frac{t}2
 \biggr)_+ \biggr] \geq 0.}
\end{equation*}
Together with 
\eqref{eq:smoothing:2}, we get, for any $\delta >0$, 
\begin{equation*}
\lim_{\epsilon \searrow 0}
\textrm{\rm Leb}_{K}
\left\{ {\mathbf y} \in E^K :
{\mathbb E} \biggl[ \biggl\vert \int_0^t 
({\rm e}^{\epsilon \Delta} X_s^{{\mathbf y}_s^*(u,\vartheta)}) \cdot 
\ud \tilde{\zeta}_s^{{\mathbf y},(u,\vartheta)} - \frac{t}{2}
\biggr\vert \biggr]
 \geq \delta 
\right\}=0.
\end{equation*}
This completes the proof.
\end{remark}

\section{Stochastic Fokker-Planck equation satisfied by the density} 
\label{se:7}

In this section, we establish the existence of a density for the marginal laws, on ${\mathbb R}$, induced  by the solution
$(X_t)_{t \in [0,T]}$ to equation \eqref{eq:4.28}.
We also show that this density satisfies a mollified version of the Dean--Kawasaki equation.

\subsection{Existence of a density}

The existence of a density is a direct consequence of the presence of the additional convolution term in the scheme \eqref{scheme1} or, equivalently, of the presence of the additional term arising from the entropy in \eqref{eq:4.28}.

At first sight, the existence of a density may not seem surprising, given the regularizing effect of the heat equation (the heat appearing in the form of a convolution in \eqref{scheme1}
and through the derivative of the entropy in
\eqref{eq:4.28}).
Nevertheless, the result here is not entirely comparable to what is known in the standard theory of 
parabolic
partial differential equations: indeed, the law of $X_t$ 
on 
${\mathbb R}$ admits a density for every $t>0$, but this density has compact support only.
The mass therefore propagates at finite speed.
This phenomenon, which we consider rather subtle, stems from the competition between, on the one hand, the Laplacian regularizing the quantile function and, on the other hand, the heat flow generated by the additional convolution in \eqref{scheme1}.
The regularity of the quantile function forces the support of the density to be compact; at the same time, this helps guaranty the regularizing effect of the semigroup generated by \eqref{eq:4.28}, as explained in Remark \ref{rem:item (iv):Def:Solution}.
On the other hand, the additional convolution in \eqref{scheme1} tends to spread out the support of the law of $X_t$ (on  
{${\mathbb R}$}), while ensuring the existence of a density.
\begin{thm}\label{Prop5.1}
   Let $(X_t,\eta_t)_{t\geq 0}$ be the unique solution to \eqref{eq:4.28} in the sense of  Definition \ref{Defi 317}. Then, $\mathbb{P}\mathrm{-a.s.}$, for almost every $t > 0$, the law ${\mu_t}$ of ${\bf x} \in {\mathbb S} \mapsto X_t({\bf x})$ (with ${\mathbb S}$ equipped with $\textrm{\rm Leb}_{\mathbb S}$) has a square integrable density $p_t(\cdot)$. {For all $T >0$,} it satisfies
\begin{align}
\label{eq:08-07:10}
\mathbb{E}\int_0^T\int_{\mathbb{R}}{|p_{t}(z)|^2}\mathrm{d}z\mathrm{d}t < + \infty. %
\end{align}
Moreover, 
{${\mathbb P}$-almost surely and for almost every $t > 0$},
the density $p_t$ is supported by 
$[X_t(0),X_t(1/2)]$ and is almost everywhere positive on 
$(X_t(0),X_t(1/2))$.
{For $\textrm{\rm Leb}_{\mathbb R}$-almost every $z \in {\mathbb R}$, $X_t$ is differentiable at $\tilde{F}_{\mu_t}(z)$, where $\tilde{F}_{\mu_t}$ is the cumulative distribution function of $\mu_t$, and for almost every 
$z \in (X_t(0),X_t(1/2))$,
\begin{equation}
\label{eq:08-07:11}
D_{\mathbf x} X_t \bigl(\tfrac12 \tilde F_{\mu_t}(z) \bigr) = \frac1{2 p_t(z)}.
\end{equation}
In particular, 
for all $T>0$, 
\begin{equation}
{\mathbb E}
\int_0^T \int_{\mathbb R}
{\mathbbm 1}_{(X_t(0),X_t(1/2))}(z)
\frac1{p_t(z)}
\ud z  \ud t =
\frac14 {\mathbb E}
\int_0^T \int_{\mathbb S}
\vert D_{\mathbf{x}}
X_t({\mathbf x}) \vert^2
\ud {\mathbf x}  \ud t
< + \infty.
\label{eq:08-07:12}
\end{equation}
Lastly, 
${\mathbb P}$-almost surely and for almost every $t>0$, 
$X_t(\cdot /2)$ is a continuous one-to-one mapping from $(0,1)$ onto $(X_t(0),X_t(1/2))$, with $\tilde F_{\mu_t}$ as inverse. 
}
\end{thm}
\begin{proof} {Throughout the proof, the value of $T>0$ is fixed.}
\vskip 4pt

\textit{First Step.} 
{We proceed as in the proof of Corollary \ref{cor:E_eps}.}
Let $\varsigma \in (0,1/4)$ and $g$ be a smooth symmetric density on ${\mathbb R}$ with support included in $(-\varsigma,\varsigma)$. Then, for ${\mathbf x} \in (\varsigma,1/2-\varsigma) + {\mathbb Z}$
\begin{equation*}
  \int_{\mathbb R} \frac1{D_{\mathbf{x}} X_r({\mathbf x}-y)}
  g(y)
\ud y 
\geq 
\biggl(   \int_{\mathbb R} D_{\mathbf{x}} X_r({\mathbf x} - y) g(y) \ud y \biggr)^{-1},
\end{equation*}
where we used Jensen's inequality together with the fact that, for almost every $y \in {\mathbb R}$, $(g(y)>0)  \Rightarrow y \in (-\varsigma,\varsigma) \Rightarrow {\mathbf x} + y \in (0,1/2) + {\mathbb Z} \Rightarrow D_{\mathbf{x}} X_r({\mathbf x}+y)\geq 0$. Proceeding in a similar way when ${\mathbf x} \in (-1/2+\varsigma,-\varsigma)+ {\mathbb Z}$, we deduce that, for 
${\mathbf x} \in (\varsigma,1/2-\varsigma) + {\mathbb Z}/2$,
\begin{equation*}
\int_{\mathbb R}
\frac1{\vert D_{\mathbf{x}} X_r({\mathbf x} - y) \vert} 
g(y) \ud y 
\geq \frac1{\vert (D_{\mathbf{x}} X_r \star g)({\mathbf x})\vert}.
\end{equation*}
And then, 
\begin{equation*}
\begin{split}
\int_{(\varsigma,1/2-\varsigma) + {\mathbb Z}/2} \frac1{\vert ( D_{\mathbf{x}} X_r \star g )({\mathbf x})\vert} \ud {\mathbf x} &\leq \int_{\mathbb S}
\int_{\mathbb R} \frac1{\vert D_{\mathbf{x}} X_r({\mathbf x} - y) \vert} g(y) 
\ud {\mathbf x}  
\ud y 
\\
&=
\int_{\mathbb R}
\biggl( \int_{\mathbb S}
 \frac1{\vert D_{\mathbf{x}} X_r({\mathbf x} - y) \vert} 
 \ud {\mathbf x} \biggr) g(y) 
\ud y
=
\int_{\mathbb S} \frac1{\vert D_{\mathbf{x}} X_r({\mathbf x}) \vert}
\ud {\mathbf x}. 
\end{split}
\end{equation*}
From  \eqref{ineq:420}, we deduce that
\begin{align}\label{ineq:4.21}
 \mathbb{E}\int_{0}^{T} \int_{(\varsigma,1/2-\varsigma)+{\mathbb Z}/2}  \frac{1}{|D_{\mathbf{x}}({X}_{r}\star g)(\mathbf{x})|} \mathrm{d}\mathbf{x}  \mathrm{d}r \leq C,
\end{align}
for a constant $C$ independent of $g$ and $\varsigma$. 

Set, for all $t\in [0,T],$ $Y_t^\varsigma := X_t\star g$ and denote by $\mu_t^{\varsigma}$ its law, so that, with the same notation as in \eqref{inverse}, $Y_t^\varsigma = F_{\mu_t^{\varsigma}}^{-1}$.
{Because, almost surely and for almost every 
$t \in [0,T]$, $X_t$ takes values in 
$\mathcal{H}_{\mathrm{Sym}}^{1}(\mathbb{S})$
and is thus continuous (and therefore) bounded on ${\mathbb S}$, 
the process $(Y_t^{\varsigma})_{0 \le t \le T}$  takes values,
almost surely and for almost every $t \in [0,T]$, in the space of $\mathcal{C}^1$ quantile functions on ${\mathbb S}$}.
With probability 1, for almost every $t \in [0,T]$, 
$\textrm{\rm Leb}_{\mathbb S}(\{ {\mathbf x} \in {\mathbb S} : \vert D_{\mathbf{x}} X_t({\mathbf x}) \vert >0\})=1$ {(see the first item in Definition \ref{Defi 317})},  which implies
that $D_{\mathbf{x}} Y^{\varsigma}_t({\mathbf x})>0$, for ${\mathbf x} \in (\varsigma,1/2-\varsigma) + {\mathbb Z}$,
and $D_{\mathbf{x}} Y_t^{\varsigma}({\mathbf x}) <0$ for ${\mathbf x} \in (-1/2+\varsigma,-\varsigma)+ {\mathbb Z}$.
For ${\mathbf x} = x+{\mathbb Z}$, with $x \in (0,\varsigma]$, 
\begin{equation*}
    \begin{split}
&D_{\mathbf{x}} Y_t^{\varsigma}({\mathbf x}) 
\\
&= \int_{\mathbb R} D_{\mathbf{x}} X_t({\mathbf x}-y) g(y) 
\ud y =   \int_{\mathbb R} D_x X_t(x-y) g(y) 
\ud y
\\
&= \int_{\mathbb R} D_x X_t(y) g(x-y) \ud y
= \int_{-1/2}^{1/2} D_{x} X_t(y) g(x-y) \ud y 
= \int_0^{1/2} D_x X_t(y) \bigl( g(x-y) - g(x+y) \bigr)\ud y. 
    \end{split}
\end{equation*}
Assuming without any loss of generality that $g$ is non-increasing on $(0,+\infty)$ and strictly decreasing on 
$(0,\varsigma)$, we deduce that $ DY_t^{\varsigma}({\mathbf x})>0$. Similarly, 
$DY_t^{\varsigma}({\mathbf x})<0$ for ${\mathbf x} = x + {\mathbb Z}$ with $x \in [-\varsigma,0)$. This means that, with probability 1, for almost every 
$t \in [0,T]$, $Y_t^{\varsigma}$ is strictly increasing on $(0,1/2)+{\mathbb Z}$ (and strictly decreasing on $(-1/2,0)$), which makes it possible to consider its inverse. Equivalently, the quantile function $\tilde{F}_{\mu_t^\varsigma}^{-1}$ (defined on $(0,1)$) is strictly increasing; its (true) inverse is $\tilde F_{\mu_t^\varsigma}$, the cumulative distribution function of $\mu_t^\varsigma$ . 
Since $D_{\mathbf{x}}Y_t^{\varsigma}>0$ on $(0,1/2)$, we deduce that 
$D_{\mathbf{x}} \tilde{F}_{\mu_t^{\varsigma}}^{-1}>0$ on $(0,1)$. Therefore, 
$\tilde F_{\mu_t^{\varsigma}}$ is continuously differentiable on $(a^{\varsigma}_t,b^{\varsigma}_t):=(\tilde F^{-1}_{\mu_t^{\varsigma}}(0),\tilde F^{-1}_{\mu_t^{\varsigma}}(1))$. The function 
\begin{equation*}
p_t^{\varsigma} : y \in {\mathbb R} \mapsto  \left\{\begin{array}{ll}
\frac{\ud}{\ud y} \tilde F_{\mu_t^{\varsigma}}(y), &\quad y \in (a_t^\varsigma,b_t^\varsigma),
\\
0 &\quad y \not \in (a_t^\varsigma,b_t^\varsigma)
\end{array}
\right.
\end{equation*}
is the density of 
$\mu_t^\varsigma$ (of course, 
it  is random). 

Hence,  by \eqref{inverse}, we have 
\begin{align*}
\int_{(\varsigma,1/2-\varsigma)+{\mathbb Z}/2}\frac{1}{\vert DY_t^{\varsigma}(\mathbf{x})\vert}\mathrm{d}\mathbf{x} &= - \int_{-1/2+\varsigma}^{-\varsigma} \frac{1}{D{F}_{\mu_t^{\varsigma}}^{-1}(x)}\mathrm{d}{x} + \int_\varsigma^{1/2-\varsigma}\frac{1}{D{F}_{\mu_t^{\varsigma}}^{-1}({x})}\mathrm{d}{x}
\\
&= \int_{-1/2+\varsigma}^{-\varsigma} \frac{1}{2D\tilde{F}_{\mu_t^{\varsigma}}^{-1}(-2{x})}\mathrm{d}{x} + \int_\varsigma^{1/2-\varsigma}\frac{1}{2D\tilde{F}_{\mu_t^{\varsigma}}^{-1}(2{x})}\mathrm{d}{x}
= \frac{1}{2}\int_{2 \varsigma}^{1-2 \varsigma} \frac{1}{D\tilde{F}_{\mu_t^{\varsigma}}^{-1}({y})}\mathrm{d}{y}.
\end{align*}
Using the change of variable $\tilde{F}_{\mu_t^{\varsigma}}^{-1}({y})= z \Leftrightarrow {y} = \tilde{F}_{\mu_t^{\varsigma}}(z) $,  we obtain 
\begin{align*}
     \int_{\tilde{F}_{\mu_t^{\varsigma}}^{-1}(2 \varsigma)}^{\tilde{F}_{\mu_t^{\varsigma}}^{-1}(1-2\varsigma)}  | p_t^{\varsigma}(z)|^2\mathrm{d}z =
     \int_{\tilde{F}_{\mu_t^{\varsigma}}^{-1}(2 \varsigma)}^{\tilde{F}_{\mu_t^{\varsigma}}^{-1}(1-2\varsigma)}  |D\tilde{F}_{\mu_t^{\varsigma}}(z)|^2\mathrm{d}z = \int_{2 \varsigma}^{1-2 \varsigma}\frac{1}{D\tilde{F}_{\mu_t^{\varsigma}}^{-1}(y)}\mathrm{d}{y} <\infty, \,\,\, \mathrm{d}t\otimes\mathrm{d}\mathbb{P}\mathrm{-a.s.}
\end{align*}
From \eqref{ineq:4.21}, we get
\begin{equation}
\label{eq:bound:L2bound:density}
{\mathbb E}
\int_0^T \int_{{\mathbb R}} {\mathbbm 1}_{(2 \varsigma,1- 2 \varsigma)}(  \tilde F_{\mu_t^\varsigma}(z)) 
\vert p_t^{\varsigma}(z) \vert^2 \ud z   \ud t \leq C. 
\end{equation}
{\ \ }
\vskip 4pt

\textit{Second Step.} Inequality 
\eqref{eq:bound:L2bound:density}says that the family $p^{\varsigma} {\mathbbm 1}_{(2 \varsigma,1- 2 \varsigma)}(\tilde F_{\mu_t^{\varsigma}})$ is weakly compact in 
$L^2(\Omega \times [0,T] \times {\mathbb R})$ (equipped with the product measure ${\mathbb P} \otimes \textrm{\rm Leb}_{\mathbb S} \otimes \textrm{\rm Leb}_{\mathbb R}$). We call 
$p$ a limit point. 

At the same time, we know that, for any bounded and continuous  function $f : {\mathbb R} \rightarrow {\mathbb R}$ and any bounded and jointly measurable 
process $(H_t)_{t \in [0,T]}$, 
\begin{equation*}
\begin{split}
&{\mathbb E}
\biggl[
\int_0^T 
H_t \biggl( \int_{\mathbb R}
 f(z) \ud
\mu_t^{\varsigma} (z) \biggr) \ud t 
\biggr] 
 = {\mathbb E}
\biggl[
\int_0^T H_t \biggl( \int_{\mathbb R}
f(z) p_t^\varsigma(z) \ud
z \biggr)  \ud t 
\biggr] 
\\
&= {\mathbb E}
\biggl[
\int_0^T 
H_t
\biggl( 
\int_{\mathbb R}
f(z) 
{\mathbbm 1}_{(2 \varsigma,1-2 \varsigma)}(\tilde F_{\mu_t^\varsigma}(z))
p_t^\varsigma(z) \ud
z \biggr) \ud t 
\biggr] +
{\mathbb E}
\biggl[
\int_0^T 
H_t
\biggl( \int_{\mathbb R}
f(z) 
{\mathbbm 1}_{{\mathbb R} \setminus (2 \varsigma,1-2 \varsigma)}(\tilde F_{\mu_t^\varsigma}(z))
p_t^\varsigma(z) \ud
z \biggr) \ud t 
\biggr]
\\
&= : T_1(\varsigma) + T_2(\varsigma). 
\end{split}
\end{equation*}
Since $(H_t)_{t \in [0,T]}$ and $f$ are bounded, we can find a constant $C$, independent of 
$\varsigma$ such that 
\begin{equation*}
\begin{split}
& \vert T_2(\varsigma) 
\vert 
 \leq C {\mathbb E} \int_0^T \int_{-\infty}^{\tilde F_{\mu_t^\varsigma}^{-1}(2 \varsigma)}
  p_t^\varsigma(z) \ud
z  \ud t +  {\mathbb E} \int_0^T \int_{\tilde F_{\mu_t^\varsigma}^{-1}(1-2 \varsigma)}^{+\infty}
  p_t^\varsigma(z) \ud
z  \ud t
\leq 4 C T \varsigma. 
\end{split}
\end{equation*}
This says that $\lim_{\varsigma \searrow 0} T_2(\varsigma)=0$.  Observe further that, with probability 1, $\mu_t^{\varsigma} := \mathrm{Leb}_{\mathbb{S}}\circ(X_t\star g)^{-1}$ converges in law to $\mu_t:= \mathrm{Leb}_{\mathbb{S}}\circ(X_t)^{-1}$  as $\varsigma\searrow 0$, for almost every $t \in [0,T]$.  We deduce that 
\begin{equation*}
\lim_{\varsigma \searrow 0}
{\mathbb E}
\biggl[
\int_0^T 
H_t 
\biggl( 
\int_{\mathbb R} f(z) \ud
\mu_t^{\varsigma} (z) \biggr) \ud t 
\biggr] =
{\mathbb E}
\biggl[
\int_0^T 
H_t 
\biggl( \int_{\mathbb R}
 f(z) \ud
\mu_t (z) \biggr) \ud t 
\biggr].
\end{equation*}
By weak convergence (up to a subsequence) of $p^\varsigma$ to $p$ (within the space $L^2(\Omega \times [0,T] \times {\mathbb R})$), we have 
\begin{equation*}
    \lim_{\varsigma \searrow 0} T_1(\varsigma) = {\mathbb E} \biggl[ \int_0^T 
      H_t  \biggl( \int_{\mathbb R} 
  f(z) p_t(z) \ud z \biggr) \ud t \biggr].
\end{equation*}
By combining the last three displays, we get 
\begin{equation*}
{\mathbb E}
\biggl[
\int_0^T 
H_t \biggl( \int_{\mathbb R} f(z) \ud
\mu_t (z) \biggr)\ud t 
\biggr] =
{\mathbb E} \biggl[ \int_0^T 
    H_t \biggl( \int_{\mathbb R} f(z) p_t(z) \ud z \biggr) \ud t \biggr],
\end{equation*}
which proves that, for any bounded and continuous function 
$f : {\mathbb R} \rightarrow {\mathbb R}$, with probability 1, for almost every $t \in [0,T]$, 
\begin{equation*}
\int_{\mathbb R} f(z) \ud \mu_t (z)= 
\int_{\mathbb R} f(z) p_t(z) \ud z.
\end{equation*}
By separability of the space ${\mathcal C}_0({\mathbb R})$ of continuous functions that tend to $0$ at infinity, we deduce that, with probability 1, for almost every $t \in [0,T]$, $\mu_t$ has $p_t$ as density. By lower semi-continuity of the square norm for the weak convergence, we deduce from \eqref{eq:bound:L2bound:density}
that 
\begin{equation*}
{\mathbb E}
\int_0^T \int_{\mathbb R}
{{\mathbbm 1}_{(0,1)}(\tilde F_{\mu_t}(z))}
\vert p_t(z) \vert^2 \ud z  \ud t \leq C.
\end{equation*}
{It is standard to check that 
$\int_{\mathbb R} {\mathbbm 1}_{\{\tilde F_{\mu_t}(z)=0\}} p_t(z) \ud z = 0$
and 
$\int_{\mathbb R} {\mathbbm 1}_{\{ \tilde F_{\mu_t}(z)=1\}} p_t(z) \ud z = 0$. We deduce that the indicator function in the above display can be easily removed. This gives the $L^2$ bound \eqref{eq:08-07:10}.}
\vskip 4pt

\textit{Third Step.}
We now address the support of the density. 
With the same notations as in the first step, 
\begin{equation*}
{\mathbb E}
\int_0^T
\int_{\mathbb S}
\vert D_{{\mathbf x}} Y_t^{\varsigma}({\mathbf x}) 
\vert^2 \ud {\mathbf x}
 \ud t 
\leq 
{\mathbb E}
\int_0^T
\int_{\mathbb S}
\vert D_{{\mathbf x}} X_t({\mathbf x}) 
\vert^2 \ud {\mathbf x}
 \ud t 
< + \infty, 
\end{equation*}
and then, there exists a constant $C$, independent of $\varsigma$, such that 
\begin{equation*}
{\mathbb E}
\int_0^T
\int_{2 \varsigma}^{1-2\varsigma}
\vert D
\tilde{F}_{\mu_t^\varsigma}^{-1}(y) \vert^2 \ud y
 \ud t 
\leq 
C.
\end{equation*}
By the same change of variable as in the first step, 
\begin{equation}
\label{eq:08-07:1}
{\mathbb E}
\int_0^T \int_{\mathbb R}
{\mathbbm 1}_{\{2 \varsigma \leq  \tilde F_{\mu_t^\varsigma}(z) \leq 1 - 2 \varsigma\}}
\frac1{p_t^\varsigma(z)}\ud z  \ud t
=
{\mathbb E}
\int_0^T \int_{\tilde{F}_{\mu_t^\varsigma}^{-1}( 2\varsigma)}^{{\tilde{F}_{\mu_t^\varsigma}^{-1}(1-2 \varsigma)}}
\frac1{p_t^\varsigma(z)}\ud z  \ud t \leq C.
\end{equation}
Since, almost surely and for almost every $t \in [0,T]$, $\mu_t$
has $p_t$ as density, the cumulative distribution function $\tilde F_{\mu_t}$ of $\mu_t$ is continuous, and by Dini's theorem, the convergence of $\tilde F_{\mu_t^\varsigma}$ to 
$\tilde F_{\mu_t}$ is uniform, in the sense that, 
${\mathbb P}$-almost surely and for almost every 
$t \in [0,T]$, 
\begin{equation*}
\lim_{\varsigma \searrow 0}
\sup_{z \in {\mathbb R}}
\vert \tilde F_{\mu_t^\varsigma}
(z) - 
\tilde F_{\mu_t}(z) \vert =0.
\end{equation*}
We now fix $\varepsilon >0$
and $\varsigma_0 \in (0,1/4)$. By Egorov's theorem, we can find a real 
$\varsigma_\varepsilon \in (0,\varsigma_0/2)$ and 
a subset $E$ in the product $\sigma$-field
$\mathfrak{F} \otimes 
{\mathcal B}([0,T])$
(with ${\mathcal B}([0,T])$ denoting the Borel subsets of $[0,T]$) such that 
${\mathbb P} \otimes \textrm{\rm Leb}_{[0,T]}(E) \geq T - \varepsilon$ and, for $(\omega,t) \in E$, 
\begin{equation*}
\sup_{z \in {\mathbb R}}
\vert \tilde F_{\mu_t^\varsigma}
(z) - 
\tilde F_{\mu_t}(z) \vert
\leq  \varsigma_0.
\end{equation*}
For $\varsigma \leq \varsigma_{\varepsilon}$, we have the following inclusion: 
\begin{equation*}
\bigl\{ (\omega,t,z) \in E \times {\mathbb R}; 
\quad 2 \varsigma_0 \leq \tilde F_{\mu_t}(z) \leq 1 - 2 \varsigma_0
\bigr\} 
\subset 
\bigl\{ (\omega,t,z) \in E \times {\mathbb R}; 
\quad 2 \varsigma \leq \tilde F_{\mu_t^{\varsigma}}(z) \leq 1 - 2 \varsigma
\bigr\}.
\end{equation*}
Together with 
\eqref{eq:08-07:1}, this yields
(still for $\varsigma \leq \varsigma_\varepsilon$)
\begin{equation}
\label{eq:08-07:2}
{\mathbb E}
\int_0^T {\mathbbm 1}_E
\int_{\mathbb R}
{\mathbbm 1}_{\{ 2 \varsigma_0 \le \tilde F_{\mu_t}(z) \leq 1 - 2 \varsigma_0\}}
\frac1{q^\varsigma_t(z)} \ud z  \ud t \leq C,
\end{equation}
where we let
$q_t^{\varsigma}(z) = 
p_t^{\varsigma}(z) {\mathbbm 1}_{(2\varsigma,1-2\varsigma)}(\tilde F_{\mu_t^{\varsigma}}(z))$.

Recall now from the second step that $(\omega,t,z) \mapsto q_t^\varsigma(z)$ has 
$(\omega,t,z) \mapsto p_t(z)$ as weak limit in $L^2(\Omega \times [0,T] \times {\mathbb R})$. 
By Mazur's theorem, we can find a sequence 
$(\tilde q^n)_{n \geq 1}$ of (finite) convex combinations of
the family 
$(q^\varsigma)_{\varsigma \in (0,\varsigma_\varepsilon)}$
that strongly converges to $p$ in $L^2(\Omega \times [0,T] \times {\mathbb R})$. Up to a subsequence the convergence holds almost everywhere on $\Omega \times [0,T] \times {\mathbb R}$. By Jensen's inequality, the reciprocal of 
the convex combination 
$\tilde q^n_t$ is dominated by the convex 
combination of the reciprocals. Thus, by \eqref{eq:08-07:2}, 
we get, for any $n \geq 1$, 
\begin{equation*}
{\mathbb E}
\int_0^T {\mathbbm 1}_E
\int_{\mathbb R}
{\mathbbm 1}_{\{ 2 \varsigma_0 \le \tilde F_{\mu_t}(z) \leq 1 - 2 \varsigma_0\}}
\frac1{\tilde q^n_t(z)} \ud z  \ud t \leq C.
\end{equation*}
And then, by Fatou's lemma, we obtain 
\begin{equation*}
{\mathbb E}
\int_0^T {\mathbbm 1}_E
\int_{\mathbb R}
{\mathbbm 1}_{\{ 2 \varsigma_0 \le \tilde F_{\mu_t}(z) \leq 1 - 2 \varsigma_0\}}
\frac1{p_t(z)} \ud z  \ud t \leq C.
\end{equation*}
Therefore, 
\begin{equation*}
\begin{split}
&{\mathbb E}
\int_0^T \int_{\mathbb R}
{\mathbbm 1}_{\{ 2 \varsigma_0 \le \tilde F_{\mu_t}(z) \leq 1 - 2 \varsigma_0\}}
{\mathbbm 1}_{\{ p_t(z) \leq  \varepsilon\}}
\ud z  \ud t 
\\
&\leq 
\bigl( {\mathbb P} \otimes 
\textrm{\rm Leb}_{[0,T]}\bigr)\bigl( E^\complement\bigr)  +
\varepsilon
{\mathbb E}
\int_0^T {\mathbbm 1}_E
\int_{\mathbb R}
{\mathbbm 1}_{\{ 2 \varsigma_0 \le \tilde F_{\mu_t}(z) \leq 1 - 2 \varsigma_0\}}
\frac1{p_t(z)} \ud z  \ud t
\\
&\leq \varepsilon 
+ C \varepsilon. 
\end{split}
\end{equation*}
Letting $\varepsilon$ and $\varsigma_0$ tend to $0$, we deduce that 
\begin{equation}
\label{eq:08-07:3}
\begin{split}
&{\mathbb E}
\int_0^T \int_{\mathbb R}
{\mathbbm 1}_{\{ 0 < \tilde F_{\mu_t}(z) < 1\}}
{\mathbbm 1}_{\{ p_t(z) =0\}}
\ud z  \ud t 
=0. 
\end{split}
\end{equation}
Since $DX_t>0$ almost everywhere on $(0,1/2)$ (with probability 1, for almost every $t$), 
the function 
$\tilde F_{\mu_t}^{-1}(\cdot) = X_t( \cdot \, /2) $ is strictly increasing on $(0,1)$. Since $\tilde F_{\mu_t}^{-1}$ is also continuous, it is a one-to-one mapping from $(0,1)$ onto $(\tilde F_{\mu_t}^{-1}(0),\tilde F_{\mu_t}^{-1}(1))=(X_t(0),X_t(1/2))$, with $\tilde F_{\mu_t}$ as inverse. Therefore, 
$\tilde F_{\mu_t}(z) \in (0,1) \Leftrightarrow z \in (X_t(0),X_t(1))$. Therefore, 
$p_{\mu_t}$ is supported by $[X_t(0),X_t(1/2)]$ and \eqref{eq:08-07:3} says that 
$p_{\mu_t}$ is almost everywhere (strictly) positive 
on $(X_t(0),X_t(1/2))$.
\vskip 4pt

\textit{Fourth Step.}
 Since $\tilde F_{\mu_t}$ is continuous, it is a standard fact that the law of $\tilde F_{\mu_t}$ under 
$\mu_t$ is the Lebesgue measure on $[0,1]$. 
Moreover, because there exists a density,  Lebesgue's theorem implies that, with probability 1, for almost every $t \geq 0$, $ \tilde F_{\mu_t}$ is differentiable almost everywhere on ${\mathbb R}$, i.e., 
\begin{equation*}
\textrm{\rm Leb}_{\mathbb R}
\bigl({\mathcal D}_{\tilde{F}_{\mu_t}}^{\complement}\bigr)
=0, 
\quad \text{with}\quad \mathcal{D}_{\tilde{F}_{\mu_t}}:= \Big\{ z\in \mathbb{R}; \ \tilde{F}_{\mu_t} \text{ is differentiable at } z \Big\}.
\end{equation*}
Using once again the fact that $\mu_t$ has a density, we deduce that 
\begin{equation*}
\mu_t\bigl(  {\mathcal D}_{\tilde{F}_{\mu_t}} \bigr) =1.
\end{equation*}

Now, since it belongs to ${\mathcal H}^1({\mathbb S})$, 
$X_t$ is also (with probability 1, for almost every $t \in [0,T]$) almost everywhere differentiable on ${\mathbb S}$ (or equivalently on $[0,1]$). In particular, letting (for fixed $\omega$ and $t$)
        $$\mathcal{D}_{X_t}:= \Big\{ y\in [0,\tfrac12]; \  X_t \text{ is differentiable at } y \Big\},$$
we have
            \begin{equation*}
             \begin{split}
            \mu_t\left(\bigl\{z \in \mathbb{R};
            \ \tfrac12  \tilde{F}_{\mu_t}(z) \in \mathcal{D}_{X_t}  \bigr\}\right) &= \int_{\mathbb{R}}{\mathbbm  1}_{\mathcal{D}_{X_t}} \bigl(\tfrac12 \tilde{F}_{\mu_t}(z)\bigr)\mathrm{d}\mu_{t}(z) 
            \\
            &= \int_0^1 {\mathbbm 1}_{\mathcal{D}_{X_t}}\bigl(\tfrac12 y\bigr)\mathrm{d}y= 2\mathrm{Leb}({\mathcal{D}_{X_t}})=1.
            \end{split}
            \end{equation*}
   In the end, we have   $\mu_t( {\mathcal D}_{\tilde{F}_{\mu_t}}  \cap \{z \in \mathbb{R}; \ \tilde{F}_{\mu_t}(z)/2 \in \mathcal{D}_{X_t}  \})=1$. Since $\mu_t$ is supported by $[X_t(0),X_t(1/2)]$ (by the third step, $p_t$ is almost everywhere positive on the interval), we deduce that, for almost every $z \in [X_t(0),X_t(1/2)]$, $\tilde F_{\mu_t}$ is differentiable at $z$ and $X_t$ is differentiable at $\tilde{F}_{\mu_t}(z)/2$. By composition, $X_t \circ (\tilde F_{\mu_t}/2)$ is differentiable at $z$, and the derivative is equal to $(1/2) D \tilde{F}_{\mu_t}(z) DX_t(\tilde{F}_{\mu_t}(z)/2)$.
   Since the composition is equal to the identity, we get, for almost every $z \in [X_t(0),X_t(1/2)]$, 
    $D\tilde{F}_{\mu_t}(z) DX_t (\tilde{F}_{\mu_t}(z)/2)=2$.
This proves \eqref{eq:08-07:11}. 

We now prove \eqref{eq:08-07:12}. We get 
\begin{equation*}
\begin{split}
{\mathbb E}
\int_0^T \int_{\mathbb R}
{\mathbbm 1}_{(X_t(0),X_t(1/2))}(z)
\frac1{p_t(z)}
\ud z  \ud t &= 
{\mathbb E}
\int_0^T \int_{\mathbb R}
{\mathbbm 1}_{(X_t(0),X_t(1/2))}(z)
\frac1{p_t^2(z)} p_t(z)
\ud z  \ud t
\\
&=\frac14 {\mathbb E}
\int_0^T \int_{\mathbb R}
{\mathbbm 1}_{(X_t(0),X_t(1/2))}(z)
\vert D_{\mathbf x} X_t \bigl(\tfrac12 \tilde F_{\mu_t}(z) \bigr) 
\vert^2 
p_t(z) 
\ud z  \ud t
\\
&=\frac14 {\mathbb E}
\int_0^T \int_{\mathbb R}
\bigl\vert D_{\mathbf x} X_t \bigl(\tfrac12 \tilde F_{\mu_t}(z) \bigr) 
\bigr\vert^2 
p_t(z) 
\ud z  \ud t
\\
&= \frac14 {\mathbb E}
\int_0^T \int_0^1
\bigl\vert D_{\mathbf x} X_t \bigl(\tfrac12 x\bigr)
\bigr\vert^2
\ud x  \ud t.
\end{split}
\end{equation*}
Letting $y=x/2$, we deduce that 
\begin{equation*}
{\mathbb E}
\int_0^T \int_{\mathbb R}
{\mathbbm 1}_{(X_t(0),X_t(1/2))}(z)
\frac1{p_t(z)}
\ud z  \ud t = \frac12  
{\mathbb E}
\int_0^T \int_0^{1/2}
\bigl\vert D_{\mathbf x} X_t (y)
\bigr\vert^2
\ud y  \ud t
= \frac14 {\mathbb E}
\int_0^T \int_{\mathbb S}
\bigl\vert D_{\mathbf x} X_t ({\mathbf x})
\bigr\vert^2
\ud {\mathbf x}  \ud t,
\end{equation*}
which completes the proof of \eqref{eq:08-07:12}.
\end{proof}
\subsection{Fokker-Planck equation}

The next step consists in deriving the Fokker-Planck equation associated with the density obtained above. To this end, we should apply an extended version of  It\^o's formula for the {\bf rearranged stochastic heat equation}, as in developed in \cite{DelHam25_2}, together with a suitable choice of the  function
with respect to which 
$(\mu_t)_{t \geq 0}$ is tested.

The key advantage of this approach is that the derivative -understood in the sense made precise below- of any sufficiently regular {\it mean-field} functional, evaluated along the law of the solution $(X_t)_{t\geq 0}$ to the {\bf rearranged stochastic heat equation with distributional drift} \eqref{eq:4.28} is expected to remain orthogonal to the reflection term $\mathrm{d}{\eta}_t$ appearing in \eqref{eq:4.28}. As a consequence, the latter does not contribute to the dynamics satisfied by the law of the process.

It is worth noting that an extension of  It\^o's formula to a drifted version of the {\bf rearranged stochastic heat equation} was already addressed in \cite{DelHam25_2}. However, the regularity assumptions imposed there on the drift term- even in their most general form- do not cover the situation considered in the present settings. Our objective is therefore to further extend this formula to the case where the drift is no longer given by a classical function, but only by a distribution.

In this context, differentiability is understood in the sense of Lions. For clarity, we recall that a functional $\psi:\mathcal{P}_2(\mathbb{R})\rightarrow \mathbb{R}$ is said to be Lions continuously differentiable if the map $v\in L^2( \mathbb{S})\mapsto \psi(\mathrm{Leb}_{\mathbb{S}}\circ v^{-1}) $ is Fr\'echet continuously differentiable. In that case, the Fr\'echet derivative at $v \in L^2(\mathbb{S})$ can be written in the form ${\bf x} \in \mathbb{S} \mapsto \partial_{\mu}\psi(\mathrm{Leb}_{\mathbb{S}}\circ v^{-1},v({\bf x}))$, where $\partial_{\mu}\psi(\mu,y)$ is a function of the variable $y\in \mathbb{R}$ belonging to $L^2(\mathbb{R},\mu)$ and depending only on the law $\mu=\mathrm{Leb}_{\mathbb{S}}\circ v^{-1}$ of $v$. For a complete reference on this topic, we refer the reader to \cite{CarDel1}.

Here is the version of It\^o's formula tailor-made for our needs:
\begin{prop}\label{prop5.2} For any smooth mean-field function $\psi: \mathcal{P}_2(\mathbb{R})\rightarrow \mathbb{R}$ with bounded, and Lipschitz continuous derivatives 
$\partial_{\mu}\psi : \mathcal{P}(\mathbb{R})\times \mathbb{R} {\rightarrow} \mathbb{R}$, $\nabla_y\partial_{\mu}\psi : \mathcal{P}(\mathbb{R})\times \mathbb{R} {\rightarrow}  \mathbb{R}$, $\partial_{\mu}^2\psi : \mathcal{P}(\mathbb{R})\times \mathbb{R} \times \mathbb{R} {\rightarrow}  \mathbb{R}$ with $\mathcal{P}(\mathbb{R})$ equipped with the $2\text{-Wasserstein}$ distance, the following holds with probability one
    \begin{equation}
    \begin{split}\label{eq:5.2}
        \psi(\mu_t) &= \psi(\mu_0) +\frac{1}{2}\int_0^t\int_{\mathbb{S}}\nabla_y\partial_{\mu}\psi(\mu_s,X_s(\mathbf{x}))\mathrm{d}\mathbf{x}\mathrm{d}s  \\
        &\quad -\int_0^t\int_{\mathbb{S}}\nabla_y\partial_{\mu}\psi(\mu_s,X_s(\mathbf{x}))[DX_s(\mathbf{x})]^2\mathrm{d}\mathbf{x}\mathrm{d}s
        +\frac{1}{2}\int_0^t\int_{\mathbb{S}}\nabla_y\partial_{\mu}\psi(\mu_s,X_s(\mathbf{x}))G_1(\mathbf{x})\mathrm{d}\mathbf{x}\mathrm{d}s \\
        &\quad  + \frac{1}{2}\int_0^t\int_{\mathbb{S}}\int_{\mathbb{S}}\partial_{\mu}^2\psi(\mu_s,X_s(\mathbf{x}),X_s(\mathbf{x}'))G_2(\mathbf{x},\mathbf{x}')\mathrm{d}\mathbf{x}\mathrm{d}\mathbf{x}'\mathrm{d}s +\int_0^t\int_{\mathbb{S}}\partial_{\mu}\psi(\mu_s,X_s(\mathbf{x}))\mathrm{d}W_s(\mathbf{x}),
        \end{split}
    \end{equation}
where $\mu_t:= \mathrm{Leb}_{\mathbb{S}}\circ X^{-1}_t$ and $X_t$ is the solution to \eqref{eq:4.28} and the functions $G_1$ and $G_2$ are given below:
    \begin{align}\label{eq:5.3}
        G_1(\mathbf{x}):= \sum_{k\in \mathbb{N}}\lambda_k^2e_k^2(\mathbf{x}), \quad G_2(\mathbf{x},\mathbf{x}'):= \sum_{k\in \mathbb{N}}\lambda_k^2e_k(\mathbf{x})e_k(\mathbf{x}'), \quad \forall \mathbf{x},\mathbf{x}' \in \mathbb{S}.
    \end{align}
and the noise $(W_t)_{t\geq 0}$ is given by \eqref{eq:noise:1.8}.
\end{prop}
{Compared to the version established in \cite{DelHam25_2}, the novelty here comes from the last term on the first line: as one may expect, it coincides exactly with the generator induced in the space of measures by the heat flow; see, for instance, \cite[Chapter 5]{CarDel1}.

The proof of Proposition \ref{prop5.2} consists in expanding the function $\psi$ along the scheme
$\{(X^{(h)}_t)_{0 \le t \le T}\}_{{h \in T/{\mathbb N}^*}}$, {for a fixed $T>0$}, and then passing to the limit as the parameter $h$ tends to $0$. In this respect, the proof bears some similarities with that of \cite[Theorem 1.1]{DelHam25_2}. However, it is not a mere copy-and-paste argument: first, because, in order to expand the scheme, we explicitly use \cite[Theorem 1.1]{DelHam25_2}; and second, because the main difficulty comes precisely from the jumps induced by the addition of the convolution in the scheme. In particular, we are not able here to prove an analog of \cite[Lemma 2.3]{DelHam25_2}, which is used throughout the proof of \cite[Theorem 1.1]{DelHam25_2} to ensure several properties of strong compactness in $L^2(\mathbb S)$. Nevertheless, in the following result, we can establish a version of \cite[Lemma 3.2]{DelHam25_2} which is sufficient for our purposes:
\begin{lemm}\label{lemm.7.4}
Consider a smooth mean-field function $\psi: \mathcal{P}_2(\mathbb{R})\rightarrow \mathbb{R}$ with bounded and jointly continuous derivative $\partial_{\mu}\psi: \mathcal{P}_2(\mathbb{R}) \times \mathbb{R}\rightarrow \mathbb{R}$. 
Then, for $T>0$ as above and for almost every $S \in [0,T]$, the followings hold:
\begin{enumerate}
    \item  the family $\left(\psi(\mu_{S}^{(h)})\right)_{{h \in T/{\mathbb N}^*}}$ converges to $\psi(\mu_{S})$, as $h\searrow 0$;
 \item the random variables
  $$ \left(\int_{0}^{S}\int_{\mathbb{S}}\partial_{\mu}\psi(\mu_s^{(h)},X_s^{(h)}(\mathbf{x}))[{DX_s^{(h)}(\mathbf{x})}]^2\mathrm{d}\mathbf{x} \mathrm{d}s\right)_{{h \in T/{\mathbb N}^*}}$$
converge in law as $h\searrow 0$ to the random variable
$\int_{0}^{S}\int_{\mathbb{S}}\partial_{\mu}\psi(\mu_s,X_s(\mathbf{x}))[{DX_s(\mathbf{x})}]^2\mathrm{d}\mathbf{x} \mathrm{d}s,$
\end{enumerate}
where $\{(X_t^{(h)})_{0 \le t \le T}\}_{h \in T/{\mathbb N}^*}$ is given by \eqref{scheme1}, and $(X_t)_{t\in[0,T]}$ is its weak limit in $\mathcal{C}([0,T],{
\mathcal{H}^{-1}_{{\rm Sym}}(\mathbb{S})})$, and where, for all $t \in [0,T]$, 
$\mu_t^{(h)}=\textrm{\rm Leb}_{\mathbb S} \circ (X_t^{(h)})^{-1}$
and
$\mu_t=\textrm{\rm Leb}_{\mathbb S} \circ (X_t)^{-1}$.
\end{lemm}
\begin{proof} The proof is divided in four steps.
\vskip 4pt
{\it First Step.} We first establish the convergence
of the $L^2$ norms (in $\omega$, $t$ and $\mathbf{x}$) 
of $\{(D X_t^{(h)})_{0 \le t \le T}\}_{h \in T/{\mathbb N}^*}$. 
Letting $\epsilon >0$ and using the triangle inequality, we have
\begin{align*}
& \int_0^T \left\vert \mathbb{E}   \|X_s^{(h)}\|_2^2 -   \mathbb{E}    \|X_s \|_2^2
\right\vert \mathrm{d}s
\\
&\leq \mathbb{E} \int_0^T \left|\|X_s^{(h)}\|_2^2- \|{\rm{e}}^{\epsilon \Delta}X_s^{(h)}\|_2^2\right|\mathrm{d}s + \mathbb{E} \int_0^T \left|\|X_s\|_2^2- \|{\rm{e}}^{\epsilon \Delta}X_s\|_2^2\right|\mathrm{d}s +  \int_0^T \left|
\mathbb{E}
\|{\rm{e}}^{\epsilon \Delta}X_s^{(h)}\|_2^2- \mathbb{E}\|{\rm{e}}^{\epsilon \Delta}X_s\|_2^2 \right|\mathrm{d}s.
\end{align*}  
Thanks to \eqref{bound 1.26} and to \eqref{260},
the sum of the first two terms on the last line tends to $0$ with $\epsilon$, uniformly in 
${h \in T/{\mathbb N}^*}$. Moreover, using the weak convergence of 
$\{(X_t^{(h)})_{0 \le t \le T}\}_{h \in T/{\mathbb N}^*}$
to 
$\{(X_t)_{0 \le t \le T}\}_{h \in T/{\mathbb N}^*}$, as stated in Proposition \ref{remark 4.1}, we deduce that the third term goes to $0$ with $h$ when $\epsilon$ is fixed. 
Hence,
the term on the left-hand side goes to $0$ with $h$, from which we deduce that, for almost every $S \in [0,T]$, 
\begin{equation}\label{eq:7:5}
    \lim_{h \searrow 0}\mathbb{E}\|X_S^{(h)}\|_2^2 = \mathbb{E}\|X_S\|_2^2.
\end{equation}
Now returning to \eqref{eq:ito:square:Xth:45}, we claim that 
\begin{equation}\label{eq:7:6}
	 2\mathbb{E}\int_{0}^{S} \|DX_{s}^{(h)}\|_2^2 \mathrm{d}s = \mathbb{E}\left(\|X_0\|_2^2 -\|X_{S}^{(h)}\|_2^2 \right)  +
    \Bigg( 1 +\sum_{k \in \mathbb{N}} \lambda_k^2\Bigg) S+ \mathcal{O}_h(1).
\end{equation}
Hence, combining \eqref{eq:7:5} and \eqref{eq:7:6} we obtain, for almost every $S \in [0,T],$
\begin{equation}
	 2\lim_{h\searrow 0}\mathbb{E}\int_{0}^{S} \|DX_{s}^{(h)}\|_2^2 \mathrm{d}s = \mathbb{E}\left(\|X_0\|_2^2 -\|X_S\|_2^2 \right)  +
    \Bigg( 1 +\sum_{k \in \mathbb{N}} \lambda_k^2\Bigg) S.
\end{equation}
Observe that the term on the right-hand side above corresponds to $2\mathbb{E}\int_{0}^{S} \|DX_{s}\|_2^2 \mathrm{d}s$ (see \eqref{261}). Hence, for almost every $S \in [0,T]$,
\begin{equation}\label{eq:7.8}
\mathbb{E}\int_{0}^{S} \|DX_{s}^{(h)}\|_2^2 \mathrm{d}s \rightarrow \mathbb{E}\int_{0}^{S} \|DX_{s}\|_2^2 \mathrm{d}s, \quad \text{ as } {h\searrow 0}.
\end{equation}
\vskip 4pt

{\it Second Step.} Fix $T>0$ as in the first step. By invoking Skorokhod's representation theorem, we can find a family of processes $\{\{(Y_s^{(h)})_{0 \leq s\leq {T}}\}_{{h \in T/{\mathbb N}^*}},(Y_s)_{0 \le s \le T}\}$,  constructed on a probability space $(\P,\mathfrak{G},\nu)$, with trajectories in $\mathbb{D}([0,{T}],\mathcal{H}^{-1}_{{\rm Sym}}(\mathbb{S}))$, such that
\begin{itemize}
\item for any {$h\in T/{\mathbb N}^*$}, $\mathscr{L}((Y_s^{(h)})_{0 \le s \le T})= \mathscr{L}((X_s^{(h)})_{0 \le s \le T})$ in $\mathbb{D}([0,T],\mathcal{H}^{-1}_{{\rm Sym}}(\mathbb{S}))$;
\item $\mathscr{L}\left((Y_s)_{0 \le s \le T}\right)= \mathscr{L}\left((X_s)_{0 \le s \le T}\right)$ in $\mathbb{D}([0,{T}],\mathcal{H}^{-1}_{{\rm Sym}}(\mathbb{S}))$, where $(X_s)_{0 \le s \le T}$ stands for the solution to \eqref{RHSE2} as given by Theorem \ref{Main 1} (it also corresponds to the limit in law of the processes $(X^{(h)}_s)_{0 \le s \le T}$ when $h \searrow 0$);
\item as $h \searrow 0$, $(Y^{(h)}_s)_{0 \le s \le T} \rightarrow (Y_s)_{0 \le s \le T} \quad \nu\text {-a.s.}$ in $\mathbb{D}([0,{T}],\mathcal{H}^{-1}_{{\rm Sym}}(\mathbb{S}))$. 
\end{itemize}
In addition, from \eqref{bound 2.8} and \eqref{bound 1.26}, 
we deduce that (we use the symbol $\varsigma$ to denote a generic element of $\Pi$): 
\begin{equation}\label{eq:7:9}
 \sup_{{h \in T/{\mathbb N}^*}} \Biggl( \int_{\P} \sup_{s\in [0,{T}]}\|Y_s^{(h)}\|_2^2(\varsigma)\mathrm{d}\nu(\varsigma)  +\int_{\P} \int_0^{{T}} \|DY_s^{(h)}\|_2^2(\varsigma){\rm d}s\mathrm{d}\nu(\varsigma)\Biggr)<\infty. 
\end{equation}
Moreover, from \eqref{eq:7:5}, we have, for almost every {$S \in [0,T]$},
\begin{equation}
    \int_{\P} \|Y_S^{(h)}\|_2^2(\varsigma)\mathrm{d}\nu(\varsigma) \rightarrow  \int_{\P}  \|Y_S\|_2^2(\varsigma)\mathrm{d}\nu(\varsigma)\,\, \text{ as } h\searrow 0, 
\end{equation}
and, from \eqref{eq:7.8}, we also have, for almost every {$S \in [0,T]$},
\begin{equation}
\label{eq:7.11}
    \int_{\P} \int_0^S \|DY_s^{(h)}\|_2^2(\varsigma){\rm d}s\mathrm{d}\nu(\varsigma) \rightarrow  \int_{\P} \int_0^S \|DY_s\|_2^2(\varsigma){\rm d}s\mathrm{d}\nu(\varsigma)\,\, \text{ as } h\searrow 0.
\end{equation}
Next, we want to prove that, for almost every {$S \in [0,T]$},
\begin{equation}\label{eq:7:12}
   \int_{\P}{\rm d}\nu(\varsigma) \|Y_S^{(h)}- Y_S \|_2^2(\varsigma) \rightarrow 0  \,\, \text{ as } {h\searrow 0},
\end{equation}
as well as 
\begin{equation}\label{eq:7.9}
    \int_{\P}{\rm d}\nu(\varsigma)\int_0^S \|DY_s^{(h)}- DY_s \|_2^2(\varsigma) \mathrm{d}s\rightarrow 0  \,\, \text{ as } {h\searrow 0}.
\end{equation}
We will only prove the second statement, since the first one can be established using the same arguments.

From the identity $\|DY_s^{(h)}- DY_s \|_2^2 = \|DY_s^{(h)}\|_2^2 -2 DY_s^{(h)}\cdot DY_s + \| DY_s \|_2^2$ together with \eqref{eq:7.11}, it is equivalent to prove that
\[ \int_{\P}{\rm d}\nu(\varsigma)\int_0^S (DY_s^{(h)}\cdot DY_s)(\varsigma)  \mathrm{d}s \rightarrow   \int_{\P}{\rm d}\nu(\varsigma)\int_0^S \| DY_s \|_2^2(\varsigma)  \mathrm{d}s \,\, \text{ as } {h\searrow 0}.\]
Let $\epsilon >0$. Notice that, for almost all $(\varsigma,s)\in \P\times [0,{T}]$,
$$ \langle DY_s^{(h)}, {\rm e}^{\epsilon\Delta}DY_s)\rangle(\varsigma) = -\langle Y_s^{(h)}, D{\rm e}^{\epsilon\Delta}DY_s)\rangle(\varsigma) \rightarrow -\langle Y_s,D {\rm e}^{\epsilon\Delta}DY_s\rangle(\varsigma) \,\, \text{ as } {h\searrow 0},$$ 
{which} follows from the fact that $(Y^{(h)}_s)_{0 \le s \leq T} \xrightarrow{h\searrow 0} (Y_s)_{0 \le s \le T} \quad \nu\text {-a.s.}$ in $\mathbb{D}([0,{T}],\mathcal{H}^{-1}_{{\rm Sym}}(\mathbb{S}))$
{(see (12.14) in \cite{Billingsley})}.
Moreover, 
by 
combining Cauchy-Schwartz's inequality
and 
\eqref{bound 1.26}, we get, 
for any $A \in \mathfrak{G}\otimes \mathcal{B}([0,S])$,  
\begin{align*}
    &\int_{\P}{\rm d}\nu(\varsigma)\int_0^S \mathbbm{1}_{A} \vert DY_s^{(h)}\cdot {\rm e}^{\epsilon\Delta}DY_s \vert(\varsigma) \mathrm{d}s\\
    &\leq \left( \int_{\P}{\rm d}\nu(\varsigma)\int_0^S \|DY_s^{(h)}\|_2^2(\varsigma) \mathrm{d}s \right)^{1/2} \left(\int_{\P}{\rm d}\nu(\varsigma)\int_0^S \mathbbm{1}_{A}\|{\rm e}^{\epsilon\Delta}DY_s\|_2^2(\varsigma) \mathrm{d}s\right)^{1/2}\\
    &\leq C \left(\int_{\P}{\rm d}\nu(\varsigma)\int_0^S \mathbbm{1}_{A}\|DY_s\|_2^2(\varsigma) \mathrm{d}s\right)^{1/2},
\end{align*}
where the constant $C$ above does not depend on $h$.
Obviously, due to the integrability of $\P\times [0,S]\ni (\varsigma,s)\mapsto \|DY_s\|_2^2(\varsigma)$, the last term 
tends to $0$ with $\nu \otimes \textrm{\rm Leb}_{\mathbb R}(A)$. This in turn implies the uniform integrability of $ \{(\varsigma,s) \mapsto (DY_s^{(h)}\cdot {\rm e}^{\epsilon\Delta}DY_s)(\varsigma)\}_{{h \in T/{\mathbb N}^*}}$ on $\P\times [0,S]$. Consequently, by the last two displays, we obtain
\begin{equation}
\label{eq:7.14}
    \lim_{h\searrow 0} \int_{\P}{\rm d}\nu(\varsigma)\int_0^S (DY_s^{(h)}\cdot {\rm e}^{\epsilon\Delta}DY_s)(\varsigma) \mathrm{d}s = \int_{\P}{\rm d}\nu(\varsigma)\int_0^S (DY_s\cdot {\rm e}^{\epsilon\Delta}DY_s)(\varsigma) \mathrm{d}s. 
\end{equation} 
On the other hand, from {Cauchy-Schwarz}' inequality and \eqref{bound 1.26}, we derive
\begin{equation}
\label{eq:7.15}
 \int_{\P}{\rm d}\nu(\varsigma)\int_0^S |DY_s^{(h)}\cdot (1-{\rm e}^{\epsilon \Delta}) DY_s|(\varsigma)  \mathrm{d}s \leq C \left(\int_{\P}{\rm d}\nu(\varsigma)\int_0^S\|(1-{\rm e}^{\epsilon \Delta})DY_s\|_{2}^2(\varsigma) {\rm d}s \right)^{1/2}. 
\end{equation}
By invoking the dominated convergence theorem and \eqref{259}, the term on the right hand side goes to 0 as $\epsilon \searrow 0$. Hence, combining 
\eqref{eq:7.14} and
\eqref{eq:7.15}, we  get
 \eqref{eq:7.9}. As noted earlier, \eqref{eq:7:12} follows from the same arguments.
\vskip 4pt
{\it Third Step.} {For $S \in [0,T]$}, write
\begin{equation}
    \psi(\mu_S^{(h)})= \psi({\rm Leb}_{\mathbb{S}}\circ(X_S^{(h)})^{-1}).
\end{equation}
We claim that the latter has the same law as the random variable $\psi({\rm Leb}_{\mathbb{S}}\circ(Y_S^{(h)})^{-1})$, where the process $\{(Y_s^{(h)})_{s\in [0,T]}\}_{{h \in T/{\mathbb N}^*}}$ is as in the second step of the proof; we recall that $(Y_s)_{s\in [0,{T}]}$ denotes {the almost sure limit of the latter in $\mathbb{D}([0,T],\mathcal H^{-1}_{{\rm Sym}}(\mathbb{S}))$}. Therefore,
{in order to establish the first item in the statement of the lemma}, 
it is enough to prove that 
$\{\psi({\rm Leb}_{\mathbb{S}}\circ(Y_S^{(h)})^{-1})\}_{{h \in T/{\mathbb N}^*}}$
converges to $\psi({\rm Leb}_{\mathbb{S}}\circ(Y_S)^{-1})$ in $\Pi$ probability as $h\searrow 0$. 
{Combining the fact that $\Psi$ is continuous with respect to the Wasserstein distance with the $L^2$ convergence established in \eqref{eq:7:12}, we deduce that, for a.e. {$S \in [0,T]$}, in $\Pi$-probability, }
\begin{align}
    \psi({\rm Leb}_{\mathbb{S}}\circ(Y_S^{(h)})^{-1}) = \tilde{\psi}(Y_S^{(h)}) \xrightarrow{h\searrow 0} \tilde{\psi}(Y_S) = \psi({\rm Leb}_{\mathbb{S}}\circ(Y_S)^{-1})
\end{align}
where $\tilde{\psi}$ stands for the Lions' lift of the functional $\psi: \mathcal{P}_2(\mathbb{R})\rightarrow \mathbb{R}$.
\vskip 4pt

{\it Fourth Step.} {Still for $S \in [0,T]$}, write
\begin{align*}
\int_{0}^{S}\int_{\mathbb{S}}\partial_{\mu}\psi(\mu_s^{(h)},X_s^{(h)}(\mathbf{x}))[{DX_s^{(h)}}({\bf x})]^2\mathrm{d}\mathbf{x}\mathrm{d}s= \int_{0}^{S}\int_{\mathbb{S}}\partial_{\mu}\psi({\rm Leb}_{{\bf S}}\circ(X_s^{(h)})^{-1},X_s^{(h)}(\mathbf{x}))[{DX_s^{(h)}}({\bf x})]^2\mathrm{d}\mathbf{x }\mathrm{d}s.
\end{align*}
Arguing as in the previous step, it is sufficient to prove that,
{for almost every $S \in [0,T]$}, the random variable 
$$\int_{0}^{S}\int_{\mathbb{S}}\partial_{\mu}\psi({\rm Leb}_{{\bf S}}\circ(Y_s^{(h)})^{-1},Y_s^{(h)}(\mathbf{x}))[{DY_s^{(h)}}({\bf x})]^2\mathrm{d}\mathbf{x}  \mathrm{d}s,$$ 
converges almost surely to $\int_{0}^{S}\int_{\mathbb{S}}\partial_{\mu}\psi({\rm Leb}_{{\bf S}}\circ(Y_s)^{-1},Y_s(\mathbf{x}))[{DY_s}({\bf x})]^2\mathrm{d}\mathbf{x}  \mathrm{d}s$, as $h\searrow 0$, in the probability space $(\P,\mathfrak{G},\nu)$. Since $\psi$ has bounded and Lipschitz-continuous derivatives, we obtain, for any $\epsilon >0$,
\begin{align}
   &\int_{\P}{\rm d}\nu\left| \int_{0}^{S}\!\!\!\int_{\mathbb{S}}\big[\partial_{\mu}\psi({\rm Leb}_{{\bf S}}\circ(Y_s^{(h)})^{-1},Y_s^{(h)}(\mathbf{x}))[{DY_s^{(h)}}({\bf x})]^2-\partial_{\mu}\psi({\rm Leb}_{{\bf S}}\circ(Y_s)^{-1},Y_s(\mathbf{x}))[{DY_s}({\bf x})]^2\big]\mathrm{d}\mathbf{x}  \mathrm{d}s \right|
   \nonumber
   \\
   &\leq C \int_{\P}{\rm d}\nu(\varsigma)\int_{0}^{S}\int_{\mathbb{S}}\big|\big([{DY_s^{(h)}}]^2 - [{DY_s}]^2\big)({\bf x})\big|(\varsigma)\mathrm{d}\mathbf{x} \mathrm{d}s
   \nonumber
   \\
   &\quad + C\int_{\P}{\rm d}\nu(\varsigma)\int_0^S \int_{\mathbb{S}} \left(1\wedge\big[\mathscr{W}_2({\rm Leb}_{\mathbb{S}}\circ(Y_s^{(h)})^{-1},{\rm Leb}_{{\bf S}}\circ(Y_s)^{-1}) + | Y_s^{(h)}-Y_s|(\mathbf{x})\big]\right) {[DY_s({\bf x})]^2}\mathrm{d}{\bf x} \mathrm{d}s
   \nonumber
   \\
   &\leq C \biggl(\int_{\P}{\rm d}\nu(\varsigma)\int_{0}^{S}\|{DY_s^{(h)}} - {DY_s}\|_2^2(\varsigma) \mathrm{d}s\biggr)^{1/2} + 
   {C} \epsilon \int_{\P}{\rm d}\nu(\varsigma)\int_0^S \|{DY_s}\|_2^2(\varsigma) \mathrm{d}s
   \nonumber
   \\ 
  &\quad +{C} \int_{\P}{\rm d}\nu(\varsigma)\int_0^S \int_{{\bf S}} \mathbbm{1}_{\{\mathscr{W}_2({\rm Leb}_{{\bf S}}\circ(Y_s^{(h)})^{-1},{\rm Leb}_{{\bf S}}\circ(Y_s)^{-1}) + | Y_s^{(h)}-Y_s|(\mathbf{x})\geq \epsilon\}}{[DY_s({\bf x})]^2}{\rm d}{\bf x} \mathrm{d}s,
\label{eq:09-07:1}
\end{align}
where we used Cauchy-Schwartz's inequality 
and \eqref{eq:7:9} to derive the first term on the {fourth} line, and where 
we omitted the variable $\varsigma$ 
in the last and ante-penultimate terms (to alleviate the notation). 
By \eqref{eq:7.9}, we know that, {for almost every $S \in [0,T]$}, $\int_{\P}{\rm d}\nu(\varsigma)\int_{0}^{S}\|{DY_s^{(h)}} - {DY_s}\|_2^2(\varsigma) \mathrm{d}s$ tends to $0$ with $h$. Obviously, from \eqref{eq:7:9}, the second term on the {fourth line} is bounded by
$C \epsilon$, with $C \geq 0$ independent of $\epsilon$.
As for the last term, we use a uniform integrability argument: since {$\int_{\P}{\rm d}\nu(\varsigma)\int_0^S \|DY_s\|_2^2(\varsigma)\mathrm{d}s <\infty$}, we know that, for all $\delta_0 > 0$, there exists  $\delta_1 >0$ such that, for any $A  \in \mathfrak{G}\otimes \mathcal{B}([0,S]) \otimes \mathcal{B}({\mathbb S})$, 
\begin{equation}\label{eq:7:13}
    \nu \otimes{\rm Leb}_{[0,T]}\otimes {\rm Leb}_{{\bf S}}(A) \leq \delta_1 \Rightarrow \int_{\P}{\rm d}\nu(\varsigma)\int_0^S \int_{\mathbb{S}} \mathbbm{1}_{A}[{DY_s}({\bf x})]^2(\varsigma){\rm d}{\bf x}\mathrm{d}s \leq \delta_0.
\end{equation}
In particular, for $A:= {\{(\varsigma,s,{\bf x}):\mathscr{W}_2({\rm Leb}_{\mathbb{S}}\circ(Y_s^{(h)})^{-1},{\rm Leb}_{\mathbb{S}}\circ(Y_s)^{-1}) + | Y_s^{(h)}-Y_s|({\bf x})\geq \epsilon\}}$,
Markov's inequality yields
\begin{align*}
  &\nu \otimes{\rm Leb}_{[0,T]}\otimes {\rm Leb}_{\mathbb{S}}\left({\{(\varsigma,s,{\bf x}):\mathscr{W}_2({\rm Leb}_{\mathbb{S}}\circ(Y_s^{(h)})^{-1},{\rm Leb}_{\mathbb{S}}\circ(Y_s)^{-1}) + | Y_s^{(h)}-Y_s|\geq \epsilon\}}\right)\\
  &\leq \frac{1}{\epsilon^2 S} \int_{\P}{\rm d}\nu(\varsigma)\int_0^S \int_{\mathbb{ S}}\left(\mathscr{W}_2({\rm Leb}_{\mathbb{S}}\circ(Y_s^{(h)})^{-1},{\rm Leb}_{\mathbb{ S}}\circ(Y_s)^{-1}) + | Y_s^{(h)}({\bf x})-Y_s({\bf x})|\right)^2{\rm d}{\bf x}{\rm d}s \\ 
  &\leq \frac{4}{\epsilon^2 S} \int_{\P}{\rm d}\nu(\varsigma)\int_0^S \|Y_s^{(h)}-Y_s\|_2^2(\varsigma){\rm d}s 
\end{align*} 
Hence, \eqref{eq:7:12} says that,
{for 
almost every $S \in [0,T]$
and for 
$\delta_1$ as in \eqref{eq:7:13}, 
we can choose $h$ small enough such that $A$ given above satisfies
\eqref{eq:7:13}}. This implies that the last term on
\eqref{eq:09-07:1}
can be made as small as possible as $h\searrow 0.$ This concludes the proof.
\end{proof}


The second intermediate result is given below:
\begin{lemm}\label{lemm 5.5} 
    Let $\psi$ be any bounded and Lipschitz continuous function from $\mathcal{P}_2(\mathbb{R})\times \mathbb{R} \rightarrow \mathbb{R}$. 
    For $T>0$ and $N \in {\mathbb N}^*$,  
    consider the 
scheme \eqref{scheme1} with step $h=T/N$, 
{together with some time $S \in [0,T]$.}
Then,
{letting {$N_S:=\lfloor S/h \rfloor \in \{1,\cdots,N\}$}, and}
denoting by 
$(t_l^{(h)}=l h)_{l=0,\cdots,N}$ the subdivision supporting the scheme, 
the family of random variables
 (to alleviate the notation, we omit the superscript $(h)$ in
 $t_l^{(h)}$ but it should be clear that the subdivision depends on $h$)
    \begin{align*}
        \Bigg({\sum_{l=0}^{N_S-1}} h \int_{\mathbb{S}}\psi\Bigl(\mu_{t_{l+1}}^{(h)},X_{t_{l+1}}^{(h)}(\mathbf{x})\Bigr)\mathrm{d}\mathbf{x}\Bigg)_{h \in T/{\mathbb N}^*}, \quad 
        \textrm{\rm with} \quad
        \mu_{t_{l+1}}^{(h)} = \mathscr{L}(X_{t_{l+1}}^{(h)}) = {\rm Leb}_{\mathbb{S}}\circ (X_{t_{{l+1}}^{-}}^{(h)})^{-1} \star \Gamma_h,
    \end{align*}
converges in law to $\int_0^{{S}} \int_{\mathbb{S}}\psi(\mu_{s},X_s(\mathbf{x}))\mathrm{d}\mathbf{x}\mathrm{d}s$ as $h$ tends to $0$, where
$(X_s)_{0 \leq s \leq T}$
is given by Theorem \ref{Main 1}.
  \end{lemm}
\begin{proof} 
\textit{First Step.}
In this step, the value of $h \in T/{\mathbb N}^*$
is fixed.
By \eqref{bound 1.26}, we first observe that, for almost every
 $s\in [t_l,t_{l+1})$, $X_s^{(h)}$ takes values in $\mathcal{H}^1_{\mathrm{Sym}}(\mathbb{S})$.  {In fact, from \cite[Lemma 2.3]{DelHam25_2}, this holds true
 for all $s \in (t_l,t_{l+1})$, and
  $X_{t_{l+1}-}^{(h)}$
 also takes values  
 in ${\mathcal H}^1_{\rm Sym}({\mathbb S})$. Moreover, for every $s \in [t_l, t_{l+1})$},
\begin{align*}
    \mathbb{E}\|DX_{t_{l+1}^{-}}^{(h)}\|_2^2 \leq C\Big(1+ \mathbb{E}\|DX_s^{(h)}\|_2^2\Big),
\end{align*}
for a constant $C$ independent of the discretization parameters.
Integrating on $[t_l,t_{l+1}]$ both sides of the latter inequality and then summing over {$l \in \{0,\cdots,N_S-1\}$},  we obtain
\begin{equation}\label{eq 5.4}
   \sum_{l=0}^{{N_S}-1} (t_{l+1}-t_{l})\mathbb{E}\|DX_{t_{l+1}^{-}}^{(h)}\|_2^2 \leq  CT + \int_{0}^{T}\mathbb{E}\|DX_s^{(h)}\|_2^2\mathrm{d}s \leq C(1+ \mathbb{E}\|X_0\|_2^2),
\end{equation}
where the last inequality follows from \eqref{bound 1.26}. 
\vskip 4pt

{\it Second Step.}
In order to simplify the analysis, we let, for any 
mapping $\varphi \in L^2({\mathbb S})$
\begin{equation*}
G(\varphi) = \int_{{\mathbb S}}
\psi\bigl( \textrm{\rm Leb}_{\mathbb S} \circ \varphi^{-1},\varphi({\mathbf x}) \bigr) \ud {\mathbf x}.
\end{equation*}
With this notation, 
\begin{equation*}
\sum_{l=0}^{{N_S}-1} h \int_{\mathbb{S}}\psi\Bigl(\mu_{t_{l+1}}^{(h)},X_{t_{l+1}}^{(h)}(\mathbf{x})\Bigr)\mathrm{d}\mathbf{x}
=
\sum_{l=0}^{{N_S}-1} h G(X_{t_{l+1}}^{(h)}).
\end{equation*}
We now  use the family of processes $\{(Y^{(h)})_{s\in [0,T]}\}_{h \in T/{\mathbb N}^*}$ constructed in the second step of the proof of Lemma \ref{lemm.7.4}. It converges almost surely to $(Y_s)_{s\in [0,T]}$ in $\mathbb{D}([0,T],\mathcal H^{-1}_{{\rm Sym}}(\mathbb{S}))$.
We aim to show that $\sum_{l}hG(Y_{t_{l+1}}^{(h)})$ converges to $\int_0^{{S}} G(Y_t)\mathrm{d}t$, as $h \searrow 0$, a.s.  on 
$(\P,\mathfrak{G},\nu).$ 

Before proceeding further, let us first emphasize from \eqref{eq 5.4} that 
\begin{equation}\label{eq 5.4 bis}
  \int_{\P}\mathrm{d}\nu(\varsigma) \sum_{l=0}^{{N_S}-1} (t_{l+1}-t_{l})\|DY_{t_{l+1}^{-}}^{(h)}\|_2^2(\varsigma) \leq C.
\end{equation}
From the triangle inequality, the following decomposition holds for any $\epsilon >0$ (on the right-hand side below, we omit the integration variable $\varsigma$):
\begin{equation*}
\begin{split}
    \int_{\P}\mathrm{d}\nu(\varsigma) \Big| \sum_{l=0}^{
    {N_S}-1} h G(Y_{t_{l+1}}^{(h)}) -\int_0^{{S}} G(Y_s)\mathrm{d}s \Big|(\varsigma)
    &\leq \int_{\P}\mathrm{d} \nu \Big| \sum_{l=0}^{ {N_S}-1} h G(Y_{t_{l+1}}^{(h)}) -\sum_{l=0}^{ {N_S}-1} h G(Y_{t_{l+1}^{-}}^{(h)}) \Big|
    \\
&\quad    + \int_{\P}\mathrm{d}\nu \Big| \sum_{l=0}^{ {N_S}-1} h G(Y_{t_{l+1}^{-}}^{(h)}) -
\sum_{l=0}^{ {N_S}-1} h G({\rm{e}}^{\epsilon \Delta} Y_{t_{l+1}^{-}}^{(h)})
\Big| 
\\
    &\quad + \int_{\P}\mathrm{d}\nu \Big| 
    \sum_{l=0}^{ {N_S}-1} h 
    \bigl[ G({\rm{e}}^{\epsilon \Delta} Y_{t_{l+1}^{-}}^{(h)}) - 
G({\rm{e}}^{\epsilon \Delta} Y_{t_{l+1}^{-}})
\bigr]\Big|
    \\
    &\quad + \int_{\P}\mathrm{d}\nu \bigg| 
    \sum_{l=0}^{ {N_S}-1} h 
    \bigl[ G({{\rm{e}}^{\epsilon \Delta} Y_{t_{l+1}^-}}) - 
    \int_0^{{S}} 
G({\rm{e}}^{\epsilon \Delta} Y_s) \ud s
\bigr] \biggr|
\\
    &\quad + \int_{\P}\mathrm{d}\nu(\varsigma) \bigg| \int_0^{{S}} G({\rm{e}}^{\epsilon \Delta} Y_s) \mathrm{d}s  -\int_0^{{S}} G(Y_s)\mathrm{d}s \bigg|
    \\
    &=: {R}_1^{(h)} + {R}_2^{(h,\epsilon)} + {R}_3^{(h,\epsilon)} +{R}_4^{(h,\epsilon)}+{R}_5^{(\epsilon)}.
\end{split}
\end{equation*}
From the Lipschitz continuity of $\psi$ together with \eqref{eq:1.16:bis}, it is easily seen that 
{there is a constant $C> 0$, independent of $h$ and $N$,} such that
\begin{equation}
\label{eq:26-06:1}
   {R_1^{(h)}} \leq  C\sqrt{h}.
\end{equation}
As for {$R_2^{(h,\epsilon)}$}, 
we obtain the following bound by  using 
the elementary inequality $1-\mathrm{e}^{-x} \leq x$, for any $x\geq 0$,
\begin{equation}
\label{eq:26-06:2}
\begin{split}
{R_2^{(h,\epsilon)}} &\leq C\int_{\P}\mathrm{d}\nu \sum_{l=0}^{{N_S}-1} (t_{l+1}-t_l) \Big(\sum_{k\in \mathbb{N}}(1- \mathrm{e}^{-4\pi^2k^2\epsilon})\langle Y_{t_{l+1}^{-}}^{(h)}, e_k\rangle^2\Big)^{1/2}
\\
&\leq C \int_{\P}\mathrm{d}\nu \sum_{l=0}^{{N_S}-1} (t_{l+1}-t_l) \|Y_{t_{l+1}^{-}}^{(h)}\|_{2,1} \sqrt{\epsilon},
\end{split}
\end{equation}
the constant $C> 0$ above is independent of $h$, $N$ and $\epsilon >0.$
Hence, combining \eqref{eq:7:9} and \eqref{eq 5.4 bis}, we deduce thereof: ${R}_2^{(h,\epsilon)} \leq C \sqrt{\epsilon}$.

Regarding $R_3^{(h,\epsilon)}$, we use the fact that the process $({\rm{e}}^{\epsilon \Delta} Y_s)_{0 \le s \le T}$ has continuous trajectories with values in $L^2({\mathbb S})$. Therefore, the convergence of 
$({\rm e}^{\epsilon \Delta} Y_s^{(h)})_{0 \le s \leq T}$ to 
$({\rm{e}}^{\epsilon \Delta} Y_s)_{0 \le s \le T}$ is uniform in $L^2({\mathbb S)})$ on $[0,T]$ {(see (12.14) in \cite{Billingsley})}. We deduce that, for $\epsilon$ fixed, 
$R_3^{(h,\epsilon)}$ tends to $0$ with $h$.

The term $R_4^{(h,\epsilon)}$
can be also treated by using the 
continuity of the trajectories of 
$({\rm e}^{\epsilon \Delta} Y_s)_{s \in [0,T]}$. This shows that, for $\epsilon$ fixed, 
$R_4^{(h,\epsilon)}$ tends to $0$ with $h$. 

Finally, for 
$R_5^{(\epsilon)}$, we have 
\begin{equation}
\label{eq:26-06:3}
  R_5^{(\epsilon)} \leq C 
  \int_{\P} \ud \nu(\varsigma) \int_0^T \|{\rm{e}}^{\epsilon\Delta}(Y_s-Y_s^{(h)})\|_2(\varsigma) \mathrm{d}s,    
\end{equation}
which tends to $0$ with $\epsilon$.

The proof then follows by choosing first $\epsilon$
small enough, which makes it possible to render $R_2^{(h,\epsilon)}$
and $R_5^{(\epsilon)}$
small (see \eqref{eq:26-06:1}, 
\eqref{eq:26-06:2}
and
\eqref{eq:26-06:3}). 
We then let $h$ tend to zero 
to render the remaining terms 
as small as needed.
\end{proof}
We are now in position to provide a proof of Proposition \ref{prop5.2}
\begin{proof}[Proof of Proposition \ref{prop5.2}] Fix $T>0$ and choose $h=T/N$ for an integer $N \in{\mathbb N}^*$. We denote by $(\mu_t^{(h)}:= \mathrm{Leb}_{\mathbb{S}}\circ (X_t^{(h)})^{-1})_{0 \le t \le T}$ the flow of marginal laws (on ${\mathbb  S}$) of the scheme $\{(X_t^{(h)})_{0 \le t \le T}\}$ given by \eqref{scheme0}. 
Denoting by $(t_n)_{n=0,\cdots,N}$ the subdivision supporting the scheme, the key step of the proof is to first expand $$\psi(\mu_{t_{n+1}}^{(h)}) = \psi(\mu_{t_{n+1}}^{(h)}) -\psi(\mu_{t_{n+1}^{-}}^{(h)}) + \psi(\mu_{t_{n+1}^{-}}^{(h)})$$ and then apply It\^o's formula (see \cite[Theorem 1.1]{DelHam25_2}) to the term $\psi(\mu_{t_{n+1}^{-}}^{(h)})$. For $S \in [t_n,t_{n+1}]$, we obtain
\begin{align*}
\psi(\mu_{{S}}^{(h)}) - \psi(\mu_{t_{n}}^{(h)}) &= \psi(\mu_{{S}}^{(h)}) -\psi(\mu_{{S^{-}}}^{(h)})  - \int_{t_n}^{{S}}\int_{\mathbb{S}}\partial_{\mu}\psi(\mu_s^{(h)},X_s^{(h)}(\mathbf{x}))[{DX_s^{(h)}(\mathbf{x})}]^2\mathrm{d}\mathbf{x}   \mathrm{d}s\\
    &\quad +\frac{1}{2}\int_{t_n}^{{S}}\int_{\mathbb{S}}\nabla_y\partial_{\mu}\psi(\mu_s^{(h)},X_s^{(h)}(\mathbf{x}))G_1(\mathbf{x}) \,  \mathrm{d}\mathbf{x}    \mathrm{d}s \\
        &\quad  + \frac{1}{2}\int_{t_n}^{{S}}\int_{\mathbb{S}}\int_{\mathbb{S}}\partial_{\mu}^2\psi(\mu_s^{(h)},X_s^{(h)}(\mathbf{x}),X_s^{(h)}(\mathbf{x}'))G_2(\mathbf{x},\mathbf{x}')\mathrm{d}\mathbf{x}  \mathrm{d}\mathbf{x}'    \mathrm{d}s\\
        &\quad + \int_{t_n}^{{S}}\int_{\mathbb{S}}\partial_{\mu}\psi(\mu_s^{(h)},X_s^{(h)}(\mathbf{x}))\mathrm{d}W_s(\mathbf{x}).
\end{align*}
Of course, when 
$S \in [t_n,t_{n+1})$, 
$\psi(\mu_S^{(h)})=\psi(\mu_{S^-}^{(h)})$.
Hence summing the above identity over 
the intervals $[0,t_1]$, 
$[t_1,t_2]$, $\cdots$, 
$[t_{N_S},S]$,  
where $N_S:=\lfloor S/h \rfloor$, 
we get
\begin{equation}\label{eq:Ito:5.6}
\begin{split}
\psi(\mu_{{S}}^{(h)}) - \psi(\mu_0) &= \sum_{n=0}^{{N_S}-1} \Bigl[\psi(\mu_{t_{n+1}}^{(h)}) -\psi(\mu_{t_{n+1}^{-}}^{(h)})\Bigr]  - \int_{0}^{{S}}\int_{\mathbb{S}}\partial_{\mu}\psi(\mu_s^{(h)},X_s^{(h)}(\mathbf{x}))[{DX_s^{(h)}(\mathbf{x})}]^2\mathrm{d}\mathbf{x}  \mathrm{d}s\\
    &\quad +\frac{1}{2}\int_{0}^{{S}}\int_{\mathbb{S}}\nabla_y\partial_{\mu}\psi(\mu_s^{(h)},X_s^{(h)}(\mathbf{x}))G_1(\mathbf{x})\mathrm{d}\mathbf{x}  \mathrm{d}s \\
        &\quad  + \frac{1}{2}\int_0^{{S}}\int_{\mathbb{S}}\int_{\mathbb{S}}\partial_{\mu}^2\psi(\mu_s^{(h)},X_s^{(h)}(\mathbf{x}),X_s^{(h)}(\mathbf{x}'))G_2(\mathbf{x},\mathbf{x}')\mathrm{d}\mathbf{x}  \mathrm{d}\mathbf{x}' \mathrm{d}s\\
        &\quad + \int_{0}^{{S}}\int_{\mathbb{S}}\partial_{\mu}\psi(\mu_s^{(h)},X_s^{(h)}(\mathbf{x}))\mathrm{d}W_s(\mathbf{x}).
\end{split}
\end{equation}
We now pass to the limit as $h\searrow0$ in the above identity. By the first statement of Lemma \ref{lemm.7.4}, for almost every $S \in [0,T]$, the left-hand side converges in law to
\[
\psi(\mu_{S})-\psi(\mu_0),
\]
where $\mu_{S}:=\mathscr{L}(X_{S})$. Furthermore, Lemma \ref{lemm.7.4} yields the convergence in law of the second term on the right-hand side to its limiting counterpart. We point out the fact that, the arguments used in the proof of Lemma \ref{lemm.7.4} can be adapted to establish the convergence in law of the third and fourth terms to their respective limiting counterparts.

Let us now focus on the second term on the right-hand side.
By invoking  It\^o's formula in \cite[Theorem 5.92]{CarDel1}, one has:
\begin{align*}
    \frac{\mathrm{d}}{\mathrm{d}s}\psi(\mu_{t_{n+1}^-}^{(h)}\star \Gamma_s)= \frac{1}{2} \int_{\mathbb{R}} \nabla_z\partial_{\mu}\psi(\mu_{t_{n+1}^{-}}^{(h)} \star \Gamma_s,z)\mathrm{d}(\mu_{t_{n+1}^{-}}^{(h)} \star \Gamma_s)(z),
\end{align*}
which in turn implies that the difference $\psi(\mu_{t_{n+1}}^{(h)}) -\psi(\mu_{t_{n+1}^{-}}^{(h)})$ can be written as follows: 
\begin{align*}
    \psi(\mu_{t_{n+1}}^{(h)}) -\psi(\mu_{t_{n+1}^{-}}^{(h)}) &= \frac{1}{2} \int_0^h\int_{\mathbb{R}} \nabla_z\partial_{\mu}\psi(\mu_{t_{n+1}^{-}}^{(h)} \star \Gamma_s,z)\mathrm{d}(\mu_{t_{n+1}^{-}}^{(h)} \star \Gamma_s)(z)\mathrm{d}s.
\end{align*}
Notice further that, 
for all $s \in [0,h]$: 
\begin{align*}
    \mathscr{W}_2^2(\mu_{t_{n+1}^{-}}^{(h)} \star \Gamma_s, \mu_{t_{n+1}^{-}}^{(h)}\star \Gamma_h) = \mathscr{W}_2^2(\mu_{t_{n+1}^{-}}^{(h)} \star \Gamma_s, \mu_{t_{n+1}^{-}}^{(h)}\star \Gamma_s \star \Gamma_{h-s}) 
    \leq C h,
\end{align*}
which implies that $\mathscr{W}_2^2(\mu_{t_{n+1}^{-}}^{(h)} \star \Gamma_s, \mu_{t_{n+1}^{-}}^{(h)}\star \Gamma_h) = \mathcal{O}_h(1)$ and, therefore,
\begin{equation}\label{eq:7:27}
\begin{split}
   \sum_{n=0}^{{N_S}-1}
   \Bigl[\psi(\mu_{t_{n+1}}^{(h)}) -\psi(\mu_{t_{n+1}^{-}}^{(h)})\Bigr] &=\frac{1}{2}\sum_{n=0}^{{N_S}-1} h\int_{\mathbb{R}} \nabla_z\partial_{\mu}\psi(\mu_{t_{n+1}^{-}}^{(h)}\star \Gamma_h,z)\mathrm{d}(\mu_{t_{n+1}^{-}}^{(h)}\star \Gamma_h)(z)+ \mathcal{O}_h(1)\\
   &=\frac{1}{2}\sum_{n=0}^{{N_S}-1} h\int_{\mathbb{S}} \nabla_z\partial_{\mu}\psi(\mu_{t_{n+1}}^{(h)},X^{(h)}_{t_{n+1}}(\mathbf{x}))\mathrm{d}\mathbf{x}+ \mathcal{O}_h(1),
   \end{split}
\end{equation}
where we used a change of variables to move from the first to the second line above together with the fact that $\mu_{t_{n+1}}^{(h)}:= \mu_{t_{n+1}^{-}}^{(h)}\star \Gamma_h$. The Landau symbol above is independent of the parameters $N$, $S$ and tends to $0$ with $h$. By invoking Lemma \ref{lemm 5.5}, we finally obtain the convergence in law as $h$ tends to 0, of the first term on the right hand side of the identity \eqref{eq:Ito:5.6} to the term $$\frac{1}{2}\int_0^{{S}}\int_{\mathbb{S}} \nabla_z\partial_{\mu}\psi(\mu_{s},X_s(\mathbf{x}))\mathrm{d}\mathbf{x}.$$
Now, concerning the term involving the stochastic integral in \eqref{eq:Ito:5.6}, it can be rewritten as follows
\[\int_{0}^{{S}}\int_{\mathbb{S}}\tilde{\psi}(X_s^{(h)},X_s^{(h)}(\mathbf{x}))\mathrm{d}W_s(\mathbf{x}) \]
where, $\tilde{\psi}$ is the Lions' lift of the functional $\partial_{\mu}\psi: \mathcal{P}_2(\mathbb{R})\times\mathbb{R} \rightarrow \mathbb{R}.$ In addition, from the Lipschitz assumption on $\tilde{\psi}$ and for any  $\epsilon >0$, one obtains similarly to \eqref{eq:26-06:2} (now stated in continuous time $s$, which is possible thanks to the second inequality in \eqref{eq 5.4}):
\begin{align}\label{bound:sto}
    \mathbb{E}\left| \int_{0}^{{S}}\int_{\mathbb{S}}\biggl(  \tilde{\psi}\big(X_s^{(h)},X_s^{(h)}(\mathbf{x}))\big)-  \tilde{\psi}\big({\rm e}^{\epsilon \Delta}X_s^{(h)},{\rm e}^{\epsilon \Delta}X_s^{(h)}(\mathbf{x})\big)
        \biggr)
    \mathrm{d}W_s({\mathbf x})
      \right|^2
    &\leq C_{T,\lambda}(1+\mathbb{E}[\|X_0\|_2^2])\epsilon.
\end{align}
Furthermore, by using the joint convergence of $({\rm e}^{\epsilon \Delta}X_t^{(h)},W_t)_{0 \le t \le T}$, for any $\epsilon>0$ fixed, we can adapt the results from \cite{Kurtz1996_0,Kurtz1996} to pass to the limit. That is, in law, 
\begin{equation*}
    \int_{0}^{{S}}\int_{\mathbb{S}}\tilde{\psi}({\rm e}^{\epsilon \Delta}X_s^{(h)},{\rm e}^{\epsilon\Delta}X_s^{(h)}(\mathbf{x}))\mathrm{d}W_s(\mathbf{x})
    \underset{h \searrow 0}{\longrightarrow}\int_{0}^{{S}}\int_{\mathbb{S}}\tilde{\psi}({\rm e}^{\epsilon \Delta}X_s,{\rm e}^{\epsilon\Delta}X_s(\mathbf{x}))\mathrm{d}W_s(\mathbf{x}).
\end{equation*}
{In turn, 
as $\epsilon \searrow 0$, 
the right-hand side can be shown to  tend to the desired limiting term, thanks to 
an appropriate version of 
\eqref{bound:sto} with $X^{(h)}$ being replaced by $X$.}

To finish, 
we notice that all the aforementioned results of weak convergence, which we stated separately, hold in fact jointly. This makes it possible to pass to the (weak) limit in 
\eqref{eq:Ito:5.6}, at least for almost every {$S \in [0,T]$}.
In this way, we get that, for almost every {$S \in [0,T]$} and ${\mathbb P}$-almost surely, \eqref{eq:5.2}
holds true (at time $t=S$). By continuity (in time) 
of the functionals appearing in 
\eqref{eq:5.2}, we get the result, 
${\mathbb P}$-almost surely, for every time $t\in [0,T]$, {and thus for every $t \geq 0$ (by choosing $T$ as large as needed)}.
This completes the proof.
\end{proof}
We now derive the equation satisfied by the density function {$z \mapsto p_t(z)$} obtained above
\begin{thm}\label{thm:eq:density}
Let $(X_t)_{t\geq 0}$ be the unique solution to the distribution-drifted RSHE~\eqref{eq:4.28} given by Theorem \ref{Main 1}. Then, the flow of density functions $(p_{t} : x \mapsto p_t(x))_{t > 0}$
of the flow of measures $(\mu_t = \textrm{\rm Leb}_{\mathbb S} \circ X_t^{-1})_{t \geq 0}$, 
starting from SPDE
the (possibly non absolutely-continuous) measure 
$\mu_0$,
satisfies (in distribution) the following corrected Dean-Kawasaki SPDE: 
\begin{equation}
\begin{split}
    {\rm d}_t p_t(x)
&= \frac{1}{2}\Delta_x p_t(x)
 - 4 \Delta_x\!\left( \mathbbm{1}_{\{p_t(x)>0\}}\frac{1}{p_t(x)}\right)
 + \frac{1}{2}\sum_{k \in \mathbb N}
 \Delta_x\!\left(\lambda_k^2\, e_k^2\!\left(\frac{\tilde F_t(x)}{2}\right) p_t(x)\right)  \\
&\quad
 - \frac{1}{2}\sum_{k \in \mathbb N}
 D_x\!\left(\lambda_k\, e_k\!\left(\frac{\tilde F_t(x)}{2}\right) p_t(x)\right)
 \,\mathrm{d}B_t^k 
    \end{split}
\end{equation}
where, 
for each $t \geq 0$, $\tilde{F_t}:= \tilde{F}_{\mu_t}$ stands for the cumulative distribution function of $\mu_t$.
\end{thm}
\begin{proof} Let $\psi$ be a smooth mean-field function satisfying the conditions of Proposition \ref{prop5.2}, whose second derivative with respect to the measure vanishes. Then, from It\^o's formula \eqref{eq:5.2} (now representing 
each $\mu_t$ by the law of $[0,1] \ni y \mapsto X_t(y/2)$), one has 
\begin{equation*}
 \begin{split}
        \psi(\mu_t) &= \psi(\mu_0) + \frac{1}{2} \int_0^t\int_0^1\partial_y\partial_{\mu}\psi(\mu_s,X_s(y/2))\mathrm{d}y\mathrm{d}s
        -\int_0^t\int_0^1\partial_y\partial_{\mu}\psi(\mu_s,X_s(y/2))[DX_s(y/2)]^2\mathrm{d}y\mathrm{d}s 
        \\
        &  
        +\frac{1}{2}\int_0^t\int_0^1\partial_y\partial_{\mu}\psi(\mu_s,X_s(y/2))G_1(y/2)\mathrm{d}y\mathrm{d}s +
        \sum_{k \in {\mathbb N}} \lambda_k \int_0^t \biggl(\int_0^1\partial_{\mu}\psi(\mu_s,X_s(y/2))e_k(y/2) \ud y \biggr) \ud B_s^k.
    \end{split}
    \end{equation*}
{We now recall 
from Proposition 
\ref{Prop5.1}
that, almost surely, and for a.e. $s$, $X_s(\cdot /2)$ is a  one-to-one mapping from $(0,1)$ onto $(X_s(0),X_s(1/2))$, with $\tilde F_{\mu_s}$ as inverse. In particular, 
we can write 
$y=\tilde F_{\mu_s}(X_s(y/2))$
for $y \in (0,1)$. Moreover, 
from 
\eqref{eq:08-07:11},
$D_{\mathbf x} X_s ( \tilde F_{\mu_s}(z)/2) = 1/(2 p_s(z))$, 
for 
$\textrm{\rm Leb}_{\mathbb R}$-a.e.
$z \in (X_s(0),X_s(1/2))$, and thus 
also
for $\mu_t$-a.e. $z \in (X_s(0),X_s(1/2))$. In particular, 
setting $z =X_s(y/2) \Leftrightarrow 
y = \tilde{F}_{\mu_s}(z)
$, we deduce that, 
for $\textrm{\rm Leb}_{[0,1]}$-a.e. $y \in (0,1)$, 
$$D_{\mathbf x} X_s (y/2) = \frac1{2 p_s(X_s(y/2))}.$$ This makes it possible to express all the functions of $y$ appearing above as functions of $X_s(y/2)$. And then using the fact that the law of $X_s(\cdot/2)$ on $((0,1),\textrm{\rm Leb}_{\mathbb S})$
is equal to $\mu_s$ (which has $p_s$ as density, with $p_s$ being supported by $[X_s(0),X_s(1/2)]$), we get}
 \begin{equation*}
    \begin{split}
       \psi(\mu_t) &= \psi(\mu_0) + \frac{1}{2} \int_0^t\int_{\mathbb{R}}\partial_y\partial_{\mu}\psi(\mu_s,z) {p_s(z)}\mathrm{d}z\mathrm{d}s 
       \\
       &\quad - {4}\int_0^t\int_{\mathbb{R}}\partial_y\partial_{\mu}\psi(\mu_s,z)\mathbbm{1}_{[a_s,b_s]}(z)\frac{1}{ p_s(z)}\mathrm{d}z\mathrm{d}s 
       \\
        &\quad +\frac{1}{2}\int_0^t\int_{\mathbb{R}}\partial_y\partial_{\mu}\psi(\mu_s,z)G_1(\tilde F_{\mu_s}(z)/2) {p_s(z)}\mathrm{d}z\mathrm{d}s
        \\
        &\quad + \sum_{k \in{\mathbb N}} \lambda_k\int_0^t \biggl( \int_{\mathbb{R}}\partial_{\mu}\psi(\mu_s,z)
        e_k(\tilde F_{\mu_s}(z)/2)
        {p_s(z)}
    \ud z \biggr)
        \mathrm{d}B_s^k, 
        \end{split}
    \end{equation*}
    where $[a_s,b_s] = [X_s(0), X_s(1/2)]$.
    Choose now $\psi(\mu_t):= \int_{\mathbb{R}}\Upsilon(z)\mathrm{d}\mu_t(z)$, for a smooth test function $\Upsilon$, which implies that $\partial_{\mu}\psi(\mu_t,z)= \Upsilon'(z)$ and $\partial_y\partial_{\mu}\psi(\mu_t,z)= \Upsilon''(z)$ (see \cite[Chapter 5, example 1]{CarDel1}). 
    And then,
    \begin{equation*}
    \begin{split}
\int_{\mathbb{R}}\Upsilon(z)p_t(z)\mathrm{d}z &=  \int_{\mathbb{R}}\Upsilon(z)p_0(z)\mathrm{d}z+ \frac{1}{2}\int_0^t\int_{\mathbb{R}}\Delta_z\Upsilon(z) p_s(z)\mathrm{d}z\mathrm{d}s 
        \\
        &\quad - {4} \int_0^t\int_{\mathbb{R}}\mathbbm{1}_{\{p_t(z)>0\}}\Delta_z\Upsilon(z)\frac{1}{p_s(z)}\mathrm{d}z\mathrm{d}s\\
        &\quad  
        +\frac{1}{2}\int_0^t\int_{\mathbb{R}}\Delta_z\Upsilon(z)
        \biggl(
        \sum_{k\in \mathbb{N}}\lambda_k^2 
        e_k^2( \tilde F_{\mu_s}(z)/2)
        \biggr)
        p_s(z) \mathrm{d}z\mathrm{d}s \\
        &\quad + \int_0^t\int_{\mathbb{R}}D_z\Upsilon(z)
        {\biggl(
        \sum_{k\in \mathbb{N}}\lambda_k e_k(\tilde F_{\mu_s}(z)/2)\biggr)}p_s(z) \mathrm{d}z \mathrm{d}B_s^k,
        \end{split}
    \end{equation*}
where we used \eqref{eq:5.3}
to get the last two lines.
This leads to the desired result.
\end{proof}

		\bibliographystyle{plain}
		\bibliography{Biblio,Biblio_2}

\end{document}